\documentclass[intlimits,12pt,reqno]{amsart}

\usepackage[T1]{fontenc}
\usepackage[utf8]{inputenc}

\usepackage{latexsym,amsthm,amsmath,amssymb,mathrsfs,mathtools,upgreek,stmaryrd}
\usepackage{tikz}
\usetikzlibrary{arrows,calc,patterns}
\usepackage{wrapfig}

\usepackage[mathscr]{eucal}

\usepackage{hyperref}
\usepackage{cleveref}

plus 6pt minus 12pt
\newcommand{\cJ}{\mathcal J}
\newcommand{\cM}{\mathcal M}

\newcommand{\Sph}{\mathbb S}

\def\sgn{\operatorname{sgn}}

\newcommand{\bbbr}{\mathbb R}
\newcommand{\bbbc}{\mathbb C}
\newcommand{\bbbb}{\mathbb B}
\newcommand{\bbbs}{\mathbb S}
\newcommand{\bbbn}{\mathbb N}
\newcommand{\bbbh}{\mathbb H}

\newcommand{\R}{\mathbb R}
\newcommand{\Z}{\mathbb Z}
\newcommand{\overbar}[1]{\mkern 1.7mu\overline{\mkern-1.7mu#1\mkern-1.5mu}\mkern 1.5mu}

\newcommand{\eps}{\varepsilon}

\numberwithin{equation}{section}

\def\dist{\operatorname{dist}}
\def\deg{\operatorname{deg}}
\def\lip{\operatorname{Lip}}

\def\ext{\operatorname{ext}}

\newtheorem{theorem}{Theorem}[section]
\newtheorem*{theorem*}{Theorem}
\newtheorem{lemma}[theorem]{Lemma}
\newtheorem{corollary}[theorem]{Corollary}
\newtheorem{proposition}[theorem]{Proposition}

\theoremstyle{definition}
\newtheorem{remark}[theorem]{Remark}
\newtheorem{definition}[theorem]{Definition}

\newcommand{\Ep}{\bigwedge\nolimits}
\def\mvint_#1{\mathchoice
          {\mathop{\vrule width 6pt height 3 pt depth -2.5pt
                  \kern -9pt \intop}\limits_{\kern -3pt #1}}%
          {\mathop{\vrule width 5pt height 3 pt depth -2.6pt
                  \kern -6pt \intop}\nolimits_{#1}}%
          {\mathop{\vrule width 5pt height 3 pt depth -2.6pt
                  \kern -6pt \intop}\nolimits_{#1}}%
          {\mathop{\vrule width 5pt height 3 pt depth -2.6pt
                  \kern -6pt \intop}\nolimits_{#1}}}

\usepackage[colorinlistoftodos,prependcaption,textsize=tiny]{todonotes}

\newcommand{\HI}{\mathcal{H}}
\def\rank{\operatorname{rank}}
\newcommand{\brac}[1]{\left ( #1 \right )}

\newcommand{\lap}{\Delta}

\usepackage{setspace}

\title[H\"older continuous mappings]{H\"older continuous mappings, differential forms and the Heisenberg groups}
\author[Haj\l{}asz]{Piotr Haj\l{}asz}
\thanks{P.H. was supported by NSF grants DMS-2055171 and DMS-2452426.}
\author[Mirra]{Jacob Mirra}
\author[Schikorra]{Armin Schikorra}
\thanks{A.S. was supported by NSF Career DMS-2044898. A.S. is a Humboldt fellow}

\address{Department of Mathematics, University of Pittsburgh, \newline \indent 301 Thackeray Hall, Pittsburgh,
Pennsylvania 15260}
\email{hajlasz@pitt.edu, armin@pitt.edu}


\begin{document}

\sloppy

\begin{abstract}
We develop analysis of H\"older continuous mappings with applications to geometry and topology of the Heisenberg groups. We cover the theory of distributional Jacobians of H\"older continuous mappings and pullbacks of differential forms under H\"older continuous mappings. That includes versions of the change of variables formula and the Stokes theorem for H\"older continuous mappings. The main applications are in the setting of the Heisenberg groups, where we
provide a simple proof of a generalization of the Gromov non-embedding theorem, and new results about the H\"older homotopy groups of the Heisenberg groups.
\end{abstract}

\maketitle
\tableofcontents

\section{Introduction}
\label{S1}
The goal of this paper is to develop differential analysis of H\"older continuous functions and mappings with applications to geometry and topology of the Heisenberg groups. 
While the aim of Introduction is to clarify the paper's objective, it does not provide precise statements or include all proven results. Instead, it provides a general overview of the paper's content.
Since the paper is long, we thought it would be important to explain the main ideas without using complicated and technical terms.

The paper begins with \Cref{pre}. Since this section, gathers preliminary material that will be used in later sections, we describe its content at the end of Introduction. The main theme of the paper starts in \Cref{DJ}.

H\"older continuous functions need not be differentiable at any point. It is also possible to construct, for any $0<\gamma<1$, a strictly increasing, $\gamma$-H\"older continuous function $f:\R\to\R$ such that $f'=0$ almost everywhere,  \cite{salem}. In that case, $f'$ does not satisfy the fundamental theorem of calculus, and hence $f'$ is not the distributional derivative. That also means that even if a H\"older continuous function is differentiable a.e., properties of the derivative do not say much about properties of the function.

On the other hand Young, \cite{Young}, proved that if $f\in C^{0,\alpha}([0,1])$ and $g\in C^{0,\beta}([0,1])$, $\alpha+\beta>1$, then the Stielties integral exists and satisfies:
\begin{equation}
\label{eq174}
\Bigg|\int_0^1 f(x)\, dg(x)\Bigg|\leq C\Vert f\Vert_{C^{0,\alpha}}\Vert g\Vert_{C^{0,\beta}}.
\end{equation}
Here is another approach to this inequality. One can prove that
$$
\Bigg|\int_0^1 f(x)\, g'(x)\, dx\Bigg|\leq C\Vert f\Vert_{C^{0,\alpha}}\Vert g\Vert_{C^{0,\beta}} 
\quad
\text{if } f\in C^{0,\alpha}([0,1]),\ g\in C^\infty([0,1]),
$$
see \Cref{T75}. Now, if $f\in C^{0,\alpha}$ $g\in C^{0,\beta}$, then we can define
\begin{equation}
\label{eq176}
\int_0^1 f(x) g'(x)\, dx:=\lim_{\eps\to 0} \int_0^1 f(x) g_\eps'(x)\, dx,
\end{equation}
where $g_\eps$ is a smooth approximation of $g$ by convolution.  
We need to understand that the expression on the left hand side of \eqref{eq176} is a formal notation since $g'$ is not a function. However, we can interpret $g'$ as a distribution acting on test functions $f\in C^{0,\alpha}$ and the way $g'$ acts on $f$ is defined through the above limit. One can also prove that the distributional integral \eqref{eq176} coincides with the Stielties integral \eqref{eq174}.

The result of Young found numerous applications in stochastic analysis \cite{FH}, but here we are interested in a different direction in which it evolved. 

More recently,
several groups of researchers developed, independently,
a multidimensional version of the Young integral
\begin{equation}
\label{eq177}
\int_\cM f dg_1\wedge\ldots\wedge dg_n, 
\quad 
f\in C^{0,\alpha}(\cM),\ g_i\in C^{0,\beta}(\cM),\ \alpha+n\beta>n,
\end{equation}
where $\cM$ is an $n$-dimensional oriented Riemannian manifold (cf.\ \Cref{T35}). 
Using hard harmonic analysis, Sickel and Youssfi \cite{SY1999} extended the integral \eqref{eq177} to potential spaces that are much more general than H\"older spaces. This result went largely unnoticed in parts of the Geometric Analysis community until Brezis and Nguyen \cite{brezisn1,brezisn2} later rediscovered it in the context of H\"older functions. Additionally, Conti, De Lellis, and Sz\'ekelyhidi \cite{Conti} studied a similar integral in a special case related to the rigidity of isometric immersions and convex integration, while Z\"ust \cite{zust2,zust1} developed the theory in the setting of currents in metric spaces, with applications to the geometry of the Heisenerg groups in mind.

We should also mention that the Onsager conjecture \cite{Onsager} led to another, but related direction of analysis of H\"older continuous functions. Onsager conjectured that weak solutions to the incompressible Euler equations satisfy
the law of conservation of energy if their spatial regularity is above $\frac{1}{3}$-H\"older and 
may violate it if the regularity is below $\frac{1}{3}$-H\"older. While the affirmative (easier) part of the conjecture was proved by Constantin, E, and Titi \cite{CET94} and Eyink \cite{Eyink}, the difficult part, counterexample when $\gamma<\frac{1}{3}$, was constructed by Isett \cite{Isett} using convex integration. Although we will not mention this direction of research in the paper, we want to emphasize that we used ideas from \cite{CET94} in the proof of \Cref{T99}.

Let us briefly mention the approaches to define the generalized integrals \eqref{eq177} given by Z\"ust \cite{zust1,zust2} and Brezis and Nguyen \cite{brezisn1,brezisn2}.

Z\"ust used Young's result \eqref{eq174} and a clever induction argument (induction over the dimension) to define \eqref{eq177}. 
Let $I^n=[0,1]^n$. Young's theorem \eqref{eq174} and a formal integration by parts allows us to define
$$
\int_{I^2} dg_1\wedge dg_2:=\int_{\partial I^2} g_1\, dg_2
\qquad
\text{for } g_1,g_2\in C^{0,\beta}(I^2),\ \textstyle{\beta>\frac{1}{2}}\, .
$$
Then, we define $\int_{I^2}f\, dg_1\wedge dg_2$ using a Riemann sum approximation: Divide $I^2$ into $k^2$ identical squares $\{I_{ki}\}_{i=1}^{k^2}$ with centers $p_{ki}$, and we define
\begin{equation}
\label{eq178}
\int_{I^2} f\, dg_1\wedge dg_2:=
\lim_{k\to\infty}\sum_{i=1}^{k^2} f(p_{ik})\int_{I_{ki}^2} dg_1\wedge dg_2=
\lim_{k\to\infty}\sum_{i=1}^{k^2} f(p_{ik})\int_{\partial I_{ki}^2} g_1\, dg_2.
\end{equation}
It turns out that the Riemann sum \eqref{eq178} converges provided
$$
f\in C^{0,\alpha}(I^2),
\quad
g_1,g_2\in C^{0,\beta}(I^2),
\quad
\alpha+2\beta>2.
$$
Integral \eqref{eq178} allows us to define
$$
\int_{I^3}dg_1\wedge dg_2\wedge dg_3:=\int_{\partial I_3}g_1\, dg_2\wedge dg_3,
\text{ provided } g_1,g_2,g_3\in C^{0,\beta}(I_3),\ \textstyle{\beta>\frac{2}{3}}\, .
$$
It is clear now how to use induction to define \eqref{eq177} when $\cM=I^n$, and the case of a general manifold is obtained through partition of unity and local coordinate systems.

The approach of Brezis and Nguyen is based on a very different idea. The key observation is that there is abounded extension operator
$$
\ext_\cM:C^{0,\alpha}(\cM)\to C^{0,\alpha}\cap W^{1,p}(\cM\times [0,1]),
\text{ provided }  p\in (1,\infty),\ \textstyle{1-\frac{1}{p}<\alpha\leq 1}.
$$
Moreover, we can guarantee that $\ext_\cM u=0$ near $\cM\times\{ 1\}$.
While this result follows from a general theory of traces in the Sobolev $W^{1,p}$ spaces, see \Cref{T6} for a straightforward argument.

Assume now that $f$ and $g_i$ are as in \eqref{eq177} and let $f^\eps$, $g_i^\eps$ be the approximations by convolution. It is easy to see that we can find $p\in (1,\infty)$ and $q\in (1,\infty)$, such that
$$
\textstyle{\alpha>1-\frac{1}{p},
\quad
\beta>1-\frac{1}{q},
\quad
\frac{1}{p}+\frac{n}{q}=1}.
$$
Let $F^\eps=\ext_\cM (f^\eps)$ and $G_i^\eps=\ext_\cM (g_i^\eps)$. Then integration by parts yields
$$
\int_\cM f^\eps \, dg_1^\eps\wedge\ldots\wedge dg_n^\eps=
\int_{\partial(\cM\times [0,1])} F^\eps\, dG_1^\eps\wedge\ldots\wedge dG_n^\eps=
\int_{\cM\times [0,1]} dF^\eps \wedge dG_1^\eps\wedge\ldots\wedge dG_n^\eps.
$$
Since $dF^\eps\in L^p$, $dG_i^\eps\in L^q$, and $\frac{1}{p}+\frac{n}{q}=1$, it follows from the H\"older inequality 
that the integral on the right hand side converges as $\eps\to 0$, and we can define
$$
\int_\cM f dg_1\wedge\ldots\wedge dg_n:=\lim_{\eps\to 0}\int_\cM f^\eps dg_1^\eps\wedge\ldots\wedge dg_n^\eps
=
\int_{\cM\times [0,1]} dF \wedge dG_1\wedge\ldots \wedge dG_n.
$$
For details, see Theorems~\ref{T32} and~\ref{T34}. Note that this can be regarded as a higher dimensional generalization of \eqref{eq176}.

The technique of Brezis and Nguyen is very powerful and it has recently been used and extended for example in \cite{GladbachO,LS19,SVS}. Similar ``compensation by extension'' ideas were also present in earlier works, e.g. \cite{Chanillo01011991,CJS89}.

\Cref{DJ} of our paper is devoted to a detailed study of integrals related to \eqref{eq177}, both in the category of H\"older continuous functions and in a wider category of fractional Sobolev spaces.
Our approach is based on that of Brezis and Nguyen.

Building on the idea of distributional Jacobians \eqref{eq177}, in \Cref{S6}, we study pullbacks of differential forms by H\"older continuous mappings. Since the pullback of a differential form has a Jacobian-like structure, one can believe that 
the generalized integral
\begin{equation}
\label{eq175}
\int_{\cM} f^*\kappa\wedge\tau
:=
\lim_{\eps\to 0}
\int_{\cM} (f^\eps)^*\kappa\wedge\tau
\end{equation}
is well defined if $f\in C^{0,\gamma}(\cM;\R^m)$, $\kappa$ in a $k$-form on $\R^m$ with $C^{0,\alpha}$ coefficients and $\tau$ is an $(n-k)$-form on $\cM$ with $C^{0,\alpha\gamma}$ coefficients, under suitable assumptions about $\alpha$ and $\gamma$, see e.g., \Cref{T79}. 

Pullbacks of differential forms under fractional Sobolev maps have also been studied in a recent paper by Detaille, Mironescu and Xiao \cite{DMX}.

In \Cref{GSC} we apply the theory of Jacobians and pullbacks from Sections~\ref{DJ} and~\ref{S6} to study change of variables and integration by parts.

If $\Omega\subset\R^n$ is a bounded domain and $f:\partial\Omega\to\R^n$ is continuous, then for each $y\in\R^n\setminus f(\partial\Omega)$ one can define the winding number $w(f,y)$, see \Cref{D5}. The winding number is defined in terms of degree and roughly speaking, it says how many times (counting orientation) the image of $\partial\Omega$ under the mapping $f$ wraps around $y$. 

The winding number appears in several formulas related to the change of variables. For example if 
$\Omega$ is a bounded domain with smooth boundary, 
$f\in C^\infty(\overbar{\Omega};\R^n)$, and $u\in C^\infty(\R^n)$, then (see \Cref{T43} and \Cref{T44})
\begin{equation}
\label{eq179}
\int_{\Omega}(u\circ f)\, df_1\wedge\dots\wedge df_n\, dx = \int_{\R^n}u(y)w(f,y)\, dy
\end{equation}
and
\begin{equation}
\label{eq180}
(-1)^{i-1}\int_{\partial\Omega} (u\circ f)\,df_1\wedge\ldots\wedge\widehat{df_i}\wedge\ldots\wedge df_n=
\int_{\R^n}\frac{\partial u}{\partial y_i}(y)w(f,y)\, dy.
\end{equation}
Note that \eqref{eq179} is the classical change of variables and \eqref{eq180} is a version of Green's theorem.

Using the theory of distributional Jacobians mentioned above we extend these and other similar formulas to the case when $f\in C^{0,\gamma}(\Omega;\R^n)$ $\gamma\in\big(\frac{n-1}{n},1\big]$, see \Cref{GSC} for precise statements. Note that under the given assumption the measure of $f(\partial\Omega)$ is zero and hence $w(f,y)$ is defined for almost all $y\in\R^n$.
Related results have been obtained in \cite{Conti,LZ,LP17,LPS21,ZustGAFA}.

For example, as a special case we obtain the following result (Corollaries~\ref{T91} and~\ref{T94}):

If $\gamma=(\gamma^x,\gamma^y)\in C^{0,\alpha}(\Sph^1;\R^2)$, $\alpha\in\big(\frac{1}{2},1\big]$ and if
$f=(f^x,f^y)\in C^{0,\alpha}(\bbbb^2;\R^2)$ is an extension of $\gamma$, then
\begin{equation}
\label{eq182}
\frac{1}{2}\int_\gamma xdy-ydx=
\frac{1}{2}\int_{\Sph^1}\gamma^xd\gamma^y-\gamma^yd\gamma^x=\int_{\bbbb^2}df^x\wedge df^y=
\int_{\bbbr^2} w(f,y)\, dy.
\end{equation}
Note that the integral on the right hand side equals the oriented area enclosed by the curve $\gamma$.

The remaining part of the paper is devoted to analysis of H\"older continuous mappings from Euclidean domains and Riemannian manifolds into the Heisenberg groups.

It is fair to say that the study of  {\em local} geometry of the Heisenberg groups has been initiated and popularized by Gromov in his seminal paper \cite{gromov2}. 
The Heisenberg group provides a {\em local} model for contact manifolds since according to the Darboux theorem, every contact manifold is locally equivalent to the Heisenberg group. While the questions in contact topology are concerned with the {\em global} structure of contact manifolds, questions in the geometry of the Heisenberg group are concerned with the {\em local} structure of contact manifolds.
The local geometry of the Heisenberg group is very complicated. In some directions it is Euclidean while in other directions is is purely fractal. For this reason Gromov \cite[0.5.C,0.5.D,3.1]{gromov2}, initiated an investigation of H\"older continuous maps into the Heisenberg groups and more general sub-Riemannian manifolds: H\"older maps can catch the fractal nature of the space. The remaining sections of the paper are devoted to this program.

\Cref{S7} is a brief overview of the theory of the Heisenberg groups. All results in this section are known and its purpose it to fix notation and list results that are needed later. Let us just mention here that the Heisenberg group is $\bbbh^n=\R^{2n+1}$ equipped with a certain Lie groups structure. 
We denote the coordinates in $\R^{2n+1}$ by $(x_1,y_1,\ldots,x_n,y_n,t)$ and then
$$
X_j=\frac{\partial}{\partial x_j}+ 2y_j\, \frac{\partial}{\partial t},
\quad
Y_j=\frac{\partial}{\partial y_j}- 2x_j\, \frac{\partial}{\partial t},
\quad
T=\frac{\partial}{\partial t},
$$
is a basis off the left invariant vector fields. However, the Heisenberg group is equipped with the $2n$-dimensional {\em horizontal} distribution $H\bbbh^n$,
$$
H_p \bbbh^n = \text{span} \left\{ X_1(p),Y_1(p),\dots,X_n(p),Y_n(p) \right\} \quad \text{for all } p \in \bbbh^n.
$$
The horizontal distribution is nothing else, but a standard {\em contact} structure on $\R^{2n+1}$, because
\begin{equation}
\label{eq183}
H_p \bbbh^n = \operatorname{ker}\upalpha(p) \subset T_p\R^{2n+1}
\quad
\text{where}
\quad
\upalpha = dt + 2 \sum_{j=1}^n (x_j \, dy_j - y_j \, dx_j)
\end{equation}
is a standard contact form.

A Lipschitz curve $\gamma=(\gamma^{x_1},\gamma^{y_1},\ldots,\gamma^{x_n},\gamma^{y_n},\gamma^t):[0,1]\to\bbbr^{2n+1}$ is {\em horizontal} if it is tangent to the horizontal distribution. It is not difficult to check that $\gamma$ is horizontal if and only if for every $s\in [0,1]$,
\begin{equation}
\label{eq181}
\gamma^t(s)-\gamma^t(0)
=
-4\sum_{j=1}^n \Bigg(\frac{1}{2}\int_{\gamma_j|_{[0,s]}} x_j\, dy_j-y_j\, dx_j\Bigg),
\quad
\text{where } \gamma_j=(\gamma^{x_j},\gamma^{y_j}).
\end{equation}
That is, the components $\gamma_j$ can be arbitrary Lipschitz curves, but then the component $\gamma^t$ is uniquely determined from \eqref{eq181}. In particular, if $\gamma$ is a closed curve, then $\gamma^t(1)-\gamma^t(0)=0$ and it follows from the standard Lipschitz case of \eqref{eq182} that the sum of oriented areas bounded by the curves $\gamma_j$ equals zero.

The Heisenberg group is equipped with the Carnot-Carath\'eodory metric $d_{cc}(p,q)$ which is the infimum o lengths of horizontal curves connecting $p,q\in\bbbh^n$. The metric $d_{cc}$ is bi-Lipschitz equivalent to the Kor\'anyi metric $d_K$ defined in \eqref{eq28}, and in what follows the Heisenberg group is regarded as a metric space $(\bbbh^n,d_K)$.
We prefer to work with the Kor\'anyi metric, because unlike the Carnot-Carath\'eodory metric, $d_K(p,q)$ can be computed explicitly.

The metric $d_K$ is not smooth and it has a fractal nature. In fact, for any compact set 
$K\subset\bbbh^n$ there is a constant $C\geq 1$ such that $C^{-1}d_{K}(p,q)\leq |p-q|\leq Cd_{K}(p,q)^{1/2}$
for all $p,q\in K$. Hence the topology induced by the $d_K$
metric is the same as the Euclidean topology, however, the Hausdorff dimension of the 
Heisenberg group equals $2n+2$, while the topological dimension is $2n+1$, so methods from smooth differential geometry cannot be directly applied to the Heisenberg group.

In \Cref{HM} we study H\"older continuous mappings from a Euclidean domain or from a Riemannian manifold into $\bbbh^n$. A Lipschitz mapping $f$ into $\R^{2n+1}$ is Lipschitz as a mapping into $\bbbh^n$ if and only if $f^*\upalpha=0$, where $\upalpha$ is the contact form from \eqref{eq183}. Moreover, Lipschitz mappings into $\bbbh^n$ always satisfy $\rank Df\leq n$, see \Cref{T110}. The main objective of \Cref{HM} is to generalize these results to the case of H\"older continuous mappings into $\bbbh^n$.

Le Donne and Z\"ust \cite{LZ} appear to be the first to use the concept of Jacobians of H\"older continuous mappings \eqref{eq177} in the setting of the Heisenberg groups. 

In this section, we develop this idea further, including the analysis of pullbacks of differential forms.
Another important idea that is heavily used in the section and also in Sections~\ref{S4},~\ref{S:Hodge} and~\ref{NHG} is described next.

If $f\in C^{0,\gamma}(\Omega;\R^{2n+1})$, $\Omega\subset\R^m$, and $\kappa\in\Omega^k(\R^{2n+1})$, then locally we have
\begin{equation}
\label{eq184}
\Vert f_\eps^*\kappa\Vert_\infty\leq C\eps^{k(\gamma-1)},
\end{equation}
see \Cref{T8}. Here $f_\eps$ is the approximation of $f$ by convolution. Estimate \eqref{eq184} easily follows from the definitions of pullback and convolution. Note that estimate \eqref{eq184} diverges to $\infty$, as $\eps\to 0$, if $0<\gamma<1$. 

In particular, \eqref{eq184} is true if $f\in C^{0,\gamma}(\Omega;\bbbh^n)$, because $\operatorname{id}:\bbbh^n\to\R^{2n+1}$ is locally Lipschitz continuous. However, in the case of the contact form $\upalpha$ we have a much better estimate (\Cref{T45}):

If $f\in C^{0,\gamma}(\Omega;\bbbh^n)$, $\Omega\subset\R^m$, then locally we have
\begin{equation}
\label{eq185}
\Vert f_\eps^*\upalpha\Vert_\infty\leq C\eps^{2\gamma-1}.
\end{equation}
This time, the estimate converges to $0$ as $\eps\to 0$, if $\gamma>\frac{1}{2}$.

\Cref{T49}, which follows from the Lefschetz \Cref{T22}, implies that any form $\kappa\in\Omega^k(\R^{2n+1})$, $n+1\leq k\leq 2n+1$, can be decomposed as
\begin{equation}
\label{eq186}
\kappa=\beta\wedge\upalpha+d(\gamma\wedge\upalpha),
\end{equation}
and hence 
$$
f^*_\eps\kappa=f_\eps^*\beta\wedge f_\eps^*\upalpha+d(f_\eps^*\gamma\wedge f^*_\eps\upalpha).
$$
This allows us to get a much better estimate than that from \eqref{eq184}, because while we estimate $f^*_\eps\beta$ and $f_\eps^*\gamma$ using \eqref{eq184}, we can estimate $f_\eps^*\upalpha$ using \eqref{eq185}.

Note that if $f\in C^{0,\gamma}_{\rm loc}(\Omega;\bbbh^n)$, then $f\in C^{0,\gamma}_{\rm loc}(\Omega;\bbbr^{2n+1})$.
Using the above ideas we prove the following characterization of H\"older continuous mappings into $\bbbh^n$ (see Theorems~\ref{T64} and~\ref{T99}).

If $f\in C^{0,\gamma}_{\rm loc}(\Omega;\R^{2n+1})$, $\gamma\in \big(\frac{1}{2},1\big]$, $\Omega\subset\R^m$, then the following conditions are equivalent:
\begin{enumerate}
\item[(a)]
$f\in C^{0,\gamma}_{\rm loc}(\Omega;\bbbh^n)$;
\item[(b)]
$f_\eps^*\upalpha\rightrightarrows 0$ uniformly on compact sets;
\item[(c)]
$f^*\upalpha=0$ in the distributional sense i.e., $\int_\Omega f^*\upalpha\wedge\phi=0$
for all  $\phi\in\Omega_c^{m-1}(\Omega)$.
\end{enumerate}
The integral in (c) is understood in the sense of \eqref{eq175}.
Quite surprisingly, in condition (b) we do not require any estimates for the rate of convergence $f_\eps^*\upalpha\rightrightarrows 0$, but afterwards, we have \eqref{eq185}.

The condition about the rank of the derivative $\rank Df\leq n$, has the following generalization, see \Cref{T47}.

Let $f\in C^{0,\gamma}_{\rm loc}(\Omega;\bbbh^n)$, $\gamma\in\big(\frac{n+1}{n+2},1\big]$, $\Omega\subset\R^m$, $m>n$. Then
\begin{equation}
\label{eq188}
\int_\Omega f^*\kappa\wedge\phi=0
\text{ for all } \kappa\in\Omega^{n+1}(\R^{2n+1}) \text{ and all } 
\phi\in\Omega_c^{m-n-1}(\Omega).
\end{equation}
Note that in the case of smooth mappings, \eqref{eq188} is equivalent to $\rank Df\leq n$.

Finally, note that \eqref{eq181} also holds for H\"older continuous curves in $\bbbh^n$, see \Cref{T67} (cf.\ \cite[Lemma~3.1]{LZ}). This is used to characterize existence of liftings of 
H\"older continuous mappings $f:\Omega\to\R^{2n}$ to H\"older continuous mappings $F:\Omega\to\bbbh^n$, see \Cref{T69}.

Using methods of \Cref{HM}, in \Cref{S4} we prove a generalization of the Gromov non-embedding theorem \cite[3.1.A]{gromov2} due to \cite{HS23}, see Theorems~\ref{T55} and~\ref{T53}. In a special case, the theorem says that there is no mapping $f\in C^{0,\gamma}(\bbbb^{n+1};\bbbh^n)$, $\gamma\in\big(\frac{n+1}{n+2},1\big]$, such that $f|_{\Sph^n}=f|_{\partial\bbbb^{n+1}}$ is a topological embedding (cf.\ \Cref{T97}). A rough idea of the proof (of this special case) is that for such a mapping $f(\Sph^n)$ would bound a ``hole''. The precise meaning of the ``hole'' can be expressed in terms of the Alexander-Poincar\'e duality (cf.\ \Cref{T21} and \Cref{R13}) according to which the de Rham cohomology with compact support of the complement satisfies
\begin{equation}
\label{eq189}
H_c^{n+1}(\R^{2n+1}\setminus f(\Sph^n))=\R.
\end{equation}
Now, if $\eta\in\Omega_c^{n+1}(\R^{2n+1}\setminus f(\Sph^n))$, $d\eta=0$, is a generator of \eqref{eq189}, then the integration by parts yields 
$$
\int_{\bbbb^{n+1}} f^*\eta\neq 0
$$
and this contradicts \Cref{T47} with  $\cM=\bbbb^{n+1}$, $\kappa=\eta$ and $\tau\equiv 1$ (cf.\ \eqref{eq188}). The actual proof of \Cref{T53} is more complicated than that of the special case described here, 
because the result is more general; we cannot refer directly to \Cref{T47}, but we mimic its proof.

The content of \Cref{S4} is somewhat related to the work of Balogh, Kozhevnikov, and Pansu \cite{BKP}. Adapting a technique for integrating differential forms along H\"older maps, developed by Z\"ust \cite{zust2,zust1} and Le Donne and Z\"ust \cite{LZ}, they proved that any $C^{0,\gamma}$ homeomorphism between $\R^{2n+1}$ and $\bbbh^n$ must satisfy $\gamma\leq\frac{n+1}{n+2}$. Their results also apply to more general Carnot groups.

\Cref{S:Hodge} provides background material about the Hodge decomposition needed for the last \Cref{NHG}. This section is devoted to study of the H\"older homotopy groups of the Heisenberg group. While several results about the Lipschitz homotopy groups have already been proved, see \Cref{T111}, very little is known about the H\"older homotopy groups. 
First note that it is very easy to prove that $\pi_k^\gamma(\bbbh^n)=0$ for all $k,n\geq 1$ and $\gamma\in\big(0,\frac{1}{2}\big)$ (see \Cref{T106}), but the case $\gamma>\frac{1}{2}$ is more complicated.

In \Cref{T108} we prove that
$$
\pi_n^\gamma(\bbbh^n)\neq 0
\quad 
\text{ for all  }{\textstyle \gamma\in \big(\frac{n+1}{n+2},1\big]}.
$$
In fact it follows immediately from the special case of the non-embedding theorem mentioned above i.e., \Cref{T97}, and \Cref{T108} should be credited to Gromov. If $n=1$, we obtain $\pi_1^\gamma(\bbbh^1)\neq 0$ for all $\gamma\in \big(\frac{2}{3},1\big]$. This result is sharp, because Wenger and Young \cite[Theorem~1.1]{wengerYoung18} proved that $\pi_1^\gamma(\bbbh^1)=0$ for all $\gamma\in\big(0,\frac{2}{3}\big)$. The case $\gamma=\frac{2}{3}$ is, however, open.

The main result of \Cref{NHG} is \Cref{th:hopfgamma}:
$$
\pi_{4n-1}^\gamma(\bbbh^{2n})\neq 0
\quad
\text{ for all } {\textstyle \gamma\in \big(\frac{4n+1}{4n+2},1\big].}
$$
The proof is difficult. The Lipschitz case $\pi_{4n-1}^{\rm Lip}(\bbbh^{2n})\neq 0$ was proved in \cite{HST14} using a certain generalization of the Hopf invariant that required the use of the Hodge decomposition. The proof presented here follows a similar approach, but now all estimates have to be adapted to the H\"older category. This makes the arguments much more difficult.

Finally, let us describe the content of the preliminary \Cref{pre}.
In Sections~\ref{EADF},~\ref{GCS} and~\ref{symp}, we review exterior algebra and differential forms, including some properties of contact and symplectic forms.  While the material is known, we couldn't locate \Cref{T48} and \Cref{T49} in the literature.

\Cref{Holder} is devoted to a general theory of H\"older continuous functions. Again, the material is well known, but it is not easy to find all  results in one place. Because of an increasing interest in analysis of H\"older continuous functions and mappings, we decided to make this section detailed and self-contained with hope that it will be a useful reference to the readers.

\Cref{Sobolev} reviews the Sobolev spaces $W^{1,p}$ and the fractional Sobolev spaces $W^{1-\frac{1}{p},p}$ that are the boundary trace spaces for $W^{1,p}$ spaces. In \Cref{T19} we state that there are bounded trace and extension operators
$$
\operatorname{Tr}:W^{1,p}(\cM)\to W^{1-\frac{1}{p},p}(\partial \cM)
\quad
\text{and}
\quad
\ext_{\partial\cM}:W^{1-\frac{1}{p},p}(\partial \cM)\to W^{1,p}(\cM).
$$
Unfortunately, no short proofs of this result are known. While there are elementary proofs, all of them are pretty technical. For that reason, we decided to include a proof of a special case of the boundedness of the extension operator in \Cref{S5}:
$$
\ext_{\partial\cM}:C^{0,\gamma}(\partial \cM)\to W^{1,p}(\cM)
\quad
\text{provided } \textstyle{1-\frac{1}{p}<\gamma\leq 1},
$$
see \Cref{T6}, \Cref{T57},  and \Cref{T7}. Since the paper is devoted to analysis of H\"older continuous mappings, this special case is the most important one.

\section{Preliminaries}
\label{pre}
In this section we collect definitions and results related to differential forms, H\"older continuous mappings,  Sobolev spaces and fractional Sobolev spaces needed in the rest of the paper. While most of the material presented here is known, some of the results proved here could be difficult to find in the literature. 
We assume that the reader is familiar with basic theory of differential forms, Riemannian manifolds, and Sobolev spaces. We recall some of the elementary definitions with the purpose of clarifying notation used in the paper.

\subsection{Notation}
\label{S2}
By $C$ we will denote a generic constant whose value may change in a single string of estimates.
If $A$ and $B$ are non-negative quantities, we write $A\lesssim B$ and $A\approx B$ if there is a constant $C\geq 1$ such that $A\leq CB$ and $C^{-1}B\leq A\leq CB$ respectively. This will always be clear from the context on what parameters $C$ depends. We will also write $A\lesssim_{U,V} B$ and $A\approx_{U,V} B$ if the implied constants depend on parameters $U$ and $V$.
If $0\neq a\in\R$, then the sign of $a$ is defined as $\sgn a=1$ if $a>0$ and $\sgn a=-1$ if $a<0$.

$\bbbn$ will denote the set of positive integers, and
$\R^n_+=\{x\in\R^n:\, x_n\geq 0\}$ will denote a half-space.
$A\Subset B$ means that the closure $\overbar{A}$ of $A$ is a compact subset of $B$.
We will also use notation $X\Subset Y$ to denote compact embedding of Banach spaces.
An open ball of radius $R$ will be denoted by $B(x,R)$ or $B^n(x,R)$ and we will also write  $B_R^n=B_R=B(0,R)$. The unit ball and the unit sphere in $\bbbr^n$ will be denoted by $\bbbb^n=B^n(0,1)$ and $\Sph^{n-1}=\partial\bbbb^n$.
The volume of the unit ball in $\mathbb{R}^n$ will be denoted by $\omega_n$.
Uniform convergence will often be denoted by $f_n\rightrightarrows f$.

If $\Omega\subset\R^n$ is a domain, then $f\in C^\infty(\overbar{\Omega})$ means that $f$ can be extended to a function in $C^\infty(\R^n)$.

We say that a bounded domain $\Omega\subset\R^n$ is Lipschitz if its boundary is locally a graph of a Lipschitz function.  A domain is smooth if its boundary is locally a graph of a smooth function. Smooth domains are examples of Riemannian manifolds with boundary.

By a smooth function or a smooth manifold we will mean a $C^\infty$ smooth function or  a manifold.
A closed manifold is a smooth compact manifold without boundary. If $\cM$ is a Riemannian manifold, then $C^{\infty}(\overbar{\cM})$ denotes the class of functions that are smooth up to the boundary. Clearly, if $\partial\cM=\varnothing$, then $C^\infty(\overbar{\cM})=C^\infty(\cM)$.

By $C^{0,\gamma}$, $0<\gamma\leq 1$ we will denote the class of $\gamma$-H\"older continuous functions.
We will also use notation $\lip:=C^{0,1}$ for Lipschitz functions.
Various classes of compactly supported functions will be denoted by $C_c^\infty, C_c^{0,\gamma}$, etc.

De~Rham cohomology and de~Rham cohomology with compact support on an oriented manifold $\cM$ will be dented by $H^k(\cM)$ and $H^k_c(\cM)$ respectively. No other cohomoloy theory will be considered.

If $X$ is a space of real valued function, then $X^n$ stands the space of mappings $f=(f_1,\ldots,f_n)$ such that $f\in X$. For example, $C^\infty(\R^n)^m$ is the class of smooth mappings from $\R^n$ to $\R^m$. If $X$ is a normed space, then $X^n$ is also a normed space with the norm $\Vert f\Vert=\sum_{i=1}^n\Vert f_i\Vert$.

The $L^p$ norm $1\leq p\leq \infty$ of a function $f$ will be denoted by $\Vert f\Vert_p$. If a function $f$ defined in $\Omega\subset\bbbr^n$ is sufficiently smooth, $D^0f=f$ and $D^m f$, $m\in\bbbn$, denotes the vector of all partial derivatives of order $m$, so
$$
|D^m f|:=\Big(\sum_{|\alpha|=m} |D^\alpha f|^2\Big)^{1/2},
\quad
\text{and}
\quad
\Vert f\Vert_{m,p}:=\sum_{k=0}^m\Vert D^k f\Vert_p,
\quad 1\leq p\leq\infty,\ m\in\bbbn,
$$
is the Sobolev norm.
The integral average will be denoted by the barred integral:
\begin{equation}
\label{eq172}
\mvint_E f\, d\mu =\frac{1}{\mu(E)}\int_E f\, d\mu.
\end{equation}

The Jacobian of a mapping $f:\Omega\subset\R^n\to\R^n$ is denoted by
$J_f(x):=\det Df(x)$.

Symbols $\upomega$ and $\upalpha$ will be solely used to denote the standard symplectic and the standard contact forms defined by
$$
\upomega=\sum_{i=1}^{n} dx_i\wedge dy_i
\qquad
\text{and}
\qquad
\upalpha=dt+2\sum_{i=1}^n (x_idy_i-y_idx_i).
$$
By writing $w_1\wedge\ldots\wedge\widehat{w_i}\wedge\ldots\wedge w_k$ we mean that the form $w_i$ is omitted in the wedge product.

\subsection{Exterior algebra and differential forms}
\label{EADF}
The canonical bases of $\R^n$ and $(\R^n)^*$ will be denoted by
$$
\frac{\partial}{\partial x_1},\ldots,\frac{\partial}{\partial x_n}
\qquad
\text{and}
\qquad
dx_1,\ldots,dx_n
$$
respectively.
The $k$-covectors
$$
dx_I:=dx_{i_1}\wedge\ldots\wedge dx_{i_k},
\quad
I=(i_1,\ldots,i_k),\ 
1\leq i_1<\ldots<i_k\leq n
$$
form a basis of $\Ep^k(\R^n)^*$. We will denote the length of $I$ by $|I|=k$. 

The $k$-covectors are identified with alternating multilinear maps on $\R^n\times\ldots\times\R^n$ that on basis elements are defined by
$$
dx_{i_1}\wedge\ldots\wedge dx_{i_k}(v_1,\ldots,v_k)=
\det
\left [\begin{array}{ccc}
dx_{i_1}(v_1) & \ldots & dx_{i_1}(v_k)\\
\vdots & \ddots & \vdots\\
dx_{i_k}(v_1) & \ldots & dx_{i_k}(v_k)
\end{array} \right]=
\det
\left [\begin{array}{ccc}
v_1^{{i_1}} & \ldots & v_k^{{i_1}}\\
\vdots & \ddots & \vdots\\
v_1^{i_k} & \ldots & v_k^{{i_k}}
\end{array} \right],
$$
where
$$
v_j=v^{1}_j\frac{\partial}{\partial x_1}+\ldots+v^{n}_j\frac{\partial}{\partial x_n}.
$$

The wedge product
$$
\wedge: \Ep^k (\mathbb{R}^n)^\ast \times \Ep^\ell (\mathbb{R}^n)^\ast \mapsto \Ep^{k+\ell} (\mathbb{R}^n)^\ast
$$ 
is  a bilinear map defined on bases as
$$
dx_I \wedge dx_J = dx_{i_1} \wedge \ldots \wedge dx_{i_k} \wedge dx_{j_1} \wedge \ldots \wedge dx_{j_\ell},
$$
where
$$
I=(i_1,\ldots,i_k),\ J=(j_1,\ldots,j_\ell),\ 1\leq i_1<\ldots<i_k\leq n,\ 
1\leq j_1<\ldots<j_\ell\leq n.
$$
Here, we identify $dx_i \wedge dx_j = -dx_j\wedge dx_i$; in particular $dx_i \wedge dx_i = 0$.

It is worth noting that for $\alpha \in \Ep^k \R^n$ and $\beta \in \Ep^\ell \R^n$ we have 
$$
 \alpha \wedge \beta = (-1)^{k\ell} \beta \wedge \alpha.
$$
$\Ep^k (\mathbb{R}^n)^\ast$ is endowed with a scalar product defined 
on the basis by
$$
\langle dx_I , dx_J \rangle = \begin{cases}
1 \quad&I=J\\
0 \quad&\text{otherwise.}
\end{cases}
$$
The Hodge star operator $\ast: \Ep^k (\mathbb{R}^n)^\ast \to \Ep^{n-k} (\mathbb{R}^n)^\ast$ is a unique linear isomorphism that satisfies
$$
\alpha \wedge (\ast \beta) = \langle \alpha,\beta\rangle\, dx^1\wedge \ldots \wedge dx^n \qquad \text{for all } \alpha,\beta \in \Ep^k (\mathbb{R}^n)^\ast.
$$
The Hodge star operator can be also computed more explicitly: If $I=(i_1,\ldots,i_k)$, $1 \leq i_1 < i_2 < \ldots < i_k \leq n$, then
\[
 \ast dx_I = \sgn(\sigma(I)) dx_J,
\]
where $J=(j_1,\ldots,j_{n-k})$, $1 \leq j_1 < j_2 < \ldots < j_{n-k} \leq n$, $I \cap J = \varnothing$, 
is the set of complementary indices and $\sgn(\sigma(I))=\pm 1$ is the sign of the permutation
$\sigma(I):(1,2,\ldots,n)\to (i_1,\ldots,i_k,j_1,\ldots,j_{n-k})$.
Moreover we have,
\[
 \ast \ast \eta = (-1)^{k(n-k)} \eta \quad \quad \eta \in \Ep^k (\R^n)^*.
\]
Smooth differential $k$-forms on $\R^{n}$ are smooth sections of $\Ep^k (\R^n)^*$:
$$
\Omega^k(\R^n):=C^\infty\left(\R^n,\Ep^k(\R^n)^*\right).
$$
If $X$ is a space of functions on $\R^n$, then we also consider the space of forms
$\Omega^k X(\R^n)$ with coefficients in $X$ (so by definition $\Omega^k(\R^n)=\Omega^kC^\infty(\R^n)$). If $X$ is a Banach space, the space $\Omega^kX(\bbbr^n)$ is equipped with a Banach norm
$$
\Vert\kappa\Vert_X=\sum_{|I|=k}\Vert \kappa_I\Vert_X,
\quad
\text{for  $\kappa=\sum_{|I|=k}\kappa_I\, dx_I$.}
$$
That includes, for example, forms with  H\"older, $L^p$, Sobolev or fractional Sobolev coefficients
$$
\Omega^k C^{0,\gamma}(\R^n),
\quad
\Omega^k L^p(\R^n),
\quad
\Omega^k W^{1,p}(\R^n),
\quad
\Omega^k W^{1-\frac{1}{p},p}(\R^n).
$$
H\"older, Sobolev or fractional Sobolev spaces will be discussed in Sections~\ref{Holder} and~\ref{Sobolev}.

The exterior derivative $d:\Omega^k(\R^n)\to\Omega^{k+1}(\R^n)$ is a linear operator uniquely determined by
$$
d\brac{f dx_{i_1}\wedge \ldots \wedge dx_{i_k}} = 
\sum_{j=1}^n \partial_{x_j} f dx^j \wedge dx_{i_1}\wedge \ldots \wedge dx_{i_k}.
$$
Clearly, this construction also gives exterior derivatives of forms with other regularity e.g.,
$$
d:\Omega^k C^1(\R^n)\to \Omega^{k+1} C^0(\R^n) 
\quad
\text{or}
\quad
d:\Omega^k W^{1,p}(\R^n)\to \Omega^{k+1} L^p(\R^n).
$$
It is worth noting that for a $k$-form $\alpha$ and an $\ell$-form $\beta$ we have
\begin{equation}\label{eq:dproductrule}
 d (\alpha \wedge \beta) = d\alpha \wedge \beta + (-1)^k \alpha \wedge d \beta.
\end{equation}
The co-differential $\delta: \Omega^k(\R^n) \to \Omega^{k-1}(\R^n)$ is defined as 
$$
\delta = (-1)^{n(k-1)+1}\ast d \ast,
$$
with the obvious generalization to $C^1$ and Sobolev forms. The Laplace-de~Rham operator $\lap: \Omega^k(\R^n) \to \Omega^k(\R^n)$ is defined as 
\[
 \lap = d \delta + \delta d.
\]
One can check that for functions $f: \R^n \to \R$ (i.e. for $k=0$-form) it coincides with the usual Laplace operator.

\subsection{Good coordinate systems}
\label{GCS}
While in \Cref{EADF} forms were defined on $\R^n$, one can use local coordinate systems to define forms (including Sobolv, H\"older etc.) on a smooth oriented manifold $\cM$.
A differential $k$-form $\gamma$ on a smooth oriented $n$-dimensional manifold $\cM$, with or without boundary, can be expressed in local coordinates
$x=(x_1,\ldots,x_n):U\subset\cM\to\R^n$ as
$$
\gamma =\sum_{|I|=k}\gamma_I\, dx_I
\quad
\text{in $U$.}
$$
We will only consider {\em good coordinate systems} $(x,U)$ that satisfy:
\begin{enumerate}
\item If $U\cap\partial\cM=\varnothing$, then $x(U)=\bbbb^n$ and  $x$ extends to a smooth coordinate system in a larger domain $V\Supset U$. 
\item If $U$ is a neighborhood of a boundary point, then $x(U)=\bbbb^n_+:=\bbbb^n\cap\R^n_+$ is the unit half-ball and $x$ extends to a smooth coordinate system in a larger domain $V\Supset U$. 
\item If $\cM$ is oriented, then $(x,U)$ is orientation preserving.
\end{enumerate}
Note that our assumptions will guarantee boundedness of derivatives of all orders.

If a function space $X$ `behaves well' under a change of variables, and $\cM$ is a smooth oriented closed $n$-dimensional manifold,
we say that $\gamma\in \Omega^k X(\cM)$ if for every good coordinate system $(x,U)$, 
we have $(x^{-1})^*\gamma\in\Omega^kX(\bbbb^n)$.
Then we define the norm of 
$\gamma\in \Omega^k X(\cM)$ by
\begin{equation}
\label{eq38}
\Vert\gamma\Vert_X:=\sum_{i=1}^N \Vert (x^{-1}_{i})^*\gamma \Vert_{X(\bbbb^n)},
\end{equation}
where the sum runs over a finite covering $\cM=\bigcup_{i=1}^N U_i$ by
good coordinate systems $\{ (x_i,U_i)\}_{i=1}^N$.
Due to compactness of $\cM$, the norms defined for different coverings are equivalent. 
This approach applies to spaces like $C^{0,\gamma}$, $L^p$, $W^{1,p}$ or $W^{1-\frac{1}{p},p}$.

If $\gamma \in \Omega^k(\cM)$ (or $\gamma\in\Omega^kW^{1,p}(\cM)$ etc.), then we can define the differential $d$, via the local representation in a good coordinate systems $U$
\[
d\gamma =\sum_{|I|=k} d\brac{\gamma_I\, dx_I} = \sum_{i=1}^n \sum_{|I|=k} \partial_{x_i} \gamma_I\, dx_i \wedge dx_I \quad \text{in $U$.}
\]
The Hodge-star operator $\ast$ was defined with respect to the Euclidean structure (scalar product) in $\R^n$. If $\cM$ is an oriented $n$-dimensional Riemannian manifold, then we have the orientation and the scalar product in every tangent space and hence in every cotangent space $T_x^*\cM$, which allows us to define the Hodge-star operator $\ast:\Ep^kT_x^*\cM\to \Ep^{n-k}T_x^*\cM$. Then we define $\delta: \Omega^k(\cM) \to \Omega^{k-1}(\cM)$ and $\lap: \Omega^k(\cM) \to \Omega^k(\cM)$ as in the Euclidean case 
$$
\delta = (-1)^{n(k-1)+1}\ast d \ast,
\qquad
 \lap = d \delta + \delta d.
$$
One can easily extend the above constructions to smooth oriented compact manifolds with boundary.

If $\cM$ is an oriented and compact $n$-dimensional Riemannian manifold without boundary, then we have the itegtation by parts formula
\begin{equation}
\label{eq137}
\int_\cM\langle d\alpha,\beta\rangle=\int_\cM\langle\alpha,\delta\beta\rangle
\quad
\text{for all $\alpha\in \Omega^k(\cM)$ and $\beta\in\Omega^{k+1}(\cM)$}.
\end{equation}

\begin{lemma}[Stokes theorem]
\label{T78}
Let $\cM$ be an oriented smooth manifold, $\dim\cM=n$. If $\omega$ is a compactly supported Lipschitz $(n-1)$-form, then
$$
\int_{\cM}d\omega=\int_{\partial\cM}\omega.
$$
\end{lemma}

Since the integration by parts holds for Lipschitz functions, the standard proof (cf.\ \cite[Theorem~16.11]{Lee}) applies verbatim to the case of Lipschitz forms.

\subsection{Symplectic and contact forms}
\label{symp}

\begin{definition}
With the coordinates in $\R^{2n}$ denoted by $(x_1,y_1,\ldots,x_n,y_n)$, the {\em standard symplectic form} is
\begin{equation}
\label{eq128}
\upomega=\sum_{i=1}^n dx_i\wedge dy_i\in\Ep^2(\R^{2n})^*
\quad
\text{that is,}
\quad
\upomega(v,w)=
\sum_{i=1}^n \left( v^{x_i}w^{y_i}- v^{y_i}w^{x_i}\right),
\end{equation}
where
$$
v=\sum_{i=1}^n v^{x_i}\frac{\partial}{\partial x_i} + v^{y_i} \frac{\partial}{\partial y_i}
\quad
\mbox{and}
\quad
w=\sum_{i=1}^n w^{x_i}\frac{\partial}{\partial x_i} + w^{y_i} \frac{\partial}{\partial y_i}.
$$
We also regard $\upomega$ as a differential form
$\upomega\in \Omega^2(\R^{2n})$ with constant coefficients.
\end{definition}

\begin{remark}
In what follows $\upomega$ will always stand for the standard symplectic form.
\end{remark}

\begin{definition}
A vector space $V\subset\R^{2n}$ is {\em isotropic} if $\upomega(v,w)=0$ for all $v,w\in V$.
\end{definition}
\begin{lemma}
\label{T62}
If $V\subset\R^{2n}$ is isotropic, then $\dim V\leq n$.
\end{lemma}
\begin{proof}
Let $\cJ:\R^{2n}\to \R^{2n}$ be an isomorphism defined by
$$ 
\cJ\left(\frac{\partial}{\partial x_i} \right) = \frac{\partial}{\partial y_i}, 
\qquad 
\cJ\left(\frac{\partial}{\partial y_i} \right) = -\frac{\partial}{\partial x_i}  
\quad\text{for $i=1,\hdots,n$.}
$$
It is easy to check that
$\upomega(v,w)=-\langle v,\cJ w\rangle$
for $v,w\in\R^{2n}$,
where $\langle\cdot,\cdot\rangle$ is the standard scalar product in $\R^{2n}$.
If $V\subset\R^{2n}$ is isotropic, 
then for all $v,w\in V$, $0=\upomega(v,w)=-\langle v,\cJ w\rangle$. Therefore,
$\cJ$ maps $V$ to the orthogonal complement of $V$ so 
$$
2\dim V=\dim V+\dim (\cJ V)\leq 2n, 
\qquad
\dim V\leq n.
$$
\end{proof}
\begin{remark}
If we identify $T\R^{2n}$ with $T\bbbc^n$ by
$$
T\R^{2n}\ni\Big(\frac{\partial}{\partial x_1},\,\frac{\partial}{\partial y_1},\ldots,\frac{\partial}{\partial x_n},\,\frac{\partial}{\partial y_n}\Big)\mapsto
\Big(\frac{\partial}{\partial x_1}+i\,\frac{\partial}{\partial y_1},\ldots,\frac{\partial}{\partial x_n}+i\,\frac{\partial}{\partial y_n}\Big)\in T\bbbc^n,
$$
then $\cJ$ corresponds to the multiplication by $i$.
\end{remark}

\begin{definition}
Given $w\in\R^n$ and $1\leq k\leq n$, we define the {\em interior product}
$$
\iota_w:\Ep^k(\R^n)^*\to \Ep^{k-1}(\R^n)^*
\quad
\text{by}
\quad
\iota_w(\xi)(v_1,\ldots,v_{k-1})=\xi(w,v_1,\ldots,v_{k-1}).
$$
\end{definition}
\begin{lemma}
\label{T10}
If $\xi\in\Ep^k(\R^n)^*$, $\eta\in\Ep^\ell(\R^n)^*$ and $v\in\R^n$, then
$$
\iota_v(\xi\wedge\eta)=\iota_v(\xi)\wedge\eta+(-1)^k\xi\wedge\iota_v(\eta).
$$
\end{lemma}
For a proof, see \cite[Lemma~14.13]{Lee} or \cite[Proposition~2.12]{warner}.

If $\gamma$ is a differential form, then we will write:
$$
\gamma^k=\underbrace{\gamma\wedge\ldots\wedge\gamma}_{k}.
$$

We have
\begin{equation}
\label{eq101}
\upomega^n=\underbrace{\upomega\wedge\ldots\wedge\upomega}_n=n!\, dx_1\wedge dy_1\wedge\ldots\wedge dx_n\wedge dy_n.
\end{equation}
\Cref{T10} and a simple induction argument imply that
\begin{equation}
\label{eq102}
\iota_v(\upomega^k)=k\upomega^{k-1}\wedge\iota_v(\upomega),
\quad
\text{for $v\in\R^{2n}$.}
\end{equation}
\begin{lemma}
\label{T22}
For $1\leq k\leq n$, the {\em Lefschetz operator}
$$
L^k:\Ep^{n-k}(\R^{2n})^*\to \Ep^{n+k}(\R^{2n})^*,
\quad
L^k(\xi)=\xi\wedge\upomega^k
$$
is an isomorphism.
\end{lemma}
\begin{remark}
\label{R1}
The result is known. Standard proofs involve representation theory of $\mathfrak{sl}(2)$, see \cite[Proposition~1.2.30]{huybrechts}. The proof presented below, due to Calabi, is very elementary, see \cite[Proposition~1.1a]{Bryant} and \cite{mathoverflow1}.
\end{remark}
\begin{proof}
Since the spaces $\Ep^{n-k}(\R^{2n})^*$ and $\Ep^{n+k}(\R^{2n})^*$ have equal dimensions, it suffices to show that $L^k$ is injective. We will use a backwards induction starting with $k=n$. If $k=n$, then \eqref{eq101} implies that
$$
L^n:\Ep^0(\R^{2n})^*=\R\to \Ep^{2n}(\R^{2n})^*
$$
is an isomorphism and hence it is injective. 

Suppose that the result is true for $k$. We need to prove it is true for $k-1$. Thus 
we need to prove that if
$L^{k-1}(\xi)=\xi\wedge\upomega^{k-1}=0$ for some
$\xi\in\Ep^{n-(k-1)}(\R^{2n})^*$, then $\xi=0$.

Clearly, $\xi\wedge\upomega^k=0$. If $v\in\R^{2n}$, then \Cref{T10} and \eqref{eq102} yield
$$
0=\iota_v(\xi\wedge\upomega^k)=
\iota_v(\xi)\wedge\upomega^k\pm \xi\wedge\iota_v(\upomega^k)=
\iota_v(\xi)\wedge\upomega^k\pm k\underbrace{\xi\wedge\upomega^{k-1}}_{0}\wedge\iota_v(\upomega)
$$
so $L^k(\iota_v(\xi))=\iota_v(\xi)\wedge\upomega^k=0$ and hence $\iota_v(\xi)=0$ by the injectivity of $L^k$ (induction hypothesis).
Since $\iota_v(\xi)=0$ for all $v\in\R^{2n}$, it follows that $\xi=0$. The proof is complete.
\end{proof}

\begin{corollary}
\label{T23}
The exterior product with the symplectic form is an isomorphism:
$$
L^1(\cdot)=
\cdot\wedge\upomega:\Ep^{n-1}(\R^{2n})^*\to\Ep^{n+1}(\R^{2n})^*.
$$
\end{corollary}
\begin{definition}
\label{def:alpha}
The {\em standard contact form} on $\R^{2n+1}$ with coordinates $(x_1,y_1,\ldots,x_n,y_n,t)$ is
$$
\upalpha=dt+2\sum_{j=1}^n(x_jdy_j-y_jdx_j)\in \Omega^1(\R^{2n+1}).
$$
\end{definition}
\begin{remark}
In what follows $\upalpha$ will always stand for the standard contact form.
\end{remark}

Coordinates in $\R^{2n}$ and $\R^{2n+1}$ will be denoted by
$$
(x_1,y_1,\ldots,x_n,y_n)
\quad
\text{and}
\quad
(x_1,y_1,\ldots,x_n,y_n,t)
$$
respectively so
\begin{equation}
\label{eq103}
\{dx_I\wedge dy_J:\, |I|+|J|=k\}
\end{equation}
is a basis of $\Ep^k(\R^{2n})^*$ and
\begin{equation}
\label{eq104}
\{dx_I\wedge dy_J,\ dx_{I'}\wedge dy_{J'}\wedge dt:\, |I|+|J|=k,\ |I'|+|J'|=k-1\}
\end{equation}
is a basis of $\Ep^k(\R^{2n+1})^*$.

We have a natural embedding 
\begin{equation}
\label{eq105}
\Ep^k(\R^{2n})^*\subset\Ep^{k}(\R^{2n+1})^*
\end{equation} 
which maps the basis \eqref{eq103} onto a subset of the basis \eqref{eq104} by the identity map. With this embedding elements of $\Ep^k(\R^{2n})^*$ are regarded as elements of $\Ep^{k}(\R^{2n+1})^*$ that do not contain $dt$.

In particular, at every point $p\in\R^{2n+1}$ we can identify
$$
d\upalpha(p)=4\sum_{i=1}^n dx_i\wedge dy_i\in \Ep^2(\R^{2n+1})^*
\quad
\text{with}
\quad
4\upomega\in \Ep^2(\R^{2n})^*\subset \Ep^2(\R^{2n+1})^*.
$$

\begin{proposition}
\label{T48}
If $\kappa\in \Omega^k(\R^{2n+1})$, $n+1\leq k\leq 2n$, then there are smooth forms
$$
\beta\in \Omega^{k-1}(\R^{2n+1})
\quad
\text{and}
\quad
\gamma\in \Omega^{2n-k}(\R^{2n+1})
$$
that do not contain components with $dt$, and satisfy
$$
\kappa=\beta\wedge\upalpha+\gamma\wedge(d\upalpha)^{k-n}.
$$
Moreover, if $\kappa(p)=0$, then $\beta(p)=0$ and $\gamma(p)=0$.
\end{proposition}
\begin{proof}
Since
$$
dx_1\wedge dy_1\wedge\ldots\wedge dx_n\wedge dy_n\wedge \upalpha(p)=
dx_1\wedge dy_1\wedge\ldots\wedge dx_n\wedge dy_n\wedge dt\neq 0,
$$
we conclude that for every $p\in\R^{2n+1}$,
$\{dx_1,dy_1,\ldots,dx_n,dy_n,\upalpha(p)\}$
is a basis of $T^*_p\R^{2n+1}=\Ep^1 T^*_p\R^{2n+1}$.
This basis induces a basis of $\Ep^k T_p^*\R^{2n+1}$ consisting of the elements of the form
$$
dx_I\wedge dy_J  
\qquad
\text{and}
\qquad
dx_{I'}\wedge dy_{J'}\wedge\upalpha(p),  
$$
where $|I|+|J|=k$ and $|I'|+|J'|=k-1$ so we have a unique decomposition
$$
\kappa(p)=\beta(p)\wedge \upalpha(p)+\delta(p),
$$
where 
$\beta(p)$ belongs to the span of $dx_{I'}\wedge dy_{J'}$, $|I'|+|J'|=k-1$, and
$\delta(p)$, belongs to the span of elements $dx_I\wedge dy_J$, $|I|+|J|=k$, that do not contain $dt$. 
Since the decomposition is unique, if $\kappa(p)=0$, then $\beta(p)=0$ and $\delta(p)=0$. This and Lemma~\ref{T22} will also imply that $\gamma(p)=0$.

Since $dx_I\wedge dy_J$, $|I|+|J|=k$ is a basis of 
$\Ep^k(\R^{2n})^*\subset\Ep^k(\R^{2n+1})^*$ (see \eqref{eq105}) and coefficients of $\delta(p)$ smoothly depend on $p$, we have that
$$
\delta\in C^\infty\left(\R^{2n+1},\Ep^k(\R^{2n})^*\right)\subset
C^\infty\left(\R^{2n+1},\Ep^k(\R^{2n+1})^*\right)=\Omega^k(\R^{2n+1}).
$$
According to \Cref{T22},
$$
L^{k-n}:\Ep^{2n-k}(\R^{2n})^*\to\Ep^k(\R^{2n})^*
\quad
\text{and hence}
\quad
(L^{k-n})^{-1}:\Ep^k(\R^{2n})^*\to\Ep^{2n-k}(\R^{2n})^*
$$
are isomorphisms. Therefore,
$$
\gamma:=4^{n-k}(L^{k-n})^{-1}\circ\delta\in
C^\infty\left(\R^{2n+1},\Ep^{2n-k}(\R^{2n})^*\right)\subset
\Omega^{2n-k}(\R^{2n+1})
$$
is a smooth form that belongs to the span of $dx_I\wedge dy_J$, $|I|+|J|=2n-k$.
Now the definition of the Lefschetz operator yields
$$
\delta=4^{k-n}L^{k-n}\circ\gamma=4^{k-n}\gamma\wedge\upomega^{k-n}=\gamma\wedge (d\upalpha)^{k-n}.
$$
In the last equality we used the identification of 
$$
4^{k-n}\upomega^{k-n}\in \Ep^{2k-2n}(\R^{2n})^*\subset \Ep^{2k-2n}(\R^{2n+1})^*
\quad
\text{with}
\quad
(d\upalpha)^{k-n}\in \Ep^{2k-2n}(\R^{2n+1})^*.
$$
The proof is complete.
\end{proof}
\begin{corollary}
\label{T49}
If $\kappa\in \Omega^k(\R^{2n+1})$, $n+1\leq k\leq 2n+1$, then there are smooth differential forms
$$
\beta,\beta'\in \Omega^{k-1}(\R^{2n+1})
\quad
\text{and}
\quad
\gamma,\gamma'\in \Omega^{k-2}(\R^{2n+1}) 
$$
such that
$$
\kappa=\beta\wedge\upalpha+\gamma\wedge d\upalpha=
\beta'\wedge\upalpha+d(\gamma'\wedge\upalpha).
$$
Moreover, if $\kappa(p)=0$, then $\beta(p)=0$ and $\gamma(p)=0$ and if $\kappa=0$ in an open set $U$, then $\beta'=0$ and $\gamma'=0$ in $U$.
\end{corollary}
\begin{proof}
If $n+1\leq k\leq 2n$, then \Cref{T48} yields
$$
\kappa=\beta\wedge\upalpha +\tilde{\gamma}\wedge (d\upalpha)^{k-n} =
\beta\wedge\upalpha + \underbrace{(\tilde{\gamma}\wedge (d\upalpha)^{k-n-1})}_{\gamma}\wedge d\upalpha.
$$
If $k=2n+1$, then $\kappa$ is proportional to the volume form $(d\upalpha)^n\wedge \upalpha$ so there is a smooth function $f$ such that
$$
\kappa=f(d\upalpha)^n\wedge\upalpha=
0\wedge\upalpha +\underbrace{f\upalpha\wedge(d\upalpha)^{n-1}}_{\gamma}\wedge d\upalpha.
$$
This proves existence of the decomposition
$\kappa=\beta\wedge\upalpha+\gamma\wedge d\upalpha$ which also yields
$$
\kappa=\big(\beta-(-1)^{k-2}d\gamma\big)\wedge\upalpha+d\big((-1)^{k-2}\gamma\wedge\upalpha\big):=
\beta'\wedge\upalpha+d(\gamma'\wedge\upalpha).
$$
The proof is complete.
\end{proof}

\subsection{The space of H\"older continuous functions}
\label{Holder}
In this section we collect basic properties of spaces of H\"older continuous functions.
While all results presented here are known, we decided to provide details, because
some of the results are difficult to find and, secondly, 
good understanding of the material presented here plays an important role throughout the paper.
\begin{definition}
\label{d1}
Let $X$ and $Y$ be metric spaces and $0<\gamma\leq 1$. 
A mapping $f:X\to Y$ is $\gamma$-H\"older continuous if
\begin{equation}
\label{eq122}
[f]_{\gamma}=[f]_{C^{0,\gamma}}:=\sup\left\{\frac{d_Y(f(x),f(y))}{d_X(x,y)^\gamma}:\, x\neq y\right\}<\infty.
\end{equation}
The space of all $\gamma$-H\"older continuous mappings $f:X\to Y$ will be denoted by $C^{0,\gamma}(X;Y)$. 

$C^{0,\gamma}_{\rm b}(X;Y)$ is the space of bounded $\gamma$-H\"older continuous mappings i.e., the image of every mapping is contained in a ball in $Y$.

The space $C^{0,\gamma+}(X;Y)$ is a space of mappings $f\in C^{0,\gamma}(X;Y)$ such that for every compact set $K\subset X$ we have
$$
\lim_{t\to 0^+}\ \sup\left\{\frac{d_Y(f(x),f(y))}{d_X(x,y)^\gamma}:\ x,y\in K,\ 0<d_X(x,y)\leq t\right\}= 0.
$$

Fix $A\subset X$. Then for $\eps>0$, we define
\begin{equation}
\label{eq129}
[f]_{\gamma,\eps,A}:=\sup\left\{\frac{d_Y(f(x),f(y))}{d_X(x,y)^\gamma}:\, x,y\in A, \ 0<d_X(x,y)< \eps\right\},
\end{equation}
and we set $[f]_{\gamma,\eps}:=[f]_{\gamma,\eps,X}$.

Clearly $[f]_{\gamma,\eps,A}\leq [f]_{\gamma,\infty,X}=[f]_{\gamma,\infty}=[f]_{\gamma}$. 
Observe also that 
$\lim_{\eps\to 0} \, [f]_{\gamma,\eps,A}=0$
if  $f\in C^{0,\gamma+}(X;Y)$, and  $A\Subset X$ is a relatively compact subset of $X$.

Note that
$C_{\rm b}^{0,\gamma}(X;\R^m)$ is a Banach space with respect to the norm
$$
\Vert f\Vert_{C^{0,\gamma}}:=\Vert f\Vert_\infty+[f]_{\gamma}.
$$
We equip $C_{\rm b}^{0,\gamma+}(X;\R^m)$ with the norm of $C_{\rm b}^{0,\gamma}(X;\R^m)$.
In the special case when $m=1$ i.e., when $Y=\R$, we use notation $C^{0,\gamma}(X)$, $C^{0,\gamma+}(X)$, $C_{\rm b}^{0,\gamma}(X)$, $C_{\rm b}^{0,\gamma+}(X)$.
\end{definition}

The space $C_{\rm b}^{0,\gamma}(X)$ is a real Banach algebra with respect to the pointwise multiplication. Namely we have:
\begin{lemma}
\label{T27}
$C_{\rm b}^{0,\gamma}(X)$ is a real Banach space. Moreover if $f,g\in C_{\rm b}^{0,\gamma}(X)$, then
$fg\in C_{\rm b}^{0,\gamma}(X)$ and
$\Vert fg\Vert_{C^{0,\gamma}}\leq \Vert f\Vert_{C^{0,\gamma}}\Vert g\Vert_{C^{0,\gamma}}$.
\end{lemma}
\begin{proof}
The proof that $C_{\rm b}^{0,\gamma}(X)$ is a Banach space is easy and left to the reader. Since
$|(fg)(x)-(fg)(y)|\leq\Vert f\Vert_\infty|g(x)-g(y)|+|f(x)-f(y)|\Vert g\Vert_\infty$, it follows that
$$
[fg]_{\gamma}\leq \Vert f\Vert_\infty [g]_{\gamma}+[f]_{\gamma}\Vert g\Vert_\infty.
$$
Therefore
$$
\Vert fg\Vert_{C^{0,\gamma}}\leq \Vert f\Vert_\infty\Vert g\Vert_\infty +
\Vert f\Vert_\infty[g]_{\gamma}+[f]_{\gamma}\Vert g\Vert_\infty\leq \Vert f\Vert_{C^{0,\gamma}}\Vert g\Vert_{C^{0,\gamma}}.
$$
\end{proof}
\begin{corollary}
\label{T28}
If $f_\eps\to f$ and $g_\eps\to g$ in $C_{\rm b}^{0,\gamma}(X)$ as $\eps\to 0$, then $f_\eps g_\eps\to fg$ in $C_{\rm b}^{0,\gamma}(X)$ as $\eps\to 0$.
\end{corollary}
\begin{proof}
\Cref{T27} yields    
$\Vert f_\eps g_\eps-fg\Vert_{C^{0,\gamma}}\leq \Vert f_\eps\Vert_{C^{0,\gamma}}\Vert g_\eps-g\Vert_{C^{0,\gamma}}+\Vert f_\eps-f\Vert_{C^{0,\gamma}}\Vert g\Vert_{C^{0,\gamma}}$ and the right hand side converges to $0$ as $\eps\to 0$. 
\end{proof}
\begin{proposition}
\label{T80}
Let $X$ be a metric space. If $f,g\in C_{\rm b}^{0,\gamma}(X;\R^m)$ and $u\in C^{0,\alpha}(\R^m)$, $\alpha,\gamma\in (0,1]$, then
$$
u\circ f\in C^{0,\alpha\gamma}_{\rm b}(X).
$$
Moreover, for any $0<\beta<\alpha$ we have
\begin{equation}
\label{eq83}
\Vert u\circ f-u\circ g\Vert_{C^{0,\beta\gamma}}\lesssim
[u]_{\alpha}\left(\Vert f\Vert_{C^{0,\gamma}}+\Vert g\Vert_{C^{0,\gamma}}\right)^{\beta}
\Vert f-g\Vert_\infty^{\alpha-\beta},
\end{equation}
where the implied constant depends on $\alpha$, $\beta$ and $\gamma$ only.

In particular, the composition operator
\begin{equation}
\label{eq88}
C_{\rm b}^{0,\gamma}(X;\R^m)\ni f\mapsto u\circ f\in C^{0,\beta\gamma}_{\rm b}(X)
\end{equation}
is continuous.
\end{proposition}
\begin{remark}
In fact, we prove a sightly stronger estimate than \eqref{eq83},
$$
\Vert u\circ f-u\circ g\Vert_{C^{0,\beta\gamma}}\lesssim
[u]_{\alpha}\left([f]_\gamma+[g]_\gamma+\Vert f-g\Vert_\infty\right)^{\beta}
\Vert f-g\Vert_\infty^{\alpha-\beta}.
$$
\end{remark}
\begin{proof}
The fact that $u\circ f\in C^{0,\alpha\gamma}_{\rm b}(X)$ follows from  the observation that $u\circ f$ is bounded in the supremum norm (because the image of $f$ is bounded and $u$ is continuous) and the easy estimate
$$
[u\circ f]_{\alpha\gamma}\leq
[u]_{\alpha}[f]_{\gamma}^\alpha.
$$
Clearly, we have
\begin{equation}
\label{eq84}
\Vert u\circ f- u\circ g\Vert_\infty\leq [u]_{\alpha}\Vert f-g\Vert_\infty^\alpha.
\end{equation}
Now, we will estimate $[u\circ f-u\circ g]_{\beta\gamma}$. Let $A>0$. The value of $A$ will be determined later.

If $d(x,y)<A$, then
\begin{equation}
\label{eq85}
\begin{split}
&
\big|\big((u\circ f)-(u\circ g)\big)(x)-\big((u\circ f)-(u\circ g)\big)(y)\big|\\
&\leq
|(u\circ f)(x)-(u\circ f)(y)|+|(u\circ g)(x)-(u\circ g)(y)|\\
&\leq
\left([u\circ f]_{\alpha\gamma}+[u\circ g]_{\alpha\gamma}\right)d(x,y)^{\alpha\gamma}\\
&\lesssim
[u]_{\alpha}\left([f]_{\gamma}+[g]_{\gamma}\right)^\alpha
d(x,y)^{(\alpha-\beta)\gamma}d(x,y)^{\beta\gamma}\\
&\leq
[u]_{\alpha}\left([f]_{\gamma}+[g]_{\gamma}\right)^\alpha
A^{(\alpha-\beta)\gamma}d(x,y)^{\beta\gamma}.
\end{split}
\end{equation}
If $d(x,y)\geq A$, then
\begin{equation}
\label{eq86}
\begin{split}
&
\big|\big((u\circ f)-(u\circ g)\big)(x)-\big((u\circ f)-(u\circ g)\big)(y)\big|\\
&\leq
2\Vert u\circ f-u\circ g\Vert_\infty\leq 2[u]_{\alpha}\Vert f-g\Vert_\infty^\alpha d(x,y)^{-\beta\gamma}d(x,y)^{\beta\gamma}\\
&\lesssim
[u]_{\alpha}\Vert f-g\Vert_\infty^\alpha A^{-\beta\gamma}d(x,y)^{\beta\gamma}.
\end{split}
\end{equation}
Now we choose $A>0$ so that the right hand sides on \eqref{eq85} and \eqref{eq86} are equal i.e., 
$$
A=\left(\frac{\Vert f-g\Vert_\infty}{[f]_{\gamma}+[g]_{\gamma}}\right)^{\frac{1}{\gamma}}.
$$
This yields
\begin{equation}
\label{eq87}
[u\circ f-u\circ g]_{\beta\gamma}\leq 
[u]_{\alpha} \left([f]_{\gamma}+[g]_{\gamma}\right)^{\beta}\Vert f-g\Vert_\infty^{\alpha-\beta}.
\end{equation}
Now, \eqref{eq84} and \eqref{eq87} imply \eqref{eq83}.
\end{proof}
Taking $m=1$, $u(x)=x$ and $\alpha=1$ yields the following results.
\begin{corollary}
\label{T70}
Let $(X,d)$ be a metric space and $0<\gamma\leq 1$. If $(f_n)\subset C_{\rm b}^{0,\gamma}(X)$ is a bounded sequence (with respect to the norm of $C^{0,\gamma}_{\rm b}(X)$) and $f_n\rightrightarrows f$ uniformly, then $f\in C_{\rm b}^{0,\gamma}(X)$ and for any $0<\beta<\gamma$, $f_n\to f$ in $C_{\rm b}^{0,\beta}(X)$ as $n\to\infty$.
\end{corollary}
\begin{corollary}
\label{T77}
Let $X$ be a metric space and let $\gamma\in (0,1]$. If $(f_n)\subset C_{\rm b}^{0,\gamma}(X)$ is a bounded sequence (with respect to the norm of $C_{\rm b}^{0,\gamma}(X)$) that is pointwise convergent on a dense subset of $X$, then there is $f\in C_{\rm b}^{0,\gamma}(X)$, such that for every compact (or more generally, totally bounded) set $K\subset X$ and every $0<\beta<\gamma$, $f_n\to f$ in $C^{0,\beta}_{\rm b}(K)$.
\end{corollary}
\begin{proof}
Since the family $(f_n)$ is equicontinuous, convergence on a dense subset of $X$ implies convergence at every point of $X$. Denote the limit function by $f$. Clearly, $f\in C_{\rm b}^{0,\gamma}(X)$. If $K\subset X$ is compact (or totaly bounded), then the pointwise convergence on $K$ of the equicontinuous family $(f_n)$ implies the uniform convergence on $K$ and the result follows from \Cref{T70}.
\end{proof}

In the next result we will need the following general version of the Arzel\`a-Ascoli theorem, see
\cite[Theorem~1.4.9]{papa}.
\begin{lemma}[Arzel\`a-Ascoli]
\label{T2}
Let $X$ be a separable metric space and let $(f_n)$ be a bounded and equicontinuous sequence of real valued
functions defined on $X$. Then, there is a subsequence $(f_{n_k})$ that is pointwise convergent on $X$ to a uniformly continuous function and the convergence if uniform on every compact (or more generally, every totally bounded) subset of $X$.
\end{lemma}
\begin{proof}
In a more standard version of the theorem it is assumed that $X$ is compact, but the proof in this more general case is almost the same and similar to the proof of \Cref{T77}. 
First, using a diagonal argument we find a subsequence $(f_{n_k})$ that is convergent on a countable and dense subset of $X$. Since the family is equicontinuous, convergence on a dense subset implies pointwise convergence at every point of $X$. If $K\subset X$ is compact (or totally bounded), then pointwise convergence on $K$ of the equicontinuous family $(f_{n_k})$ implies the uniform convergence of $(f_{n_k})$ on $K$.
\end{proof}
\begin{corollary}
\label{T1}
Let $X$ be a separable metric space and $0<\gamma\leq 1$. If
$(f_n)\subset C_{\rm b}^{0,\gamma}(X)$ is a bounded sequence (in the norm of $C^{0,\gamma}_{\rm b}$), then there is a function $f\in C_{\rm b}^{0,\gamma}(X)$ and a subsequence
$(f_{n_k})$ such that for any compact (or more generally, totally bounded) set $K\subset X$ and any $0<\beta<\gamma$,
$\Vert f_{n_k}-f\Vert_{C^{0,\beta}(K)}\to 0$ as $k\to\infty$.
\end{corollary}
\begin{proof}
It follows from \Cref{T2} that there
is a subsequence $(f_{n_k})$ that converges uniformly to a function $f$ on every compact (totally bounded) set $K\subset X$. 
Clearly, $f\in C_{\rm b}^{0,\gamma}(X)$. Now, \Cref{T70} yields the convergence
$\Vert f_{n_k}-f\Vert_{C^{0,\beta}(K)}\to 0$ as $k\to\infty$ for any $0<\beta<\gamma$.
\end{proof}

If $u\in C^{0,1}$, \Cref{T80} yields continuity of the composition operator \eqref{eq88} for all $\beta<1$. In order to prove continuity with $\beta=1$, i.e., continuity in $C^{0,\gamma}_{\rm b}(X)$, we need stronger assumptions about $u$.

\begin{definition}
\label{d3}
A continuous, non-decreasing, and concave function $\omega:[0,\infty)\to [0,\infty)$, such that $\omega(0)=0$ is a modulus
of continuity of a function $f:X\to\bbbr$ defined on a metric space $(X,d)$ if
\begin{equation}
\label{eq21}
|f(x)-f(y)|\leq\omega(d(x,y))
\quad
\text{for all $x,y\in X$.}
\end{equation}
\end{definition}
\begin{lemma}
\label{T98}
If $\omega_1:[0,\infty)\to [0,\infty)$ is non-decreasing, bounded and $\lim_{t\to 0^+}\omega_1(t)=0$, then there is a continuous bounded concave function $\omega:[0,\infty)\to [0,\infty)$, $\omega(0)=0$, such that $\omega_1(t)\leq \omega(t)$ for all $f\in [0,\infty)$.
\end{lemma}
\begin{proof}
Let
\begin{equation}
\label{eq22}
\omega(t):=\inf\{\alpha t+\beta:\, \text{$\omega_1(s)\leq\alpha s+\beta$ for all $s\in [0,\infty)$}\}.
\end{equation}
It is easy to see that $\omega$ is concave and $\omega\geq\omega_1$, so $\omega$ is nonnegative.
Moreover, concavity implies continuity of $\omega$ on $(0,\infty)$, see
\cite[Theorem~A,~p.4]{RV}.
Since $\alpha\geq 0$ in \eqref{eq22}, $\omega$ is non-decreasing. It remains to show that $\lim_{t\to 0^+}\omega(t)=0$.
For $0<\beta<\sup_{s>0}\omega_1(s)$ define $\alpha=\sup_{s>0}(\omega_1(s)-\beta)/s$. Note that $\alpha$ is positive and finite, because $\omega_1(s)-\beta<0$ for small $s$.
Clearly, $\omega_1(s)\leq\alpha s+\beta$ for all $s>0$.
Hence $\omega(t)\leq\alpha t+\beta$, so $0\leq\limsup_{t\to 0^+}\omega(t)\leq \beta$. Since $\beta>0$ can be arbitrarily small, we conclude that $\lim_{t\to 0^+}\omega(t)=0$.
\end{proof}
The following result is well known.
\begin{proposition}
\label{T29}
Every bounded and uniformly continuous function on a metric space has a modulus of continuity.
\end{proposition}
\begin{proof}
If $f$ is constant, then we take $\omega(t)\equiv 0$. Thus assume that $f$ is not constant.
Let
$\omega_1(t)=\sup\{|f(x)-f(y)|:\, d(x,y)\leq t\}$. Since the function $f$ is uniformly continuous and bounded,
the function $\omega_1:[0,\infty)\to [0,\infty)$ is non-decreasing, bounded, and $\lim_{t\to 0^+}\omega_1(t)=0$.
Clearly, $|f(x)-f(y)|\leq \omega_1(d(x,y))$ for all $x,y\in X$, and
the result follows from \Cref{T98}.
\end{proof}

\begin{lemma}
\label{T30}
Assume that $u\in C^1(\R^m)$ is such that $\Vert u\Vert_\infty+\Vert\nabla u\Vert_\infty<\infty$, and $\nabla u$ has the modulus of continuity $\omega$ i.e.,
$$
|\nabla u(x)-\nabla u(y)|\leq\omega(|x-y|)
\quad\text{for all $x,y\in\R^m$.}
$$
\begin{itemize}
\item[(a)] If $f\in C_{\rm b}^{0,\gamma}(X;\R^m)$, $0<\gamma\leq 1$, then $u\circ f\in C_{\rm b}^{0,\gamma}(X)$ and
$$
\Vert u\circ f\Vert_{C^{0,\gamma}}\leq \Vert u\Vert_\infty+\Vert\nabla u\Vert_\infty [f]_{\gamma}.
$$
\item[(b)] If $f,g\in C^{0,\gamma}_{\rm b}(X;\R^m)$, $0<\gamma\leq 1$, then
$$
\Vert (u\circ f)-(u\circ g)\Vert_\infty\leq \Vert\nabla u\Vert_\infty\Vert f-g\Vert_\infty
$$
and
\begin{equation}
\label{eq23}
[(u\circ g)-(u\circ f)]_{\gamma}\leq\omega(\Vert f-g\Vert_\infty)[f]_{\gamma}+
\Vert\nabla u\Vert_\infty [f-g]_{\gamma}.
\end{equation}
Therefore, the mapping
$$
C^{0,\gamma}_{\rm b}(X;\R^m)\ni f\mapsto u\circ f\in C_{\rm b}^{0,\gamma}(X)
$$
is continuous.
\end{itemize}
\end{lemma}
\begin{remark}
A particularly important case is when 
$\nabla u\in C_{\rm b}^{0,\gamma}$ for some $0<\gamma\leq 1$, i.e. when $\omega(t)=Ct^\gamma$.
\end{remark}
\begin{proof}
(a) The inequality $\Vert u\circ f\Vert_\infty\leq \Vert u\Vert_\infty$ is obvious. To prove the estimate
$[u\circ f]_{\gamma}\leq\Vert\nabla u\Vert_\infty [f]_{\gamma}$ we argue as follows
\begin{align*}
&|(u\circ f)(x)-(u\circ f)(y)|
=
\left|\int_0^1\frac{d}{dt}u(tf(x)+(1-t)f(y))\, dt\right|\\
&\leq 
|f(x)-f(y)|\int_0^1|\nabla u(tf(x)+(1-t)f(y))|\, dt
\leq\Vert\nabla u\Vert_\infty [f]_{\gamma}d(x,y)^\gamma.
\end{align*}
\noindent (b)
The inequality 
$\Vert (u\circ f)-(u\circ g)\Vert_\infty\leq \Vert\nabla u\Vert_\infty\Vert f-g\Vert_\infty$ follows from almost the same argument as the one used in part (a). To prove \eqref{eq23} we use a similar argument, but the estimates are slightly more involved. For simplicity of notation let
$$
\xi_t=tf(x)+(1-t)f(y)
\quad
\text{and}
\quad
\eta_t=tg(x)+(1-t)g(y).
$$
Then
\begin{align*}
&
\big((u\circ f)-(u\circ g)\big)(x)-\big((u\circ f)-(u\circ g)\big)(y)
=
\int_0^1\frac{d}{dt}(u(\xi_t)-u(\eta_t))\, dt\\
&=
\int_0^1\nabla u(\xi_t)(f(x)-f(y))-\nabla u(\eta_t)(g(x)-g(y))\, dt\\
&=
\int_0^1(\nabla u(\xi_t)-\nabla u(\eta_t))(f(x)-f(y))+
\nabla u(\eta_t)((f-g)(x)-(f-g)(y))\, dt
\end{align*}
so
\begin{align*}
&
\big|\big((u\circ f)-(u\circ g)\big)(x)-\big((u\circ f)-(u\circ g)\big)(y)\big|\\
&\leq
\sup_{t\in [0,1]}|\nabla u(\xi_t)-\nabla u(\eta_t)|\, [f]_{\gamma}d(x,y)^\gamma +
\Vert \nabla u\Vert_\infty [f-g]_{\gamma}d(x,y)^\gamma,
\end{align*}
and it remains to observe that
$|\xi_t-\eta_t|\leq \Vert f-g\Vert_\infty$.
\end{proof}
\begin{corollary}
\label{T31}
Let $u\in C^1(\R^m)$ and $0<\gamma\leq 1$. If $f\in C^{0,\gamma}_{\rm b}(X;\R^m)$, then $u\circ f\in C_{\rm b}^{0,\gamma}(X)$.
Moreover, the composition operator $C^{0,\gamma}_{\rm b}(X;\R^m)\ni f\mapsto u\circ f\in C_{\rm b}^{0,\gamma}(X)$ is continuous.
\end{corollary}
\begin{proof}
By the definition of the class $C^{0,\gamma}_{\rm b}(X;\R^m)$, the mapping $f$ is bounded. 
Let $\varphi\in C_c^\infty(\bbbr^m)$ be equal $1$ on the image of $f$, and let $\tilde{u}=\varphi u$. \Cref{T29} shows that $\tilde{u}$ satisfies the assumptions of \Cref{T30} and hence
$u\circ f=\tilde{u}\circ f\in C_{\rm b}^{0,\gamma}(X)$. If $f_i\to f$ in $C^{0,\gamma}_{\rm b}(X;\R^m)$, then images of all mappings $f,f_i$ are contained in one bounded set. Taking $\varphi\in C_c^\infty(\bbbr^m)$ equal $1$ on that set allows us to use \Cref{T30} and conclude that $u\circ f_i=\tilde{u}\circ f_i\to \tilde{u}\circ f=u\circ f$ in $C_{\rm b}^{0,\gamma}$.
\end{proof}

\begin{definition}
\label{d2}
Fix a mollifier kernel $\eta\in C_c^\infty(\bbbb^n)$, $\eta\geq 0$, $\int_{\bbbr^n}\eta\, dx =1$, and for $\eps>0$ define
$\eta_\eps(x)=\eps^{-n}\eta(x/\eps)$. For $f\in L^1_{\rm loc}(\bbbr^n)$ we define the \emph{$\eps$-mollification} of $f$ by
$$
f_\eps(x):= f*\eta_\eps(x)= \eps^{-n} \int_{\bbbr^n} f(z)\eta\left(\frac{x-z}{\eps}\right)\, dz.
$$
\end{definition}
Clearly, $f_\eps\in C^\infty$. If $f\in L^p(\bbbr^n)$, $1\leq p<\infty$, then 
\begin{equation}
\label{eq3}
\Vert f*\eta_\eps\Vert_p\leq \Vert f\Vert_p
\quad
\text{and}
\quad
\text{ $\Vert f_\eps-f\Vert_p\to 0$ as $\eps\to 0$.}
\end{equation}
Since $\int_{\bbbr^n} \partial \eta_\eps/\partial x_i\, dx =0$, we have a useful identity that will be repeatedly used in later sections
\begin{equation}
\label{eq5}
\begin{split}
\frac{\partial f_\eps}{\partial x_i}(x)
&=
f*\frac{\partial\eta_\eps}{\partial x_i}(x)=\left((f-f(x))*\frac{\partial\eta_\eps}{\partial x_i}\right)(x)\\
&=
\eps^{-n-1}\int_{B_\eps}(f(x-z)-f(x))\frac{\partial\eta}{\partial x_i}\left(\frac{z}{\eps}\right)\, dz.
\end{split}
\end{equation}

The next results will discuss approximation of H\"older functions by smooth ones. 
\begin{proposition}
\label{T3}
If $f\in C_{\rm b}^{0,\gamma}(\R^n)$, $0<\gamma\leq 1$, then
$\sup_{\eps>0}\Vert f_\eps\Vert_{C^{0,\gamma}}\leq\Vert f\Vert_{C^{0,\gamma}}$, 
\begin{equation}
\label{eq4}
\lim_{\eps\to 0}\Vert f_\eps\Vert_{C^{0,\gamma}}=\Vert f\Vert_{C^{0,\gamma}}
\quad
\text{and}
\quad
\lim_{\eps\to 0}\Vert f_\eps-f\Vert_{C^{0,\beta}}=0
\end{equation}
for any $0<\beta<\gamma$.
In particular $f$ can be approximated by smooth functions
in the $C_{\rm b}^{0,\beta}$ norm.
\end{proposition}
\begin{remark}
Later we will see that smooth functions are {\em not} dense in $C_{\rm b}^{0,\gamma}(\bbbr^n)$, see \Cref{T4}.
\end{remark}
\begin{proof}
It is easy to see that $\Vert f_\eps\Vert_\infty\leq\Vert f\Vert_\infty$, and  $[f_\eps]_{\gamma}\leq [f]_{\gamma}$, so
$\sup_{\eps>0}\Vert f_\eps\Vert_{C^{0,\gamma}}\leq\Vert f\Vert_{C^{0,\gamma}}$.
It is also easy to see that $f_\eps\rightrightarrows f$ uniformly, and hence $\lim_{\eps\to 0}\Vert f_\eps\Vert_\infty=\Vert f\Vert_\infty$.

For $x\neq y$, we have
$$
\frac{|f(x)-f(y)|}{|x-y|^\gamma}=
\lim_{\eps\to 0} \frac{|f_\eps(x)-f_\eps(y)|}{|x-y|^\gamma}
\leq
\liminf_{\eps\to 0}[f_\eps]_{\gamma} \leq \limsup_{\eps\to 0}[f_\eps]_{\gamma} \leq [f]_{\gamma}.
$$
Taking the supremum over all $x\neq y$, yields $\lim_{\eps\to 0}[f_\eps]_{\gamma}=[f]_{\gamma}$. This and uniform convergence of $f_\eps\rightrightarrows f$ proves the first limit in \eqref{eq4}. Convergence in 
the $C_{\rm b}^{0,\beta}$ norm follows from \Cref{T70}.
\end{proof}

\begin{corollary}
\label{T71}
Let $\cM$ be a compact Riemannian manifold with or without boundary, and let $f\in C_{\rm b}^{0,\gamma}(\cM)$, $0<\gamma\leq 1$. Then there is a sequence $(f_k)\subset C^\infty(\overbar{\cM})$ such that $\sup_k\Vert f_k\Vert_{C^{0,\gamma}}\lesssim\Vert f\Vert_{C^{0,\gamma}}$ and $f_k\rightrightarrows f$ uniformly. Therefore, $f_k\to f$ in $C_{\rm b}^{0,\beta}(\cM)$ for any \mbox{$0<\beta<\gamma$}.
\end{corollary}
\begin{proof}
Assume first that $\partial\cM=\varnothing$.
Using the Whitney embedding theorem we may assume that $\cM$ is smoothly embedded into $\R^{2n+1}$, where $n=\dim\cM$. The new Riemannian metric on $\cM$ inherited from the embedding into $\R^{2n+1}$, will be comparable to the original one due to compactness of $\cM$.

According to the tubular neighborhood theorem, see for instance \cite[\textsection 2.12.3]{Simon96}, there is an open set $U\subset\R^{2n+1}$ containing $\cM$, such that the nearest point projection $\pi:U\to\cM$ is uniquely defined, smooth and Lipschitz. 

Let $\phi\in C_c^\infty(U)$ satisfy $\phi=1$ on $\cM$. Then we can define bounded extension and trace operators
$$
C_{\rm b}^{0,\gamma}(\cM)\stackrel{E}{\longrightarrow} C_c^{0,\gamma}(U)\stackrel{T}{\longrightarrow}C_{\rm b}^{0,\gamma}(\cM)
$$
by
$$
Ef(x)=\phi(x)(f\circ\pi)(x)
\quad
\text{and}
\quad
Tf=f|_{\cM},
$$
so $(T\circ E)f=f$ for $f\in C_{\rm b}^{0,\gamma}(\cM)$.
This allows us to conclude the result as a direct consequence of \Cref{T3}.
Indeed, we approximate $f$ by $T\circ(Ef)_\eps$.

If $\partial\cM\neq\varnothing$, we can glue two copies of $\cM$ along the boundary to and obtain a closed manifold $\widetilde{\cM}$, extend $C_{\rm b}^{0,\gamma}(\cM)$ functions to $\widetilde{\cM}$ and apply the above argument to $\widetilde{\cM}$. 
\end{proof}
\begin{remark}
We will use a similar argument in the proof of \Cref{T50}. 
\end{remark}
\begin{remark}
\label{R8}
If in addition the manifold $\cM$ is oriented,
\Cref{T71} easily extends to the case of approximation of forms $\omega\in\Omega^kC_{\rm b}^{0,\gamma}(\cM)$. We just need to replace the extension operator by 
$E\omega(x):=\phi(x)(\pi^*\omega)(x)\in\Omega^kC_c^{0,\gamma}(U)$ and approximate $E\omega$ by a standard mollification in $\bbbr^{2n+1}$. We leave details to the reader.
\end{remark}
\begin{proposition}
\label{T4}
Let $0<\gamma<1$. Then a function  $f\in C_{\rm b}^{0,\gamma}(\bbbr^n)$ 
can be approximated by a sequence smooth functions $f_k\in C^\infty$ 
so that for every compact set $K\subset\bbbr^n$, $\Vert f_k-f\Vert_{C^{0,\gamma}(K)}\to 0$ as $k\to\infty$, if and only if
$f\in C_{\rm b}^{0,\gamma+}(\bbbr^n)$.
\end{proposition}
\begin{remark}
If $f\in C_{\rm b}^{0,1/2}(\bbbr)$ equals
$|x|^{1/2}$ in a neighborhood of $0$, then $f\in C_{\rm b}^{0,1/2}\setminus C_{\rm b}^{0,1/2+}$ and hence it cannot be approximated by smooth functions on compact sets in the $C_{\rm b}^{0,1/2}$ norm.
\end{remark}
\begin{remark}
The result is false for $\gamma=1$, because $C^{0,1+}$ consists of constant functions only (partial derivatives are equal zero). In the proof of Proposition~\ref{T4} we use the assumption $\gamma<1$ only in one place when we consider the exponent $1/(1-\gamma)$. 
\end{remark}
\begin{proof}
Suppose that $f\in C_{\rm b}^{0,\gamma}(\bbbr^n)$ can be approximated by a sequence of smooth functions
$f_k\in C^\infty$. We need to show that $f\in C_{\rm b}^{0,\gamma+}(\bbbr^n)$.
Let $B_R=B(0,R)$. Let $\eps>0$ be given. Then for a sufficiently large $k$,
$$
|(f_k-f)(y)-(f_k-f)(x)|\leq\frac{\eps}{2}|x-y|^\gamma
\quad
\text{for all $x,y\in B_R$.}
$$
Let $M=\sup_{B_R}|\nabla f_k|$. The mean value theorem yields
$$
|f(y)-f(x)|\leq\frac{\eps}{2}|x-y|^\gamma +|f_k(y)-f_k(x)|\leq
\left(\frac{\eps}{2}+M|x-y|^{1-\gamma}\right)|x-y|^\gamma
$$
so
$$
|f(y)-f(x)|\leq\eps |x-y|^\gamma
\quad
\text{for all $x,y\in B_R$ satisfying $|x-y|<(\eps/2M)^{1/(1-\gamma)}$.}
$$
This proves that $f\in C_{\rm b}^{0,\gamma+}$.

Suppose now that
$f\in C_{\rm b}^{0,\gamma+}(\bbbr^n)$. We will show that the approximation by mollification $f_t$ has the desired property i.e., 
for every ball $B_R$, $\Vert f_t-f\Vert_{C^{0,\gamma}(B_R)}\to 0$ as $t\to 0$. 
Since $f_t\rightrightarrows f$ uniformly, it remains to estimate the constant in the H\"older estimate of the difference $f_t-f$.
Let $\eps>0$ be given. It follows from the definition of $C^{0,\gamma+}$
that there is $R>\tau>0$ such that if $x,y\in B_{2R}$, $|x-y|<\tau$, then $|f(x)-f(y)|\leq \frac{1}{2}\eps|x-y|^\gamma$. Hence,  
$|f_t(x)-f_t(y)|\leq\frac{1}{2}\eps|x-y|^\gamma$ 
if $0<t<R$ and
$x,y\in B_R$ satisfy $|x-y|<\tau$. This easily follows from the definition of $f_t$, because $f_t(x)$ 
is a weighted average of $f$ on the ball $B(x,t)\subset B_{2R}$. 
Therefore,
\begin{equation}
\label{eq18}
|(f_t-f)(x)-(f_t-f)(y)|\leq \eps |x-y|^\gamma 
\quad
\text{for all $x,y\in B_R$ satisfying $|x-y|<\tau$.}
\end{equation}
Let $0<\delta<R$ be such that $\Vert f_t-f\Vert_\infty<\eps\tau^\gamma/2$ for all $0<t<\delta$. 
If $x,y\in B_R$, $|x-y|\geq\tau$, then 
\begin{equation}
\label{eq19}
|(f_t-f)(x)-(f_t-f)(y)|\leq 2\Vert f_t-f\Vert_\infty<\eps\tau^\gamma\leq\eps |x-y|^\gamma.
\end{equation}
We proved in \eqref{eq18} and \eqref{eq19} that if $0<t<\delta$, then
$$
|(f_t-f)(x)-(f_t-f)(y)|\leq \eps |x-y|^\gamma 
\quad
\text{for all $x,y\in B_R$.}
$$
This proves that $f_t\to f$ in $C_{\rm b}^{0,\gamma}(B_R)$. The proof is complete.
\end{proof}
\begin{corollary}
\label{T4.5}
Let $\cM$ be a connected compact Riemannian manifold with or without boundary and let $0<\beta<\gamma\leq 1$. Then
\begin{enumerate}
\item The embedding $C_{\rm b}^{0,\gamma}(\cM)\Subset C_{\rm b}^{0,\beta}(\cM)$ is compact.
\item Every function $f\in C_{\rm b}^{0,\gamma}(\cM)$ can be approximated by $C^\infty(\cM)$ functions in the $C_{\rm b}^{0,\beta}$ norm.
\item If $0<\gamma<1$, then $f\in C_{\rm b}^{0,\gamma}(\cM)$ can be approximated by $C^{\infty}(\cM)$ functions in the $C_{\rm b}^{0,\gamma}$ norm
if and only if $f\in C_{\rm b}^{0,\gamma+}(\cM)$.
\end{enumerate}
\end{corollary}
\begin{proof}
(1) follows from \Cref{T1}. (2) follows from \Cref{T71} and (3) is a consequence of \Cref{T4} and the proof of \Cref{T71}.
\end{proof}

\subsection{Sobolev spaces}
\label{Sobolev}
There are many excellent textbooks on Sobolev spaces and we assume that the reader is familiar with the material of Chapter~5 in \cite{EG}.

Let $\Omega\subset\bbbr^n$ be open, $1\leq p\leq \infty$, and $m\in\bbbn$. The Sobolev space $W^{m,p}(\Omega)$ is the space of all functions $f\in L^p(\Omega)$ with distributional derivatives satisfying $D^\alpha f\in L^p(\Omega)$, for all $|\alpha|\leq m$. $W^{m,p}(\Omega)$ equipped with the norm
$\Vert f\Vert_{m,p}:=\sum_{k=0}^m\Vert D^k f\Vert_p$
is a Banach space.

We will be mainly interested in the first order Sobolev space $W^{1,p}(\Omega)$, $\Vert f\Vert_{1,p}=\Vert f\Vert_p+\Vert\nabla f\Vert_p$, as well as in the fractional Sobolev spaces defined below.
If $1\leq p<\infty$, $W^{1,p}_0(\Omega)$ is the closure of $C_c^\infty(\Omega)$ in the Sobolev norm.

By analogy with the case of H\"older continuous functions we will use notation
$$
[f]_{1,p}=\Vert Df\Vert_p.
$$
If $\cM$ is a smooth compact manifold, with or without boundary, we define the Sobolev space $W^{1,p}(\cM)$ as the set of all functions $f$ whose representation in a good coordinate system (Section~\ref{GCS}) belongs to $W^{1,p}$. The space $W^{1,p}(\cM)$ is equipped with a norm defined by  a formula similar to \eqref{eq38}.
It is well known that $W^{1,\infty}(\bbbr^n)$ coincides with the space on bounded Lipschitz functions and clearly the same is true for $W^{1,\infty}(\cM)$.

The next result will be useful.
\begin{lemma}
\label{T61}
If $f,g\in W^{1,n}(\Omega;\R^n)$ and $f-g\in W_0^{1,n}(\Omega;\R^n)$, where $\Omega\subset\R^n$ is open, then
\begin{equation}
\label{eq70}
\int_\Omega df_1\wedge\ldots\wedge df_n=\int_\Omega dg_1\wedge\ldots\wedge dg_n.
\end{equation}
\end{lemma}
\begin{proof}
Assume first that $f,g\in C^\infty(\Omega;\bbbr^n)$ and $f-g\in C_c^\infty(\Omega;\bbbr^n)$.
Since
\begin{equation}
\label{eq71}
df_1\wedge\ldots\wedge df_n-dg_1\wedge\ldots\wedge dg_n=
\sum_{i=1}^n dg_1\wedge\ldots\wedge dg_{i-1}\wedge d(f_i-g_i)\wedge df_{i+1}\wedge\ldots\wedge df_n,
\end{equation}
integration and Stokes' theorem yield
$$
\int_\Omega df_1\wedge\ldots\wedge df_n-dg_1\wedge\ldots\wedge dg_n=
\sum_{i=1}^n(-1)^{i-1}\int_{\partial\Omega}
(f_i-g_i)dg_1\wedge\ldots\wedge dg_{i-1}\wedge df_{i+1}\wedge\ldots\wedge df_n=0,
$$
because $f_i-g_i=0$ in a neighborhood of $\partial\Omega$.

Assume now that  $f,g\in W^{1,n}(\Omega;\R^n)$ and $f-g\in W_0^{1,n}(\Omega;\R^n)$. Let
$g^\eps\in C^\infty(\Omega;\bbbr^n)$ and $\phi^\eps\in C_c^\infty(\Omega;\bbbr^n)$ be such that $g^\eps\to g$ and $\phi^\eps\to f-g$ in $W^{1,n}$ as $\eps\to 0^+$. Then $f^\eps:=g^\eps+\phi^\eps\to f$. 

The result is true for $f^\eps$ and $g^\eps$, and it suffices to prove that the integrals \eqref{eq70} for $f^\eps$ and $g^\eps$ converge to the corresponding integrals for $f$ and $g$.

Applying \eqref{eq71} to $f$ and to $f^\eps$ in place of $g$, and using Hadamard's inequality
$$
|dh_1\wedge\ldots\wedge dh_n|\leq |\nabla h_1|\ldots|\nabla h_n|
$$
(the volume of a parallelepiped in bounded by the product of lengths of the edges), we obtain
\begin{equation}
\label{eq78}
|df_1\wedge\ldots\wedge df_n-df^\eps_1\wedge\ldots\wedge df^\eps_n|\leq
\sum_{i=1}^n |\nabla f^\eps_1|\ldots|\nabla f^\eps_{i-1}|\, |\nabla(f_i-f^\eps_i)|\,|\nabla f_{i+1}|\ldots |\nabla f_n|.
\end{equation}
Since the integral of the right hand side converges to zero by H\"older's inequality, it follows that
\begin{equation}
\label{eq73}
\int_{\bbbr^n} df^\eps_1\wedge\ldots\wedge df^\eps_n\to
\int_{\bbbr^n} df_1\wedge\ldots\wedge df_n
\quad
\text{as $\eps\to 0^+$.}
\end{equation}
A similar argument applies to $g$ and $g^\eps$.
\end{proof}

In the remaining part of this section we will discuss fractional Sobolev spaces. While there are many excellent textbooks discussing the theory of Sobolev spaces,
it is much harder to find a good introduction to the theory of fractional Sobolev spaces. 
We refer the reader to Chapters~3 and~4 of \cite{demengel} and Chapters~14 and~15 of \cite{leoni} for proofs of results stated below.
In Remark~\ref{R7} we will provide detailed references.
\begin{definition}
Let $\cM$ be an $n$-dimensional Riemannan manifold without boundary (in particular $\cM$ can be a domain in $\bbbr^n$).
Let $0<s<1$, $1<p<\infty$. The fractional Sobolev space $W^{s,p}(\cM)$ is the space of all functions $f\in L^p(\cM)$ such that 
$\Vert f\Vert_{s,p}= \Vert f\Vert_p+[f]_{s,p}<\infty$, where
$$
[f]_{s,p}=\left(\int_{\cM}\int_{\cM}\frac{|f(x)-f(y)|^p}{d(x,y)^{n+sp}}\, dx\, dy\right)^{1/p}.
$$
The subspace $W^{s,p}_0(\cM)\subset W^{s,p}(\cM)$ is the closure of $C_c^\infty(\cM)$ in the $W^{s,p}$ norm.
\end{definition}
A word of warning. In the literature there are different notions of ``fractional Sobolev spaces''. Another popular choice are the so-called Bessel-potential spaces. The space that we are considering here is often referred to as Sobolev-Slobodeckij space, Besov space, or the trace space (cf. Lemma~\ref{T19}). The seminorm $[f]_{s,p}$ is often called the Gagliardo-norm.
\begin{lemma}
\label{T58}
Let $\cM=\R^n$ or let $\cM$ be a smooth closed $n$-dimensional Riemannian manifold. Let $0<s<1$, $1<p<\infty$. Then $W^{s,p}(\cM)$ is a Banach space and smooth functions $C^\infty(\cM)\cap W^{s,p}(\cM)$ form a dense subset of $W^{s,p}(\cM)$.  
\end{lemma}
\begin{remark}
\label{R11}
With the same technique one can show that smooth forms $\Omega^k(\cM)$ are dense in $\Omega^kW^{s,p}(\cM)$, $0<s<1$, $1<p<\infty$, $0\leq k\leq n$, see also Remark~\ref{R8}.
\end{remark}
\begin{lemma}
\label{T19}
Let $\cM=\R^n_+$ or
let $\cM$ be a smooth compact $n$-dimensional Riemannian manifolds with boundary, $n\geq 2$. 
Let $1<p<\infty$.
Then, there is a unique bounded linear trace operator
$$
\operatorname{Tr}:W^{1,p}(\cM)\to W^{1-\frac{1}{p},p}(\partial \cM)
$$
such that $\operatorname{Tr}f=f|_{\partial\cM}$ if $f\in C^\infty(\overbar{\cM})$. 
Moreover there is a bounded linear extension operator
\begin{equation}
\label{eq39}
\ext_{\partial\cM}:W^{1-\frac{1}{p},p}(\partial \cM)\to W^{1,p}(\cM)
\end{equation}
such that $\operatorname{Tr}\circ\ext_{\partial\cM}=\operatorname{Id}$ on $W^{1-\frac{1}{p},p}(\partial \cM)$.
\end{lemma}
In particular, if $\Omega$ is a bounded domain with smooth boundary, then Sobolev functions in $W^{1,p}(\Omega)$, $1<p<\infty$, have a well defined trace on the boundary that belongs to the fractional Sobolev space $W^{1-\frac{1}{p},p}(\partial\Omega)$. 

\begin{remark}
\label{R7}
For a proof of Lemma~\ref{T58}, see Proposition~4.27 and Proposition~4.24 in \cite{demengel}, and for a proof of Lemma~\ref{T19}, see Proposition~3.31 and Theorem~3.9 in \cite{demengel}. 
The reader may also find proofs of both Lemmata~\ref{T58} and~\ref{T19} in \cite{leoni}. 

While the proofs presented in \cite{demengel} are mostly in the Euclidean case, the extension to the case of manifolds is straightforward (use partition of unity and local coordinate systems). In fact \cite[Proposition~3.31]{demengel} deals with the case of traces on the boundary of a smooth domain and the argument needed for the case of manifolds is similar.

While we will not prove the above results, in the next section we will show the construction of the extension operator. 
\end{remark}

\begin{corollary}
\label{T60}
Let $\cM$ be a smooth closed $n$-dimensional Riemannian manifold and $1<p<\infty$, $1-\frac{1}{p}<\gamma\leq 1$. Then the embedding $C_{\rm b}^{0,\gamma}(\cM)\Subset W^{1-\frac{1}{p},p}(\cM)$ is compact.
\end{corollary}
\begin{proof}
Let  $1-\frac{1}{p}<\gamma'<\gamma$. Boundedness of the embedding 
$C_{\rm b}^{0,\gamma'}(\cM)\subset W^{1-\frac{1}{p},p}(\cM)$ easily follows from the definition of the fractional Sobolev space so the result follows from Theorem~\ref{T4.5}: 
$C_{\rm b}^{0,\gamma}(\cM)\Subset C_{\rm b}^{0,\gamma'}(\cM)\subset W^{1-\frac{1}{p},p}(\cM)$.
\end{proof}

\subsection{The Gagliardo extension}
\label{S5}
In this section we will explain the construction of an extension operator from Lemma~\ref{T19} in detail.
While we will not prove Lemma~\ref{T19}, we will 
prove a result weaker than \eqref{eq39}, the existence of a bounded extension 
\begin{equation}
\label{eq72}
\ext_{\partial\cM}:C_{\rm b}^{0,\gamma}(\partial\cM)\to W^{1,p}(\cM),
\quad
1-\frac{1}{p}<\gamma\leq 1.
\end{equation}
see Corollary~\ref{T57}. 
\begin{definition}[Gagliardo extension]
\label{D4}
Let $\eta$ be a mollifier kernel as in Definition~\ref{d2} and let $\psi\in C^\infty([0,\infty))$ be a cut-off function
satisfying $0\leq\psi\leq 1$, $\psi(t)=1$ for $0\leq t\leq 1/2$ and $\psi(t)=0$ for $t\geq 3/4$.
For a function $f\in L^1(\bbbr^n)$  we define the extension operator $\ext_R f$ by
\begin{equation}
\label{eq36}
\ext_R f(x,t)=\psi(t/R)(f*\eta_t)(x)
\quad
\text{for $x\in \bbbr^n$ and $t>0$.}
\end{equation}
\end{definition}
Note that $\ext_Rf\in C^\infty(\bbbr^n\times(0,\infty))$. Moreover if $f\in L^1(\R^n)$ equals zero outside the ball $B_R$,
$\ext_Rf$ has compact support inside $B_{2R}\times[0,R)$. 
That is, it does not necessarily vanish at the bottom part of the boundary $B_R\times \{0\}$, but it equals zero in a neighborhood of
the remaining part of the boundary $S_{2R}\times[0,R]\cup B_{2R}\times\{R\}$.

\begin{lemma}
\label{T59}
Let $1<p<\infty$ and let $n\geq 2$, then
$$
\Vert \ext_R f\Vert_{W^{1,p}(\R^n_+)}\lesssim \Vert f\Vert_{W^{1-\frac{1}{p},p}(\R^{n-1})}
\quad
\text{for any $R>0$ and $f\in W^{1-\frac{1}{p},p}(\R^{n-1})$.}
$$
\end{lemma}
For a proof see \cite[Proposition~3.62]{demengel} and also \cite{leoni}.
Using this result, partition of unity and local coordinates one can construct the extension \eqref{eq39}. Below, we will provide a detailed construction of such an extension, and we will prove boundedness of \eqref{eq72},
see Corollary~\ref{T57}. A similar proof based on Lemma~\ref{T59} leads to \eqref{eq39}. We leave details to the reader.

We will need the following fact.
\begin{lemma}
\label{T81}
Let $\Omega\subset\R^n$ be a bounded Lipschitz domain, $\gamma\in (0,1]$, and $m>0$. If $F\in C^1(\Omega)$ satisfies 
$$
|\nabla F(x)|\leq m\dist (x,\partial\Omega)^{\gamma-1}
\quad
\text{for all }
x\in\Omega,
$$
then $F\in C_{\rm b}^{0,\gamma}(\Omega)$ and $[F]_{\gamma}\leq C(\Omega,\gamma)m$.
\end{lemma}
\begin{proof}
If $x,y\in\Omega$, then we can find a curve $\gamma:[0,L]\to\Omega$ parametrized by arc-length, such that
$$
L\lesssim |x-y|
\qquad
\text{and}
\qquad
\dist(\gamma(t),\partial\Omega)\gtrsim\min\{t,L-t\}.
$$
Here the implied constants depend on $\Omega$ only.
This is a well known and easy to prove fact. It is also a consequence of the fact that bounded Lipschitz domains are uniform, see e.g., \cite[p. 225]{HM} and references therein. Now we have
\[
\begin{split}
|F(x)-F(y)|
&\leq
\int_0^L|\nabla F(\gamma(t))|\, dt\leq
m\int_0^L \dist(\gamma(t),\partial\Omega)^{\gamma-1}\, dt\\
&\lesssim
m\int_0^L\min \{ t,L-t\}^{\gamma-1}\, dt=
2m\int_0^{L/2}t^{\gamma-1}\, dt=
\frac{2^{1-\gamma}m}{\gamma} L^\gamma\lesssim m|x-y|^\gamma.
\end{split}
\]
The proof is complete.
\end{proof}

The next two results provide basic local estimates for the extension operator in the H\"older case.
\begin{lemma}
\label{T82}
Let $\gamma\in (0,1]$. Assume that a function $f\in C_{\rm b}^{0,\gamma}(\R^n)$ vanishes outside $B_R$. Then
$\ext_Rf\in C^\infty(\R^{n+1}_+)$ satisfies
$$
|D(\ext_Rf)(x,t)|\lesssim [f]_{\gamma}\, t^{\gamma-1}.
$$
\end{lemma}
\begin{proof}
Since $\int_{\bbbr^n}\partial\eta_t/\partial x_i\, dx =0$ we can subtract the constant $f(x)$ from 
the function $f$ in the following estimate
\begin{equation*}
\begin{split}
\left|\frac{\partial}{\partial x_i}(\ext_R f(x,t))\right|
&=
\left|\psi(t/R)(f-f(x))*\frac{\partial\eta_t}{\partial x_i}(x)\right|
\lesssim
t^{-n-1}\int_{B_t}|f(x-z)-f(x)|\, dz\\
&\lesssim
t^{-n-1}t^n[f]_{\gamma}\, t^\gamma 
= 
t^{\gamma-1}[f]_{\gamma}.
\end{split}
\end{equation*}
The estimate for the partial derivative with respect to $t$ is similar, but slightly more involved. We will need three simple observations
\begin{itemize}
\item When taking the derivative $\partial/\partial t$, we can subtract the constant number $f(x)$;
\item Since $f$ vanishes on the boundary of $B_R$, $\Vert f\Vert_\infty\leq [f]_{\gamma}\,R^\gamma$;
\item We have $|\partial\eta_t/\partial t|\lesssim t^{-n-1}\chi_{B_t}$.
\end{itemize}
Applying the product rule to $\psi(t/R)\cdot(f*\eta_t)(x)$ and using the above observations we have
\begin{align}
\label{eq10}
&\left|\frac{\partial}{\partial t}(\ext_R f(x,t))\right|
\lesssim
R^{-1}|f*\eta_t|(x)+\Big|\frac{\partial}{\partial t}(\underbrace{f*\eta_t-f(x)}_{(f-f(x))*\eta_t})\Big|\\
\nonumber
&\lesssim
R^{-1}\Vert f\Vert_\infty +\Big(|f-f(x)|*\left|\frac{\partial \eta_t}{\partial t}\right|\Big)(x)
\lesssim
[f]_{\gamma}\, R^{\gamma-1}+ t^{-n-1}\int_{B_t}|f(x-z)-f(x)|\, dz\\
\nonumber
&\lesssim 
[f]_{\gamma}(R^{\gamma-1}+t^{\gamma-1}).
\end{align}
Since $\ext_Rf(x,t)=0$ when $t\geq R$, the above estimate improves to
$$
\left|\frac{\partial}{\partial t}(\ext_R f(x,t))\right|\lesssim [f]_{\gamma}\, t^{\gamma-1}.
$$
We proved that the full derivative in $x$ and $t$ satisfies
$|D(\ext_R f(x,t))|\lesssim [f]_{\gamma}\,t^{\gamma-1}$.
\end{proof}

\begin{lemma}
\label{T5}
Let $1<p<\infty$, $1-\frac{1}{p}<\gamma\leq 1$ and $R>0$. Assume that a function $f\in C_{\rm b}^{0,\gamma}(\R^n)$ vanishes outside $B_R$. Then
\begin{align}
\label{eq6}
[\ext_R f]_{C^{0,\gamma}(\R^{n+1}_+)}
&\lesssim [f]_{\gamma},\\
\label{eq7}
[\ext_R f]_{W^{1,p}(\bbbr^{n+1}_+)}
&\lesssim [f]_{\gamma},
\end{align}
where the implied constants depend in particular on $R$ and other given constants.
\end{lemma}
\begin{proof}
Assume that $f\in C_{\rm b}^{0,\gamma}(\R^n)$, and $f=0$ in $\R^n\setminus B_R$.
Note that $\ext_R f$ vanishes outside $B_{2R}\times [0,R]$.

In particular, if $(x,t)\in\Omega$, where $\Omega=B_{5R}\times[0,5R]$ and 
$\ext_R f(x,y)\neq 0$, then $\dist\{(x,t),\partial\Omega\}=t$.
Therefore, Lemma~\ref{T82} shows that the function $F=\ext_R f$ satisfies assumptions of Lemma~\ref{T81} with 
$m\approx [f]_{\gamma}$, and \eqref{eq6} follows directly from Lemma~\ref{T81}.

Now we will prove \eqref{eq7}.
Since $\gamma>1-1/p$, we have
$(\gamma-1)p>-1$, so the function $t^{(\gamma-1)p}$ is integrable near zero and
Lemma~\ref{T82} yields
\begin{equation*}
[\ext_Rf]^p_{W^{1,p}(\R^{n+1}_+)}
=
[\ext_Rf]^p_{W^{1,p}(B_{2R}\times [0,R])}
\lesssim 
\int_0^R\int_{B_{2R}}[f]_{\gamma}^p\, t^{(\gamma-1)p}\, dx\, dt
\lesssim
[f]^p_{\gamma}\,R^{n+1+(\gamma-1)p}.
\end{equation*}
The proof is complete.
\end{proof}
The next result is a consequence of Lemma~\ref{T5}.
\begin{proposition}
\label{T6}
Let $\cM$ be a compact $n$-dimensional Riemannian manifold without boundary and let $1<p<\infty$.
Then, for $1\geq \gamma> 1-\frac{1}{p}$ there is a bounded linear extension operator
$$
\ext_{\cM}: C_{\rm b}^{0,\gamma}(\cM)\to W^{1,p}\cap C_{\rm b}^{0,\gamma}(\cM\times [0,1])
$$
such that the extension of $f\in C_{\rm b}^{0,\gamma}(\cM)$, is $C^\infty$ smooth in $\cM\times (0,1)$, and equal zero in a neighborhood of $\cM\times\{ 1\}$. 
\end{proposition}
\begin{remark}
\label{R5}
The proof of Proposition~\ref{T6} is based on Lemma~\ref{T5}. If instead, we use Lemma~\ref{T59} we obtain the existence of a bounded extension $\ext_{\cM}:W^{1-\frac{1}{p},p}(\cM)\to W^{1,p}(\cM\times [0,1])$. The construction of the extension operator remains the same.
\end{remark}
The following result is a direct consequence of Proposition~\ref{T6} and the collar neighborhood theorem according to which a neighborhood of $\partial\cM$ is diffeomorphic to $\partial\cM\times [0,1]$.
\begin{corollary}
\label{T57}
Let $\cM$ be a compact $n$-dimensional Riemannian manifold with boundary, $1<p<\infty$ and 
$1\geq \gamma> 1-\frac{1}{p}$. Then, there is a bounded linear extension operator
$$
\ext_{\partial\cM}: C_{\rm b}^{0,\gamma}(\partial\cM)\to W^{1,p}\cap C_{\rm b}^{0,\gamma}(\cM).
$$
\end{corollary}
\begin{proof}[Proof of Proposition~\ref{T6}]
Denote the Riemannian distance on $\cM$ by $d$.
Let $\varphi_\alpha:B_3\to\cM$, $\alpha=1,2,\ldots,N$ be parametrizations of $\cM$ such that $\bigcup_{\alpha=1}^N\varphi_\alpha(B_1)=\cM$ (here $B_3=B(0,3)$ and $B_1=B(0,1)$ are balls in $\R^n$). Let $\{\psi_\alpha\}_{\alpha=1}^N$ be a partition of unity subordinated to the covering $\{\varphi_\alpha(B_1)\}_{\alpha=1}^N$ of $\cM$.

Given $f\in C^{0,\gamma}(\cM)$, write $f = \sum_{\alpha}{\psi_\alpha f}$.  Note that each of the functions $(\psi_\alpha f)\circ\varphi_\alpha$ has compact support
in $B_1$. 
Note also that
$$
\varphi_\alpha(B_2)\times [0,\infty)\ni (x,t)\stackrel{\Xi_\alpha}{\longmapsto} (\varphi_\alpha^{-1}(x),t)\in B_2\times [0,\infty)
$$
is a diffeomorphism with bounded derivatives. 
Recall that the extension operator $\ext_1$ is defined in $\R^n$ by \eqref{eq36} with $R=1$.
Then define the extension
$$
\ext_{\cM} f(x,t) = \sum_{\alpha:\, x\in\varphi_\alpha(B_2)}(\ext_1((\psi_\alpha f) \circ \varphi_\alpha))  
\circ\Xi_\alpha(x,t)
:= \sum_\alpha \ext_\alpha f,
$$
where
\begin{equation}
\label{eq37}
\ext_\alpha f(x,t)=
\begin{cases}
(\ext_1((\psi_\alpha f) \circ \varphi_\alpha))\circ\Xi_\alpha(x,t)& \text{if $x\in\varphi_\alpha(B_2)$,}\\
0 & \text{if $x\in \cM\setminus\varphi_\alpha(B_2)$.}
\end{cases}
\end{equation}
The diffeomorphism $\Xi_\alpha$ maps $\partial(\varphi_\alpha(B_2))\times [0,\infty)$ onto $\partial B_2\times [0,\infty)$ and $\ext_1((\psi_\alpha f)\circ\varphi_\alpha)$ equals zero in a neighborhood of 
$\partial B_2\times [0,\infty)$ so the expression in the first line of \eqref{eq37} equals zero for $x$ near $\partial(\varphi_\alpha(B_2))$ and so the second line of \eqref{eq23} defines a smooth extension beyond 
$\varphi_\alpha(B_2)$.

Evidently, $\ext_{\cM} f$ is continuous on $\cM\times [0,1]$, agrees with $f$ on $\cM\times \{ 0 \}$, vanishes in a neighborhood of $\cM\times \{ 1\}$ and is
smooth in $\cM\times (0,1)$.

Note that $(\psi_\alpha f)\circ\varphi_\alpha\in C_{\rm b}^{0,\gamma}(B_1)$ has compact support in $B_1$ and
\begin{equation}
\label{eq12}
[(\psi_\alpha f)\circ\varphi_\alpha]_{\gamma}\lesssim \Vert f\Vert_{C^{0,\gamma}}.
\end{equation}
Indeed, for $u,v\in\varphi_\alpha(B_1)$ we have
\begin{equation*}
\begin{split}
|\psi_\alpha(u)f(u)-\psi_\alpha(v)f(v)|
&\leq
\Vert f\Vert_\infty |\psi_\alpha(u)-\psi_\alpha(v)|+\Vert \psi_\alpha\Vert_\infty|f(u)-f(v)|\\
&\lesssim\
(\Vert f\Vert_\infty+[f]_{\gamma})d(u,v)^\gamma
\end{split}
\end{equation*}
and \eqref{eq12} follows, because $d(\varphi_\alpha(x),\varphi_\alpha(y))$ is comparable to $|x-y|$
for $x,y\in B_1$.
Hence $\ext_1((\psi_\alpha f)\circ\varphi_\alpha)$ has support inside $B_2\times[0,1)$ and Lemma~\ref{T5} in concert with \eqref{eq12} yields
\begin{equation}
\label{eq123}
[\ext_1((\psi_\alpha f)\circ\varphi_\alpha)]_{C^{0,\gamma}(\bbbr^{n+1}_+)}+
[\ext_1((\psi_\alpha f)\circ\varphi_\alpha)]_{W^{1,p}(\bbbr^{n+1}_+)}\lesssim \Vert f\Vert_{C^{0,\gamma}}.
\end{equation}
Since $\Xi_\alpha$  is a diffeomorphism with bounded derivatives, it follows that
$$
[\ext_\cM f]_{C^{0,\gamma}(\cM\times [0,1])} +
[\ext_\cM f]_{W^{1,p}(\cM\times [0,1])}
\lesssim \Vert f\Vert_{C^{0,\gamma}}.
$$
We still need to estimate $\Vert\ext_\cM f\Vert_\infty$ and $\Vert\ext_\cM f\Vert_p$.
For $0\leq t\leq 1$ we have
\[
\begin{split}
|(\ext_\cM f)(x,t)|
&\leq 
|(\ext_\cM f)(x,t)-(\ext_\cM f)(x,0)|+|f(x)|\\
&\leq
[\ext_\cM f]_\gamma+\Vert f\Vert_\infty\lesssim \Vert f\Vert_{C^{0,\gamma}}.
\end{split}
\]
Therefore, $\Vert\ext_\cM f\Vert_p\lesssim\Vert\ext_\cM f\Vert_\infty\lesssim\Vert f\Vert_{C^{0,\gamma}}$.
This estimate along with \eqref{eq123} completes the proof.
\end{proof}
Using similar arguments we can also construct an extension operator for forms.
\begin{corollary}
\label{T7}
Let $\cM$ be an oriented, compact, $n$-dimensional Riemannian manifold without boundary and let $1<p<\infty$ and $0\leq k\leq n$. Then there is a bounded linear extension operator
$$
\ext_{\cM}: \Omega^k W^{1-\frac{1}{p},p}(\cM)\to \Omega^k W^{1,p}(\cM\times [0,1])
$$
such that the extension $\ext_{\cM}\kappa$ of any $\kappa\in W^{1-\frac{1}{p},p}$, is $C^\infty$ smooth in $\cM\times (0,1)$, and equal zero in a neighborhood of $\cM\times\{ 1\}$.
Moreover, if
$1\geq \gamma> 1-\frac{1}{p}$ and $0\leq k\leq n$, then
$$
\ext_{\cM}: \Omega^k C_{\rm b}^{0,\gamma}(\cM)\to \Omega^k (W^{1,p}\cap C_{\rm b}^{0,\gamma})((\cM\times [0,1])).
$$
\end{corollary}
\begin{proof}
Let $\{\varphi_\alpha\}$, $\{\psi_\alpha\}$ and $\ext_\alpha$
be as in the proof of Proposition~\ref{T6} (see Remark~\ref{R5}). Then 
$\varphi_\alpha^{-1}=(x_1^\alpha,\ldots,x_n^\alpha)$ is a coordinate system in $\varphi_\alpha(B_2)\subset\cM$. This coordinate system extends to the coordinate system 
$$ 
\Xi_\alpha(x,t)=(\varphi_\alpha^{-1}(x),t)=(x_1^\alpha,\ldots,x_n^\alpha,t)
\qquad
\text{on} 
\quad 
\varphi_\alpha(B_2)\times [0,\infty). 
$$
In particular, $dx_I^\alpha$ can be regarded as a differential form on $\varphi_\alpha(B_2)\subset\cM$ as well as a form on  $\varphi_\alpha(B_2)\times [0,\infty)\subset\cM\times [0,\infty)$.

If $\kappa$ is a $k$-form on $\cM$, then
$\kappa=\sum_{\alpha}\psi_\alpha\kappa$. Since $\psi_\alpha\kappa$ has compact support in $\varphi_\alpha(B_1)$, it has a unique representation as
$$
\psi_\alpha\kappa = \sum_{|I|=k} \kappa_I^\alpha\, dx_I^\alpha,
$$
where the coefficients $\kappa_I^\alpha$ are functions with the same regularity as $\kappa$, compactly supported in  $\varphi_\alpha(B_1)$. Now we define the extension as
$$
\ext_{\cM}\kappa=\sum_{\alpha}\sum_{|I|=k}\ext_\alpha(\kappa_I^\alpha)\, dx_I^\alpha,
$$
where $dx_I^\alpha$ are regarded as differential forms on $\cM\times [0,\infty)$.
Since $\ext_\alpha(\kappa_I^\alpha)(x,t)$ vanishes if $x\not\in \varphi_\alpha(B_2)$,
$\ext_\alpha(\kappa_I^\alpha)\, dx_I^\alpha$ is a globally defined form  on $\cM\times [0,\infty)$, that is equal zero if $x\not\in\varphi_\alpha(B_2)$.
\end{proof}

\section{Distributional Jacobian}
\label{DJ}
If functions $f_1,\ldots,f_n$, belong to the fractional Sobolev space $W^{s,p}(\R^n)$, $0<s<1$, $1<p<\infty$, the pointwise Jacobian $df_1\wedge\ldots \wedge df_n$ is not well defined, but we may try to define the {\em distributional Jacobian} acting on $\varphi\in C_c^\infty(\R^n)$ as the limit
$$
\langle\varphi,df_1\wedge\ldots\wedge df_n\rangle=
\int_{\R^n}\varphi\, df_1\wedge\ldots\wedge df_n:=
\lim_{\eps\to 0} \int_{R^n}\varphi\, df_1^\eps\wedge\ldots\wedge df_n^\eps,
$$
where $C^\infty(\R^n)\ni f_i^\eps\to f_i$ in $W^{s,p}(\R^n)$. 
However, it is not obvious under what assumptions about $s$ and $p$, the limit exists and is independent of the choice of a smooth approximation.

Although, we denoted the limit as the integral $\int_{\R^n} \varphi\, df_1\wedge\ldots\wedge df_n$ (assuming it exists and is independent of the choice of a smooth approximation), we need to understand that it is only a formal notation for the distribution acting on the test function $\varphi$, the distribution that is defined as a limit of the true integrals; notation 
$\langle\varphi,df_1\wedge\ldots\wedge df_n\rangle$ is however, a more standard way to denote evaluation of a distribution on a test function. 

In this section we will explore the above idea
in a more general context of differential forms. Our results will generalize some of the results of Brezis and Nguyen  \cite[Theorem~3 and Corollary~1]{brezisn1} and of Z\"ust \cite[Theorem~3.2]{zust1}, \cite{zust2}. See also \cite{Conti}, and \cite{SY1999,LS19} for corresponding results in the abstract harmonic analysis.

\begin{theorem}
\label{T32}
Let $\cM$ be a smooth oriented and closed $n$-dimensional Riemannian manifold. Let
$$
\lambda\in \Omega^\ell W^{1-\frac{1}{q},q}(\cM)
\quad
\text{and}
\quad
\gamma_i\in \Omega^{\ell_i-1}W^{1-\frac{1}{p_i},p_i}\left(\cM\right),
\
i=1,2,\ldots,k,
$$
where $\ell\geq 0$, $\ell_i\geq 1$, $1<q,p_i<\infty$,  $\ell+\ell_1+\ldots+\ell_k=n$, and
\begin{equation}
\label{eq24}
\frac{1}{q}+\frac{1}{p_1}+\ldots+\frac{1}{p_{k}}=1.
\end{equation}
Assume we have approximations
$$
\Omega^\ell W^{1-\frac{1}{q},q}\left(\cM\right)\ni\lambda^\eps\to\lambda\ \text{in $W^{1-\frac{1}{q},q}$}
\quad
\text{and}
\quad
\Omega^{\ell_i-1}\lip\left(\cM\right)\ni \gamma_i^\eps\to \gamma_i\ \text{in $W^{1-\frac{1}{p_i},p_i}$}
$$
as $\eps\to 0$. Then the limit
\begin{equation}
\label{eq31}
\langle \lambda,d\gamma_1\wedge\ldots\wedge d\gamma_k\rangle=
\int_{\cM}\lambda\wedge d\gamma_1\wedge\ldots\wedge d\gamma_{k}:=
\lim_{\eps\to 0}\int_{\cM} \lambda^\eps\wedge d\gamma_1^\eps\wedge\ldots\wedge d\gamma_k^\eps
\end{equation}
exists and it does not depend on a choice of approximations $\lambda^\eps$, $\gamma_i^\eps$.
\end{theorem}
\begin{remark}
\label{R10}
Existence of a smooth (and hence Lipschitz) approximation $\gamma_i^\eps\to\gamma_i$ follows from Remark~\ref{R11}. Existence of an approximation $\lambda^\eps\to\lambda$ is obvious, because we can take $\lambda^\eps=\lambda$. Note that in this case we have
$$
\int_{\cM}\lambda\wedge d\gamma_1\wedge\ldots\wedge d\gamma_{k}=
\lim_{\eps\to 0}\int_{\cM} \lambda\wedge d\gamma_1^\eps\wedge\ldots\wedge d\gamma_k^\eps.
$$
\end{remark}
\begin{proof}
Let $0\leq m\leq n$ be an integer and let
$$
\ext_{\cM}:\Omega^m W^{1-\frac{1}{p},p}\left(\cM\right)\to \Omega^m W^{1,p}\big(\cM\times [0,1]\big)
$$
be the extension operator from Corollary~\ref{T7}. If $m=0$, then this is an extension operator on functions.
Recall that $\ext_{\cM}\kappa$ is $C^\infty$ on $\cM\times (0,1)$, it vanishes in a neighborhood of $\cM\times\{1\}$ and it is smooth up to $\cM\times\{0\}$, if $\kappa\in C^\infty$. Moreover, $\ext_{\cM}\kappa$ is defined by the same formula for all $p$, but of course, it depends on $m$. 

Let $\lambda^{\eps,\delta}\in \Omega^\ell(\cM)$, $\lambda^{\eps,\delta}\to\lambda^\eps$ in $W^{1-\frac{1}{q},q}$ as $\delta\to 0$.
Since $\lip(\cM)\subset W^{1,p_i}(\cM)$, we can find a smooth approximation 
$\gamma_i^{\eps,\delta}\in\Omega^{\ell_i-1}(\cM)$ , $\gamma_i^{\eps,\delta}\to\gamma_i^\eps$ in $W^{1,p_i}$. Since $W^{1,p_i}(\cM)\subset W^{1-\frac{1}{p_i},p_i}(\cM)$ (see e.g. \cite[Proposition~28]{GoldsteinH}), we also have that $\gamma_i^{\eps,\delta}\to \gamma_i^\eps$ in $W^{1-\frac{1}{p_i}, p_i}(\cM)$.

Let $\Lambda=\ext_{\cM}\lambda\in W^{1,q}$, $\Gamma_i=\ext_{\cM} \gamma_i\in W^{1,p_i}$ 
and let $\Lambda^\eps$, $\Lambda^{\eps,\delta}$, $\Gamma_i^\eps$ and $\Gamma_i^{\eps,\delta}$ be the extensions of $\lambda^\eps, \lambda^{\eps,\delta}$, $\gamma_i^\eps$ and $\gamma_i^{\eps,\delta}$ respectively. Since the extension operator is continuous, we have
$$
\Lambda^{\eps,\delta}
\xrightarrow{\delta\to 0}\Lambda^\eps,
\quad
\Lambda^\eps\xrightarrow{\eps\to 0}\Lambda \ \text{in $W^{1,q}$}
\quad
\text{and}
\quad
\Gamma_i^{\eps,\delta}\xrightarrow{\eps\to 0}\Gamma_i^\eps,
\quad
\Gamma_i^\eps\xrightarrow{\eps\to 0}\Gamma_i 
\ \text{in $W^{1,p_i}$}.
$$
Observe that
\[
\begin{split}
\int_{\cM} \lambda^{\eps,\delta}\wedge d\gamma_1^{\eps,\delta}\wedge\ldots\wedge d\gamma_k^{\eps,\delta}
&=
\int_{\partial(\cM\times [0,1])}\Lambda^{\eps,\delta}\wedge d\Gamma_1^{\eps,\delta}\wedge\ldots\wedge d\Gamma_{k}^{\eps,\delta}\\
&= 
\int_{\cM\times [0,1]}d\Lambda^{\eps,\delta}\wedge d\Gamma_1^{\eps,\delta}\wedge\ldots\wedge d\Gamma_{k}^{\eps,\delta}.
\end{split}
\]
The first equality follows from the fact that the extension of smooth forms is smooth up to the boundary and that it
vanishes on $\cM\times\{1\}$, which is a part of the boundary of $\cM\times[0,1]$.
The second equality is the Stokes theorem.

Letting $\delta\to 0$ we obtain
\begin{equation}
\label{eq124}
\int_{\cM} \lambda^{\eps}\wedge d\gamma_1^\eps\wedge\ldots\wedge d\gamma^\eps_{k}
=
\int_{\cM\times [0,1]}d\Lambda^{\eps}\wedge d\Gamma_1^\eps\wedge\ldots\wedge d\Gamma_{k}^\eps.
\end{equation}
Indeed, 
$$
\lambda^{\eps,\delta}\to\lambda^\eps \text{ in } L^q,
\quad
d\gamma_i^{\eps,\delta}\to d\gamma_i^\eps \text{ in } L^{p_i},
\quad
d\Lambda^{\eps,\delta}\to d\Lambda^\eps \text{ in } L^q,
\quad
d\Gamma_i^{\eps,\delta}\to d\Gamma_i^\eps \text{ in } L^{p_i},
$$
so \eqref{eq124} is a consequence of \eqref{eq24} and H\"older's inequality, just like in \eqref{eq73}.
Now the same argument shows that
\begin{equation}
\label{eq125}
\int_{\cM} \lambda^{\eps}\wedge d\gamma_1^\eps\wedge\ldots\wedge d\gamma^\eps_{k}
=
\int_{\cM\times [0,1]}d\Lambda^{\eps}\wedge d\Gamma_1^\eps\wedge\ldots\wedge d\Gamma_{k}^\eps
\xrightarrow{\eps\to 0}
\int_{\cM\times [0,1]}d\Lambda \wedge d\Gamma_1 \wedge\ldots\wedge d\Gamma_{k}.
\end{equation}
We proved existence of the limit \eqref{eq31}. It is independent of the approximation, because the right hand side of \eqref{eq125} does not depend on the approximation.
\end{proof}

In the case in which $p_i=p$ for all $i=1,2,\ldots,k$, we have $p>k$ and $q=\frac{p}{p-k}$. In that case we get
\begin{theorem}
\label{T33}
Let $\cM$ be a smooth oriented and closed $n$-dimensional Riemannian manifold. Let
$$
\lambda,\tilde{\lambda}\in \Omega^\ell W^{1-\frac{1}{q},q}\left(\cM\right)
\quad
\text{and}
\quad
\gamma_i,\tilde{\gamma}_i\in \Omega^{\ell_i-1}W^{1-\frac{1}{p},p}\left(\cM\right),\ i=1,2,\ldots,k
$$
where $\ell\geq 0$, $\ell_i\geq 1$,  $\ell+\ell_1+\ldots+\ell_k=n$, 
$1\leq k<p<\infty$,  $q=\frac{p}{p-k}$. Then
the assumptions of Theorem~\ref{T32} are satisfied with $p_i=p$. Moreover,
\begin{align*}
&\left|\,\int_{\cM}\lambda\wedge d\gamma_1\wedge\ldots\wedge d\gamma_k-\int_{\cM}\tilde{\lambda}\wedge d\tilde{\gamma}_1\wedge\ldots\wedge d\tilde{\gamma}_k\right|\\
&\lesssim    
\left(\Vert\lambda-\tilde{\lambda}\Vert_{W^{1-\frac{1}{q},q}}\Vert \gamma\Vert_{W^{1-\frac{1}{p},p}}
+\Vert\tilde{\lambda}\Vert_{W^{1-\frac{1}{q},q}}\Vert \gamma-\tilde{\gamma}\Vert_{W^{1-\frac{1}{p},p}}\right)
\left(\Vert \gamma\Vert_{W^{1-\frac{1}{p},p}}+\Vert \tilde{\gamma}\Vert_{W^{1-\frac{1}{p},p}}\right)^{k-1},
\end{align*}
where $\gamma=(\gamma_1,\ldots,\gamma_k)$ and $\tilde{\gamma}=(\tilde{\gamma}_1,\ldots,\tilde{\gamma}_k)$.
\end{theorem}
\begin{proof}
Let $\lambda^\eps,\tilde{\lambda}^\eps,\gamma_i^\eps,\tilde{\gamma}_i^\eps$ be smooth approximations and let
$\Lambda^\eps,\tilde{\Lambda}^\eps,\Gamma_i^\eps,\tilde{\Gamma}_i^\eps$ be their extensions to $\cM\times [0,1]$ as in the proof of Theorem~\ref{T32}. Also let $\Gamma^\eps=(\Gamma_1^\eps,\ldots,\Gamma_k^\eps)$ and 
$\tilde{\Gamma}^\eps=(\tilde{\Gamma}_1^\eps,\ldots,\tilde{\Gamma}_k^\eps)$. We have
\begin{align*}
d\Lambda^\eps\wedge d\Gamma_1^\eps\wedge\ldots\wedge d\Gamma_k^\eps
&-
d\tilde{\Lambda}^\eps\wedge d\tilde{\Gamma}_1^\eps\wedge\ldots\wedge
d\tilde{\Gamma}_k^\eps    
=
d(\Lambda^\eps-\tilde{\Lambda}^\eps)\wedge d\Gamma_1^\eps\wedge\ldots\wedge d\Gamma_k^\eps\\
&+
\sum_{j=1}^k d\tilde{\Lambda}^\eps\wedge d\tilde{\Gamma}_1^\eps\wedge\ldots\wedge d\tilde{\Gamma}_{j-1}^\eps\wedge d(\Gamma_j^\eps-\tilde{\Gamma}_j^\eps)\wedge d{\Gamma}_{j+1}^\eps\wedge\ldots\wedge d{\Gamma}_k^\eps.
\end{align*}
Therefore,
\begin{equation}
\label{eq92}
\begin{split}
&|d\Lambda^\eps\wedge d\Gamma_1^\eps\wedge\ldots\wedge d\Gamma_k^\eps-
d\tilde{\Lambda}^\eps\wedge d\tilde{\Gamma}_1^\eps\wedge\ldots\wedge
d\tilde{\Gamma}_k^\eps|\\
&\lesssim 
|D(\Lambda^\eps-\tilde{\Lambda}^\eps)|\, |D\Gamma^\eps|^k+
|D\tilde{\Lambda}^\eps|\, |D(\Gamma^\eps-\tilde{\Gamma}^\eps)|(|D \Gamma^\eps|+|D\tilde{\Gamma}^\eps|)^{k-1}.
\end{split}
\end{equation}
Integration, Stokes' theorem, and H\"older's inequality yield
\begin{align*}
&
\left|\,\int_{\cM}\lambda^\eps\wedge d\gamma^\eps_1\wedge\ldots\wedge d\gamma^\eps_k-
\tilde{\lambda}^\eps\wedge d\tilde{\gamma}^\eps_1\wedge\ldots\wedge d\tilde{\gamma}^\eps_k\right|\\
&=
\left|\, \int_{\cM\times [0,1]}d\Lambda^\eps\wedge d\Gamma_1^\eps\wedge\ldots\wedge d\Gamma_k^\eps-
d\tilde{\Lambda}^\eps\wedge d\tilde{\Gamma}_1^\eps\wedge\ldots\wedge
d\tilde{\Gamma}_k^\eps\right|\\
&\lesssim
\left(\,\int_{\cM\times [0,1]}|D(\Lambda^\eps-\tilde{\Lambda}^\eps)|^q\right)^{1/q}
\left(\,\int_{\cM\times [0,1]}|D\Gamma^\eps|^p\right)^{k/p}\\
&+
\left(\,\int_{\cM\times [0,1]}|D\tilde{\Lambda}^\eps|^q\right)^{1/q}
\left(\,\int_{\cM\times [0,1]}|D(\Gamma^\eps-\tilde{\Gamma}^\eps)|^p\right)^{1/p}
\left(\,\int_{\cM\times [0,1]}|D\Gamma^\eps|^p+|D\tilde{\Gamma}^\eps|^p\right)^{(k-1)/p}\\
&\lesssim
\Vert\lambda^\eps-\tilde{\lambda}^\eps\Vert_{W^{1-\frac{1}{q},q}}
\Vert \gamma^\eps\Vert_{W^{1-\frac{1}{p},p}}^k
+
\Vert\tilde{\lambda}^\eps\Vert_{W^{1-\frac{1}{q},q}}
\Vert \gamma^\eps-\tilde{\gamma}^\eps\Vert_{W^{1-\frac{1}{p},p}}
\left(\Vert \gamma^\eps\Vert_{W^{1-\frac{1}{p},p}}+\Vert \tilde{\gamma}^\eps\Vert_{W^{1-\frac{1}{p},p}}\right)^{k-1}
\end{align*}
and the result follows upon passing to the limit as $\eps\to 0$.
\end{proof}
\begin{corollary}
\label{T51}
Let $\cM$ be a smooth oriented and closed $n$-dimensional Riemannian manifold. Let
$$
\lambda\in \Omega^k W^{1-\frac{1}{q},q}\left(\cM\right)
\quad
\text{and}
\quad
\gamma\in \Omega^\ell W^{1-\frac{1}{p},p}\left(\cM\right),
$$
where $k,\ell\geq 0$, $k+\ell=n-1$, $1<p,q<\infty$ and $p^{-1}+q^{-1}=1$. Then
$$
\int_{\cM}d\lambda\wedge\gamma=(-1)^{k+1}\int_{\cM}\lambda\wedge d\gamma.
$$
\end{corollary}
\begin{proof}
Under the given assumptions the generalized integrals, defined as limits of integrals of smooth approximations, exist. Since
$$
0=\int_{\cM}d(\lambda^\eps\wedge\gamma^\eps)=\int_{\cM} d\lambda^\eps\wedge\gamma^\eps+
(-1)^k\int_{\cM}\lambda^\eps\wedge d\gamma^\eps,
$$
the result follows upon passing to the limit as $\eps\to 0$.
\end{proof}
Several results stated below are straightforward consequences of Theorems~\ref{T32} and~\ref{T33}. However, we state them explicitly to provide `ready to use' tools.  
\begin{corollary}
\label{T11}
Let $\cM$ be a smooth oriented and closed $n$-dimensional Riemannian manifold. Let
$$
f\in W^{1-\frac{1}{q},q}(\cM),\ g_i\in W^{1-\frac{1}{p_i},p_i}(\cM),
\quad
1<q,p_i<\infty,\ i=1,2,\ldots,n,
$$
where
\begin{equation}
\label{eq201}
\frac{1}{q}+\frac{1}{p_1}+\ldots+\frac{1}{p_n}=1.
\end{equation}
If
$$
\text{$W^{1-\frac{1}{q},q}(\cM)\ni f^\eps\to f$ in $W^{1-\frac{1}{q},q}$}
\quad
\text{and}
\quad
\text{$\lip(\cM)\ni g_i^\eps\to g_i$ in $W^{1-\frac{1}{p_i},p_i}$}
$$
as $\eps\to 0$, then the limit
$$
\langle f,dg_1\wedge\ldots\wedge dg_n\rangle= \int_{\cM} f dg_1\wedge\ldots\wedge dg_n:=\lim_{\eps\to 0} \int_{\cM}f^\eps\, dg_1^\eps\wedge\ldots\wedge dg_n^\eps
$$
exists and it does not depend on the choice of  approximations $f^\eps$, $g_i^\eps$.
\end{corollary}
\begin{remark}
In this particular instance of Theorem~\ref{T32}, $f$ and $g_i$ are functions, so $dg_1\wedge\ldots\wedge dg_n$ is a generalized distributional Jacobian of a mapping $g=(g_1,\ldots,g_n):\cM\to\R^n$ acting on a function $f$.
\end{remark}
\begin{corollary}
\label{T13}
Let $\cM$ be a smooth oriented and closed $n$-dimensional Riemannian manifold. Let
$$
f,\tilde{f}\in W^{1-\frac{1}{q},q}(\cM),\ g,\tilde{g}\in W^{1-\frac{1}{p},p}(\cM)^n,
\quad
n<p<\infty,\ q=\frac{p}{p-n}.
$$
Then assumptions of Corollary~\ref{T11} are satisfied with $p_i=p$ and the distributional Jacobains 
$dg_1\wedge\ldots\wedge dg_n$ and $d\tilde{g}_1\wedge\ldots\wedge d\tilde{g}_n$ exist. Moreover 
\begin{align*}
&\left|\int_{\cM} fdg_1\wedge\ldots\wedge dg_n-\int_{\cM}\tilde{f}\tilde{g}_1\wedge\ldots\wedge d\tilde{g}_n\right|\\
&\lesssim
\left(\Vert f-\tilde{f}\Vert_{W^{1-\frac{1}{q},q}}\Vert g\Vert_{W^{1-\frac{1}{p},p}}
+\Vert\tilde{f}\Vert_{W^{1-\frac{1}{q},q}}\Vert g-\tilde{g}\Vert_{W^{1-\frac{1}{p},p}}\right)
\left(\Vert g\Vert_{W^{1-\frac{1}{p},p}}+\Vert \tilde{g}\Vert_{W^{1-\frac{1}{p},p}}\right)^{n-1}.
\end{align*}
\end{corollary}

Since $C_{\rm b}^{0,\alpha}\subset W^{1-\frac{1}{p},p}$ for $1-\frac{1}{p}<\alpha\leq 1$ (by Corollary~\ref{T60}), we can replace the fractional Sobolev spaces in 
the above results by H\"older spaces. However, since smooth mappings are not dense in the H\"older space, we have to relax the approximation condition. This gives the following result.
\begin{theorem}
\label{T34}
Let $\cM$ be a smooth oriented and compact $n$-dimensional Riemannian manifold with or without boundary. Let
$$
\lambda\in \Omega^\ell C_{\rm b}^{0,\alpha}\left(\cM\right)
\quad
\text{and}
\quad
\gamma_i\in \Omega^{\ell_i-1}C_{\rm b}^{0,\beta_i}\left(\cM\right),
\quad
i=1,2,\ldots,k,
$$
where 
$$
\ell+\ell_1+\ldots+\ell_k=n,\
\ell\geq 0,\ \ell_i\geq 1, 
\qquad
\alpha+\beta_1+\ldots+\beta_k>k,\ \alpha,\beta_i\in (0,1].
$$
Assume that $0<\alpha'\leq\alpha$ and $0<\beta_i'\leq\beta_i$ are such that
$\alpha'+\beta_1'+\ldots+\beta_k'>k$ and that
we have approximations
$$
\Omega^\ell C_{\rm b}^{0,\alpha'}\left(\cM\right)\ni\lambda^\eps\to\lambda\ \text{in $C_{\rm b}^{0,\alpha'}$}
\quad
\text{and}
\quad
\Omega^{\ell_i-1}\lip(\cM)\ni \gamma_i^\eps\to \gamma_i\ \text{in $C_{\rm b}^{0,\beta_i'}$}
$$
as $\eps\to 0$. Then the limit
\begin{equation}
\label{eq13}
\langle \lambda,d\gamma_1\wedge\ldots\wedge d\gamma_k\rangle=
\int_{\cM}\lambda\wedge d\gamma_1\wedge\ldots\wedge d\gamma_{k}:=
\lim_{\eps\to 0}\int_{\cM} \lambda^\eps\wedge d\gamma_1^\eps\wedge\ldots\wedge d\gamma_k^\eps
\end{equation}
exists and it does not depend on a choice of approximations $\lambda^\eps$, $\gamma_i^\eps$.
\end{theorem}
\begin{remark}
If families $(\lambda^\eps)\subset C_{\rm b}^{0,\alpha}$ and $(\gamma_i^\eps)\subset \lip$, $i=1,2,\ldots,n$ are bounded in $C_{\rm b}^{0,\alpha}$ and $C_{\rm b}^{0,\beta_i}$ respectively, and
$$
\lambda^\eps\rightrightarrows\lambda,
\quad
\gamma_i^\eps\rightrightarrows\gamma_i, \ i=1,2,\ldots,n,
\quad
\text{uniformly as $\eps\to 0$,}
$$
then according to Corollary~\ref{T70} the assumptions of Theorem~\ref{T34} are satisfied and hence \eqref{eq13} holds true. Existence of such a family $(\lambda^\eps)$ is obvious (take $\lambda^\eps=\lambda$) and the existence of families $(\gamma_i^\eps)$ is guaranteed by Corollary~\ref{T71} and Remark~\ref{R8}.
\end{remark}
\begin{remark}
In Theorem~\ref{T32} we assumed that $\cM$ has no boundary, but in Theorem~\ref{T34} we allow $\cM$ to have boundary. This is because the assumptions in Theorem~\ref{T34} are stronger: if forms are H\"older continuous on $\cM$, then their restrictions to the boundary are H\"older continuous too and we can apply the extension operator to the restrictions of the forms to the boundary.
\end{remark}
\begin{proof}
Under the given assumptions we can find $1<q,p_i<\infty$ such that
$$
\alpha'>1-\frac{1}{q},
\quad
\beta_i'>1-\frac{1}{p_i},
\quad
\frac{1}{q}+\frac{1}{p_1}+\ldots+\frac{1}{p_k}=1.
$$
Assume first that $\partial\cM=\varnothing$.
Since by Corollary~\ref{T60},
$C_{\rm b}^{0,\alpha'}\subset W^{1-\frac{1}{q},q}$, $C_{\rm b}^{0,\beta_i'}\subset W^{1-\frac{1}{p_i},p_i}$, the result follows from Theorem~\ref{T32}.

Assume now that $\partial\cM\neq\varnothing$. The proof in this case is elementary yet delicate.

Using the collar neighborhood theorem, we can 
construct a smooth manifold $\widetilde{\cM}$ without boundary, by gluing two copies of $\cM$ along the boundary. 
A neighborhood of the boundary of the `twin manifold' $\widetilde{\cM}\setminus\cM$, is diffeomorphic to $\partial\cM\times[0,1]$, where $\partial\cM\times\{0\}$ corresponds to $\partial\cM$ and $\partial\cM\times\{ 1\}$ is in the interior of the twin manifold. By writing $\partial\cM\times[0,1]$ in the proof below, we will always refer to this neighborhood.

If $(x,U)$, $x(U)=\bbbb^n_+$, is a good coordinate system on $\cM$ near a boundary point in $\partial\cM$, then the coefficients of the forms $\lambda$, $\lambda^\eps$, $\gamma_i$ and $\gamma_i^\eps$  expressed in this coordinate system are H\"older and Lipschitz continuous on
$$
\bbbb^{n-1}=\{x\in\bbbb^n_+:\, x_n=0\}.
$$
They are H\"older continuous with exponents $\alpha'$ and $\beta_i'$ and coefficients of $\gamma_i^\eps$ are Lipschitz continuous. This includes coefficients of components of the forms that involve $dx_n$. 

An obvious modification of the proof of Corollary~\ref{T7} that takes into account extension of coefficients involving $dx_n$, gives us an extension of the forms $\lambda$, $\lambda^\eps$ and $\gamma_i$, $\gamma_i^\eps$ to the forms (denoted by the same symbols) such that
$$
\lambda,\lambda^\eps \in\Omega^\ell C_{\rm b}^{0,\alpha'}(\widetilde{\cM}),
\qquad
\gamma_i,\in \Omega^{\ell_i-1}C_{\rm b}^{0,\beta_i'}(\widetilde{\cM}),
\quad
\gamma_i^\eps\in \Omega^{\ell_i-1}\lip(\widetilde{\cM}),
$$
and
$$
\lambda,\lambda^\eps \in\Omega^\ell W^{1,q}(\widetilde{\cM}\setminus\cM),
\qquad
\gamma_i,\gamma_i^\eps\in \Omega^{\ell_i-1}W^{1,p_i}(\widetilde{\cM}\setminus\cM).
$$
While the proof of Corollary~\ref{T7} extends the forms to $\partial\cM\times[0,1]$ only, the forms vanish near $\partial\cM\times\{1\}$, so we assume that the forms $\lambda$, $\lambda^\eps$, $\gamma_i$, $\gamma_i^\eps$
are equal zero beyond $\partial\cM\times\{1\}$ in the twin manifold $\widetilde{\cM}\setminus\cM$.

Note that if $\gamma_i^\eps$ is $C^\infty$ on $\overbar{\cM}$, then the  extension of $\gamma_i^\eps$ to $\widetilde{\cM}$ is Lipschitz continuous only. While it is $C^\infty$ in $\widetilde{\cM}\setminus\partial\cM$, smoothness at the points of $\partial\cM$ is not guaranteed -- extending smooth functions from a half-space to the whole space is not that easy (see e.g. \cite{Seeley}).

Since the extension operator is bounded, it follows that
\begin{equation} 
\label{eq91}
\lambda^\eps\to\lambda
\text{ in } W^{1,q}
\quad
\text{and}
\quad
\gamma_i^\eps\to \gamma_i
\text{ in } W^{1,p_i}
\quad
\text{ on }
\widetilde{\cM}\setminus \cM.
\end{equation}

Since the extension operator is bounded in the H\"older norms $C_{\rm b}^{0,\alpha'}$ and $C_{\rm b}^{0,\beta_i'}$, convergence in the H\"older norms on $\cM$ implies that
$$
\lambda^\eps\to\lambda
\text{ in } C_{\rm b}^{0,\alpha'}
\quad
\text{and}
\quad
\gamma_i^\eps\to \gamma_i
\text{ in } C_{\rm b}^{0,\beta_i'}
\quad
\text{ on }
\widetilde{\cM}.
$$
and hence
$$
\lambda^\eps\to\lambda
\text{ in } W^{1-\frac{1}{q},q}
\quad
\text{and}
\quad
\lip\ni\gamma_i^\eps\to \gamma_i
\text{ in } W^{1-\frac{1}{p_i},p_i}
\quad
\text{ on }
\widetilde{\cM}.
$$
Therefore, Theorem~\ref{T32} yields that the limit
\begin{equation}
\label{eq89}
\lim_{\eps\to 0} 
\int_{\widetilde{\cM}}\lambda^\eps\wedge d\gamma^\eps_1\wedge\cdots\wedge d\gamma^\eps_k
\end{equation}
exists and is independent of the choice of the original approximations on $\cM$.
Note that the integral
$$
\int_{\widetilde{\cM}\setminus\cM}\lambda\wedge d\gamma_1\wedge\cdots\wedge d\gamma_k
$$
is defined in the classical sense, because the forms belong to suitable Sobolev spaces.
Thus, to prove existence of the limit \eqref{eq13} and its independence of approximations, it suffices to show that
\begin{equation}
\label{eq90}
\int_{\widetilde{\cM}\setminus\cM}\lambda\wedge d\gamma_1\wedge\cdots\wedge d\gamma_k=
\lim_{\eps\to 0}
\int_{\widetilde{\cM}\setminus\cM}\lambda^\eps\wedge d\gamma^\eps_1\wedge\cdots\wedge d\gamma^\eps_k,
\end{equation}
because then, we will have that
$$
\lim_{\eps\to 0}
\int_{\cM}\lambda^\eps\wedge d\gamma^\eps_1\wedge\cdots\wedge d\gamma^\eps_k=
\lim_{\eps\to 0} 
\int_{\widetilde{\cM}}\lambda^\eps\wedge d\gamma^\eps_1\wedge\cdots\wedge d\gamma^\eps_k-
\int_{\widetilde{\cM}\setminus\cM}\lambda\wedge d\gamma_1\wedge\cdots\wedge d\gamma_k.
$$
The proof of \eqref{eq90} uses \eqref{eq91} and is in fact, 
very similar to that of \eqref{eq73}. We leave details to the reader; they may find some additional details in the last part of the proof of Corollary~\ref{T35}, where we estimated differences of integrals over $\widetilde{\cM}\setminus\cM$.
\end{proof}
\begin{corollary}
\label{T35}
Let $\cM$ be a smooth oriented and compact $n$-dimensional Riemannian manifold with or without boundary. Let
$$
\lambda,\tilde{\lambda}\in \Omega^\ell C_{\rm b}^{0,\alpha}\left(\cM\right)
\quad
\text{and}
\quad
\gamma_i,\tilde{\gamma}_i\in \Omega^{\ell_i-1}C_{\rm b}^{0,\beta}\left(\cM\right),
$$
where 
$$
\ell+\ell_1+\ldots+\ell_k=n,\
\ell\geq 0,\ \ell_i\geq 1, 
\qquad
\alpha+k\beta>k,\ \alpha,\beta\in (0,1].
$$
Then the assumptions of Theorem~\ref{T34} are satisfied with $\beta_i=\beta$.
Moreover, 
\begin{align*}
&\left|\,\int_{\cM}\lambda\wedge d\gamma_1\wedge\ldots\wedge d\gamma_k-\int_{\cM}\tilde{\lambda}\wedge d\tilde{\gamma}_1\wedge\ldots\wedge d\tilde{\gamma}_k\right|\\
&\lesssim    
\left(\Vert\lambda-\tilde{\lambda}\Vert_{C^{0,\alpha}}\Vert \gamma\Vert_{C^{0,\beta}}
+\Vert\tilde{\lambda}\Vert_{C^{0,\alpha}}\Vert \gamma-\tilde{\gamma}\Vert_{C^{0,\beta}}\right)
\left(\Vert \gamma\Vert_{C^{0,\beta}}+\Vert \tilde{\gamma}\Vert_{C^{0,\beta}}\right)^{k-1},
\end{align*}
where $\gamma=(\gamma_1,\ldots,\gamma_k)$ and $\tilde{\gamma}=(\tilde{\gamma}_1,\ldots,\tilde{\gamma}_k)$.
\end{corollary}
\begin{proof}
We will use notation from the proof of Theorem~\ref{T34}. Assume that
$$
\alpha\geq\alpha'>1-\frac{1}{q},
\qquad
\beta\geq\beta'>1-\frac{1}{p},
\qquad
\frac{1}{q}+\frac{k}{p}=1,
$$
and assume that the generalized integrals are defined through approximations in $C_{\rm b}^{0,\alpha'}$ and $C_{\rm b}^{0,\beta'}$. Therefore, we can assume that the forms $\gamma$ and $\tilde{\gamma}$ are Lipschitz continuous.

The triangle inequality
$$
\Bigg|\int_{\cM}\heartsuit-\int_{\cM}\diamondsuit\,\Bigg|\leq 
\Bigg|\int_{\widetilde{\cM}}\heartsuit-\int_{\widetilde{\cM}}\diamondsuit\,\Bigg|+
\Bigg|\int_{{\widetilde{\cM}}\setminus\cM}\heartsuit-\int_{{\widetilde{\cM}}\setminus\cM}\diamondsuit\,\Bigg|
$$
allows us to estimate differences of the integrals over $\widetilde{\cM}$ and $\widetilde{\cM}\setminus\cM$ separately.

Since the extension operator constructed in the proof of Theorem~\ref{T34} is bounded in the H\"older norms,
$$
C_{\rm b}^{0,\alpha'}\subset W^{1-\frac{1}{q},q},\ 
C_{\rm b}^{0,\beta'}\subset W^{1-\frac{1}{p},p}
\quad
\text{and}
\quad
\Vert\cdot\Vert_{C^{0,\alpha'}}\lesssim \Vert\cdot\Vert_{C^{0,\alpha}},\ 
\Vert\cdot\Vert_{C^{0,\beta'}}\lesssim \Vert\cdot\Vert_{C^{0,\beta}},
$$
we have that
$$
\Vert\lambda\Vert_{W^{1-\frac{1}{q},q}(\widetilde{\cM})}\lesssim 
\Vert\lambda\Vert_{C^{0,\alpha}(\cM)},
\qquad
\Vert\gamma\Vert_{W^{1-\frac{1}{p},p}(\widetilde{\cM})}\lesssim 
\Vert\gamma\Vert_{C^{0,\beta}(\cM)}
$$
and similar estimates for $\tilde{\lambda}$ and $\tilde{\gamma}$. Now the estimate for the difference of integrals over $\widetilde{\cM}$ follows from Theorem~\ref{T33}.

To estimate the difference of integrals over $\widetilde{\cM}\setminus \cM$ we use the inequality similar to \eqref{eq92}
\[
\begin{split}
&|\lambda\wedge d\gamma_1\wedge\ldots\wedge d\gamma_k-
\tilde{\lambda}\wedge d\tilde{\gamma}_1\wedge\ldots\wedge d\tilde{\gamma}_k|\\
&\leq
|\lambda-\tilde{\lambda}|\, |d\gamma_1\wedge\ldots\wedge d\gamma_k|+
|\tilde{\lambda}|\, 
|d\gamma_1\wedge\ldots\wedge d\gamma_k-
d\tilde{\gamma}_1\wedge\ldots\wedge d\tilde{\gamma}_k|\\
&\lesssim
|\lambda-\tilde{\lambda}|\, |D\gamma|^k+|\tilde{\lambda}|\,|D(\gamma-\tilde{\gamma})|
(|D\gamma|+|D\tilde{\gamma}|)^{k-1}.
\end{split}
\]
Next, we integrate this inequality over $\widetilde{\cM}\setminus\cM$ and estimate the right hand side with the help of H\"older's inequality
\[
\begin{split}
\Bigg|\int_{{\widetilde{\cM}}\setminus\cM}\heartsuit-\int_{{\widetilde{\cM}}\setminus\cM}\diamondsuit\,\Bigg|
&\lesssim
\Vert\lambda-\tilde{\lambda}\Vert_q\Vert D\gamma\Vert_p^k+
\Vert\tilde{\lambda}\Vert_q\Vert D(\gamma-\tilde{\gamma})\Vert_p
(\Vert D\gamma\Vert_p+\Vert D\tilde{\gamma}\Vert_p)^{k-1}\\
&\leq
(\Vert\lambda-\tilde{\lambda}\Vert_q\Vert D\gamma\Vert_p+
\Vert\tilde{\lambda}\Vert_q\Vert D(\gamma-\tilde{\gamma})\Vert_p)
(\Vert D\gamma\Vert_p+\Vert D\tilde{\gamma}\Vert_p)^{k-1}.
\end{split}
\]
Here the norms on the right hand side are computed over $\widetilde{\cM}\setminus\cM$. 
Now the desired estimate in the H\"older norm on $\cM$ follows from the fact that the extensions $\lambda$, $\tilde{\lambda}$ and $\gamma$, $\tilde{\gamma}$ from the boundary $\partial\cM$ to $\widetilde{\cM}\setminus\cM$ are bounded operators from $C_{\rm b}^{0,\alpha'}$ and $C_{\rm b}^{0,\beta'}$ to $W^{1,q}$ and $W^{1,p}$ respectively. It remains only to observe that 
$\Vert\cdot\Vert_{C^{0,\alpha'}(\partial\cM)}\leq \Vert\cdot\Vert_{C^{0,\alpha}(\cM)}$ and similarly for $\beta'$ and $\beta$.
\end{proof}
Let us now state a special case of the result that will be useful in the applications.
\begin{corollary}
\label{T72}
Let $\Omega\subset\R^n$ be a bounded domain with Lipschitz boundary. 
Let 
$f\in C_{\rm b}^{0,\alpha}(\Omega)$, 
$g,\in C_{\rm b}^{0,\beta}(\Omega)^n$, $\alpha,\beta\in (0,1]$,
$\alpha+n\beta>n$.

If $\alpha'\in (0,\alpha]$, $\beta'\in(0,\beta]$, $\alpha'+n\beta'>n$ and
\begin{equation}
\label{eq75}
\text{$C_{\rm b}^{0,\alpha'}(\Omega)\ni f^\eps\to f$ in $C_{\rm b}^{0,\alpha'}$}
\quad
\text{and}
\quad
\text{$\lip({\Omega})^n\ni g^\eps\to g$ in $C_{\rm b}^{0,\beta_i'}$}
\quad
\text{as $\eps\to 0$,}
\end{equation}
then the limit
\begin{equation}
\label{eq76}
\langle f,dg_1\wedge\ldots\wedge dg_n\rangle=
\int_{\Omega} f dg_1\wedge\ldots\wedge dg_n:=\lim_{\eps\to 0} \int_{\Omega}f^\eps\, dg_1^\eps\wedge\ldots\wedge dg_n^\eps
\end{equation}
exists and it does not depend on the choice of approximations $f^\eps$, $g^\eps$.

Moreover, if $f,\tilde{f}\in C_{\rm b}^{0,\alpha}(\Omega)$, 
$g,\tilde{g}\in C_{\rm b}^{0,\beta}(\Omega)^n$, $\alpha,\beta\in (0,1]$,
$\alpha+n\beta>n$, then
\begin{align*}
&\left|\int_{\Omega} fdg_1\wedge\ldots\wedge dg_n-\int_{\Omega}\tilde{f}\tilde{g}_1\wedge\ldots\wedge d\tilde{g}_n\right|\\
&\lesssim
\left(\Vert f-\tilde{f}\Vert_{C^{0,\alpha}}\Vert g\Vert_{C^{0,\beta}}
+\Vert\tilde{f}\Vert_{C^{0,\alpha}}\Vert g-\tilde{g}\Vert_{C^{0,\beta}}\right)
\left(\Vert g\Vert_{C^{0,\beta}}+\Vert \tilde{g}\Vert_{C^{0,\beta}}\right)^{n-1}.
\end{align*}
\end{corollary}
\begin{proof}
While in Theorem~\ref{T34} and Corollary~\ref{T35} we assume that the boundary of $\cM$ is smooth, the proof can easily be adapted to the case of bounded domains with Lipschitz boundary. A crucial step in the proof was an extension of forms from $\cM$ to $\partial\cM\times [0,1]\subset\widetilde{\cM}$ with the help of Corollary~\ref{T7}.

Let $B\subset\bbbr^n$ be a ball such that $\Omega\Subset B$. Since $\Omega$ at its boundary points is locally bi-Lipschitz equivalent to $\bbbr^n _{+}$, and  both $C_{\rm b}^{0,\gamma}$ and $W^{1,p}$ spaces are invariant under a bi-Lipschitz change of variables, 
an obvious modification of the proof of Proposition~\ref{T6} allows us to extend the 
functions $f$, $f^\eps$ and $g$, $g^\eps$ from $\Omega$ (in fact, from $\partial\Omega)$
to compactly supported functions in $B$ and $B$ can be embedded into an oriented and closed manifold. The remainder of the proof follows arguments used in the proofs of Theorem~\ref{T34} and Corollary~\ref{T35}.
\end{proof}

\subsection{The one dimensional case}

The next result is a straightforward consequence of Corollary~\ref{T72}. However, since it will play a very important role in the study of the Heisenberg group, we prefer to give an elementary and direct proof.

Recall that the existence of a smooth approximation $f^\eps$ and $g^\eps$ in the statement of Proposition~\ref{T66} is guaranteed by Proposition~\ref{T3}. Indeed, we can always assume that $0<\alpha'<\alpha$ and $0<\beta'<\beta$.
\begin{proposition}
\label{T66}
Let $f\in C_{\rm b}^{0,\alpha}([a,b])$, $g\in C_{\rm b}^{0,\beta}([a,b])$, $0<\alpha,\beta\leq 1$, $\alpha+\beta>1$.
Let $0<\alpha'\leq\alpha$, $0<\beta'\leq\beta$ be such that $\alpha'+\beta'>1$. 
If
$$
\text{$\lip([a,b])\ni f^\eps\to f$ in $C_{\rm b}^{0,\alpha'}([a,b])$}
\quad
\text{and}
\quad
\text{$\lip([a,b])\ni g^\eps\to g$ in $C_{\rm b}^{0,\beta'}([a,b])$,}
$$
then the limit
\begin{equation}
\label{eq50}
\int_a^b f\, dg :=
\lim_{\eps\to 0} \int_a^b f^\eps\, d g^\eps=
\lim_{\eps\to 0} \int_a^b f^\eps(x)(g^\eps)'(x) \, dx
\end{equation}
exists and it does not depend on the choice of the approximation. Moreover, if
$\tilde{f}\in C_{\rm b}^{0,\alpha}([a,b])$, $\tilde{g}\in C_{\rm b}^{0,\beta}([a,b])$, then
\begin{equation}
\label{eq49}   
\left|\int_a^b f\, dg-\int_a^b \tilde{f}\, d\tilde{g}\right|\lesssim
\Vert f-\tilde{f}\Vert_{C^{0,\alpha}}[g]_\beta+
\Vert\tilde{f}\Vert_{C^{0,\alpha}}[g-\tilde{g}]_\beta.
\end{equation}
\end{proposition}
\begin{remark}
This result is essentially due to Young \cite{Young} who defined the integral $\int_a^b f\, dg$ of H\"older continuous functions (under assumptions of Proposition~\ref{T66}) as the Stieltjes integral. See also \cite{zust2,zust1} for extension of Young's approach to higher dimensional integrals. However, our approach is different. 
\end{remark}
We will precede the proof by presenting some auxiliary results.
\begin{lemma}
\label{T74}
Let $\alpha,\beta\in (0,1]$, $\alpha+\beta>1$. If $f\in C^{0,\alpha}([a,b])$, $g\in \lip([a,b])$, $f(a)=f(b)$ and $g(a)=g(b)$, then
$$
\Big|\int_a^b f(x)g'(x)\, dx\Big|\leq C(\alpha,\beta)|b-a|^{\alpha+\beta}\, [f]_{\alpha}\,[g]_{\beta}.
$$
\end{lemma}
\begin{proof}
Our proof will be based on elementary estimates of Fourier series.
By a linear change of variables we may assume that $[a,b]=[0,1]$.
Since $f(0)=f(1)$ and $g(0)=g(1)$, we can extend $f$ and $g$ to $1$-periodic functions on $\R$. Note that
\begin{equation}
\label{eq79}
[f]_{C^{0,\alpha}(\R)}\lesssim [f]_{C^{0,\alpha}([0,1])},
\end{equation}
but we do not necessarily have equality, because in the periodic extension of $f$ we have to glue $f$ near $0$ with $f$ near $1$. Similarly $[g]_{C^{0,\beta}(\R)}\lesssim [g]_{C^{0,\beta}([0,1])}$.
In what follows $[\,\cdot\,]_{\alpha}$ and $[\,\cdot\,]_{\beta}$ will stand for $[\,\cdot\,]_{C^{0,\alpha}([0,1])}$ and $[\,\cdot\,]_{C^{0,\beta}([0,1])}$ respectively.

Recall that the Fourier coefficients of $f$ are defined by
$\hat{f}(n)=\int_0^1f(x)e^{-2\pi inx}\, dx$, $n\in\Z$.
For $\tau>0$, let $f_\tau(x):=f(x+\tau)$, so
$$
\hat{f}_\tau(n)=\int_0^1f(x+\tau)e^{-2\pi inx}\, dx=
\int_0^1f(x)e^{-2\pi in(x-\tau)}\, dx=
e^{2\pi in\tau}\hat{f}(n).
$$
Therefore,
$(f-f_\tau)^\wedge(n)=
\hat{f}(n)(1-e^{2\pi in\tau})$,
and Parseval's identity yields
\begin{equation}
\label{eq80}
\sum_{n=-\infty}^\infty |\hat{f}(n)|^2|1-e^{2\pi in\tau}|^2=\Vert f-f_\tau\Vert^2_{L^2([0,1])}\lesssim [f]^2_{\alpha}\,\tau^{2\alpha}.
\end{equation}
Note that we have here inequality $\lesssim$ instead of $\leq$, because of \eqref{eq79}.

Our aim is to use \eqref{eq80} to obtain a good upper bound for $|\hat{f}(n)|^2$. To this end we want to select $\tau$ so that $|1-e^{2\pi in\tau}|\geq 1$. Such $\tau$ has to depend on $n$ and we can get an estimate only for $n$ being in a certain range like $2^k\leq |n|<2^{k+1}$.

Fix an integer $k\geq 0$ and let $\tau=2^{-(k+2)}$. If $2^k\leq n<2^{k+1}$, then
$\pi/2\leq 2\pi n\tau<\pi$. 
If $2^k\leq -n<2^{k+1}$, then $-\pi<2\pi n\tau\leq-\pi/2$. Therefore,
$$
\operatorname{Re} e^{2\pi in\tau}\leq 0,
\quad
\text{and hence}
\quad
|1-e^{2\pi in\tau}|>1,
\quad
\text{when $2^k\leq |n|<2^{k+1}$}.
$$
This and \eqref{eq80} yield
\begin{equation}
\label{eq81}
\sum_{2^k\leq |n|<2^{k+1}}|\hat{f}(n)|^2\leq
\sum_{n=-\infty}^\infty |\hat{f}(n)|^2|1-e^{2\pi in\tau}|^2\lesssim
[f]^2_{\alpha}\,\tau^{2\alpha}\approx
[f]^2_{\alpha}\,2^{-2k\alpha}.
\end{equation}
Similarly,
\begin{equation}
\label{eq82}
\sum_{2^k\leq |n|<2^{k+1}}|\hat{g}(n)|^2\lesssim
[g]^2_{\beta}\,2^{-2k\beta}.
\end{equation}
Recall that $\hat{g'}(n)=2\pi in\hat{g}(n)$. Therefore, Parseval's identity yields
\[
\begin{split}
\Bigg|\int_0^1 f(x)g'(x)\, dx\Bigg|
&\lesssim
\sum_{n=-\infty}^\infty |n|\, |\hat{f}(n)|\, |\hat{g}(n)|
\lesssim
\sum_{k=0}^\infty 2^k \sum_{2^k\leq |n|<2^{k+1}} |\hat{f}(n)|\, |\hat{g}(n)|\\
&\leq
\sum_{k=0}^\infty 2^k [f]_{\alpha}\, 2^{-k\alpha}[g]_{\beta}\, 2^{-k\beta}
\approx 
[f]_{\alpha}\,[g]_{\beta}.
\end{split}
\]
The last inequality is a consequence of \eqref{eq81}, \eqref{eq82} and the Cauchy-Schwarz inequality. Finally, convergence of the last sum follows from the assumption that $\alpha+\beta>1$.
\end{proof}
\begin{corollary}
\label{T75}
Let $\alpha,\beta\in (0,1]$, $\alpha+\beta>1$. If $f\in C^{0,\alpha}([a,b])$ and $g\in\lip([a,b])$, then
\begin{equation}
\label{eq56}
\Bigg|\int_a^b f(x)g'(x)\, dx\Bigg|\leq C(\alpha,\beta,b-a) \Vert f\Vert_{C^{0,\alpha}}[g]_{\beta}.
\end{equation}
If in addition $f(p)=0$ for some $p\in [a,b]$, then
\begin{equation}
\label{eq113}
\Bigg|\int_a^b f(x)g'(x)\, dx\Bigg|\leq C(\alpha,\beta)|b-a|^{\alpha+\beta} [f]_{\alpha}\,[g]_{\beta}.
\end{equation}
\end{corollary}
\begin{remark}
Note that the dependence on $b-a$ in \eqref{eq56} is not as nice as in \eqref{eq113}. It is due to lack of nice scaling of $\Vert f\Vert_{C^{0,\alpha}}$ with respect to linear transformations.
\end{remark}
\begin{proof}
As before, we may assume that $[a,b]=[0,1]$. The values of the constants in the inequalities \eqref{eq56} and \eqref{eq113} will follow from a simple linear change of variables.
First, observe that it suffices to prove a seemingly weaker inequality
\begin{equation}
\label{eq114}
\Bigg|\int_0^1 f(x)g'(x)\, dx\Bigg|\lesssim \Vert f\Vert_{C^{0,\alpha}}\Vert g\Vert_{C^{0,\beta}}.
\end{equation}
Indeed, if $h\in C^{0,\gamma}([0,1])$ and $h(p)=0$ for some $p\in [0,1]$, then
$$
\Vert h\Vert_\infty\leq [h]_{\gamma}
\qquad
\text{so}
\qquad
\Vert h\Vert_{C^{0,\gamma}}\lesssim [h]_{\gamma}.
$$
Thus, \eqref{eq114} applied to $g(x)-g(0)$ yields \eqref{eq56} and \eqref{eq56} implies \eqref{eq113}, when $f(p)=0$. 

Therefore, if remains to prove \eqref{eq114}.
Let $m:=f(1)-f(0)$ and $M=g(1)-g(0)$, so $|m|\leq 2\Vert f\Vert_\infty$ and $|M|\leq 2\Vert g\Vert_\infty$. Let $\phi(x):=f(x)-mx$ and $\psi(x)=g(x)-Mx$. Then $\phi(0)=\phi(1)$, $\psi(0)=\psi(1)$ and we may apply Lemma~\ref{T74} to the functions $\phi$ and $\psi$. Note that
$$
[\phi]_{\alpha}\leq [f]_{\alpha}+2\Vert f\Vert_\infty\leq 2\Vert f\Vert_{C^{0,\alpha}}.
$$
Similarly, $[\psi]_{\beta}\leq 2\Vert g\Vert_{C^{0,\beta}}$. We have
\[
\begin{split}
&
\Bigg|\int_0^1 f(x)g'(x)\, dx\Bigg|=
\Bigg|\int_0^1(\phi(x)+mx)(\psi(x)+Mx)'\, dx\Bigg|\\
&\leq
\Bigg|\int_0^1\phi(x)\psi'(x)\, dx\Bigg|+
\Bigg|M\int_0^1\phi(x)\, dx\Bigg|+
\Bigg|m\int_0^1x\psi'(x)\, dx\Bigg|+
\Bigg|mM\int_0^1 x\, dx\Bigg|\\
&=
A+B+C+D.
\end{split}
\]
Since the functions $\phi$ and $\psi$ satisfy assumptions of Lemma~\ref{T74}, we have
$$
A\lesssim [\phi]_{\alpha}\,[\psi]_{\beta}\lesssim
\Vert f\Vert_{C^{0,\alpha}}\Vert g\Vert_{C^{0,\beta}}.
$$
Clearly, $B+D\lesssim\Vert f\Vert_\infty\Vert g\Vert_\infty$. To estimate $C$ we use simple integration by parts
$$
C=\Bigg|m\int_0^1 x\psi'(x)\, dx\Bigg|=
|m|\, \Bigg|\psi(1)-\int_0^1\psi(x)\, dx\Bigg|\lesssim\Vert f\Vert_\infty\Vert g\Vert_\infty.
$$
Adding the estimates for $A$, $B$, $C$ and $D$ yields the result.
\end{proof}

\begin{proof}[Proof of Proposition~\ref{T66}]
If $f^\eps,\tilde{f}^\eps,g^\eps,\tilde{g}^\eps\in\lip([a,b])$, then Corollary~\ref{T75} and the traingle inequality  
$$
\Bigg|\int_a^b f^\eps (g^\eps)'\, dx-\int_a^b\tilde{f}^\eps(\tilde{g}^\eps)'\, dx\Bigg|
\leq 
\Bigg|\int_a^b (f^\eps-\tilde{f}^\eps)(g^\eps)'\, dx\Bigg|+
\Bigg|\int_a^b \tilde{f}^\eps(g^\eps-\tilde{g}^\eps)'\, dx\Bigg|
$$
imply that 
\begin{equation}
\label{eq126}
\Bigg|\int_a^b f^\eps\, dg^\eps-\int_a^b \tilde{f}^\eps\, d\tilde{g}^\eps\Bigg|\lesssim
\Vert f^\eps-\tilde{f}^\eps\Vert_{C^{0,\alpha'}}[g^\eps]_{\beta'}+
\Vert\tilde{f}^\eps\Vert_{C^{0,\alpha'}}[g^\eps-\tilde{g}^\eps]_{\beta'}.
\end{equation}
This inequality implies the existence of limit \eqref{eq50}. Passing to the limit as $\eps\to 0$, we see that 
\eqref{eq126} is also true for $f$, $\tilde{f}$, $g$ and $\tilde{g}$. Then \eqref{eq49} is a consequence of $\alpha'\leq\alpha$ and $\beta'\leq\beta$.
\end{proof}

\section{Pullbacks of differential forms}
\label{S6}

In Section~\ref{DJ} we discussed generalized Jacobians of H\"older and fractional Sobolev mappings. If $f$ is a smooth mappings, then 
$df_1\wedge\ldots\wedge df_n=f^*(dy_1\wedge\ldots\wedge dy_n)$, so the Jacobian is a special case of a pullback of a differential form.
In this section we will generalize results of Section~\ref{DJ} to the case of pullbacks of more general forms. While most of the results discussed here  can be generalized to mappings in
$W^{1-\frac{1}{p},p}$, and even to mappings in more general  Triebel-Lizorkin or Besov spaces  \cite{SY1999}, for the sake of simplicity we will assume that the mappings are H\"older continuous.
Section~\ref{S3} contains related material.

To simplify notation, we will often write in this section $\Vert\cdot\Vert_\alpha$ in place of
$\Vert\cdot\Vert_{C^{0,\alpha}}$. This should not lead to a confusion with the $L^p$-norm, because for the H\"older spaces we will always use Greek letters to denote the exponent of H\"older continuity. Thus, for example, in \eqref{eq94}, $\Vert\cdot\Vert_\infty$ is the supremum norm, while other norms are H\"older.

Let $\cM$ be a smooth oriented and compact $n$-dimensional Riemannian manifold with or without boundary. Let $f\in\lip(\cM;\R^m)$.
Any form $\kappa\in\Omega^k C_{\rm b}^{0,\alpha}(\R^m)$, $1\leq k\leq n$, can be written as
$$
\kappa=\sum_{|I|=k}\kappa_I\, dy_I , 
\quad 
\text{where  $\kappa_I\in C_{\rm b}^{0,\alpha}(\R^m)$, so}
\quad
f^*\kappa=\sum_{|I|=k}(\kappa_I\circ f)\, df_I. 
$$
Since in the discussion below, each term can be treated separately, we can assume that
$$
\kappa=\psi\, dy_1\wedge\ldots\wedge dy_k, \ \psi\in C_{\rm b}^{0,\alpha}(\R^m),
\quad
\text{and hence}
\quad
f^*\kappa=(\psi\circ f)\, df_1\wedge\ldots\wedge df_k.
$$
Therefore, for $\tau\in \Omega^{n-k}C^0(\cM)$, we have
$$
\int_{\cM} f^*\kappa\wedge\tau=
(-1)^{k(n-k)}\int_{\cM} (\psi\circ f)\tau\wedge df_1\wedge\ldots \wedge df_k.
$$
Assume now that
$$
f\in C_{\rm b}^{0,\gamma},\ \frac{k}{k+\alpha}<\gamma'\leq\gamma\leq 1
\quad
\text{and}
\quad
\lip({\cM};\R^m)\ni f^\eps\to f \ \text{in $C_{\rm b}^{0,\gamma'}$.}
$$
It follows from Proposition~\ref{T80} that for $0<\beta<\alpha$,
\begin{equation}
\label{eq94}
\Vert \psi\circ f-\psi\circ f^\eps\Vert_{\beta\gamma'}\lesssim
[\psi]_{\alpha}(\Vert f\Vert_{\gamma'}+\Vert f^\eps\Vert_{\gamma'})^{\beta}
\Vert f-f^\eps\Vert_\infty^{\alpha-\beta}.
\end{equation}
If $\tau\in \Omega^{n-k}C_{\rm b}^{0,\alpha\gamma}(\cM)$, then $\Vert\tau\Vert_{\beta\gamma'}\lesssim\Vert\tau\Vert_{\alpha\gamma}$ and Lemma~\ref{T27} yield
\begin{equation}
\label{eq95}
\Vert (\psi\circ f)\tau-(\psi\circ f^\eps)\tau\Vert_{\beta\gamma'}\lesssim
[\psi]_{\alpha}\Vert\tau\Vert_{\alpha\gamma}(\Vert f\Vert_{\gamma'}+\Vert f^\eps\Vert_{\gamma'})^{\beta}
\Vert f-f^\eps\Vert_\infty^{\alpha-\beta},
\end{equation}
so $(\psi\circ f^\eps)\tau\to(\psi\circ f)\tau$ in $C_{\rm b}^{0,\beta\gamma'}$ as $\eps\to 0$.

(Inequality \eqref{eq94} and hence \eqref{eq95} holds true with $\gamma'$ replaced by $\gamma$. The proof remains the same.)

If $\xi\in (\frac{k}{k+\alpha},\gamma')$, then we can find $\beta\in (0,\alpha)$, such that $\alpha\xi=\beta\gamma'$ and hence
$$
(\psi\circ f^\eps)\tau\to(\psi\circ f)\tau
\text{ in } C_{\rm b}^{0,\alpha\xi}
\quad
\text{for }
\frac{k}{k+\alpha}<\xi<\gamma'.
$$
We will apply now Theorem~\ref{T34} (cf.\ Corollary~\ref{T35}) to
$$
\lambda:=(\psi\circ f)\tau,\
{\lambda^\eps}:=(\psi\circ f^\eps)\tau\in C_{\rm b}^{0,\alpha\xi}
\quad
\text{and}
\quad
f_i,f_i^\eps\in C_{\rm b}^{0,\gamma'}
$$
($f_i$ plays a role of $\gamma_i$). 
Since $\alpha\xi+k\gamma'>k$
(because $\xi,\gamma'>\frac{k}{k+\alpha}$), Theorem~\ref{T34} yields the existence of the following limit
and its independence of the choice of a Lipschitz approximation~$f^\eps$:
$$
\int_{\cM}(\psi\circ f)\, df_1\wedge\ldots\wedge df_k\wedge\tau:=
\lim_{\eps\to 0}\int_{\cM}(\psi\circ f^\eps)\, df_1^\eps\wedge\ldots\wedge df_k^\eps\wedge\tau.
$$
We proved the following result:
\begin{theorem}
\label{T79}
Let $\cM$ be a smooth oriented and compact $n$-dimensional Riemnnian manifold with or without boundary. Let 
$$
f\in C_{\rm b}^{0,\gamma}(\cM;\R^m),
\quad
\text{where}
\quad
1\leq k\leq n,\ k\in\bbbn,\ \alpha\in (0,1],\ \frac{k}{k+\alpha}<\gamma\leq 1.
$$
If
$$
\lip({\cM};\R^m)\ni f^\eps\to f
\quad
\text{in $C_{\rm b}^{0,\gamma'}$ for some $\frac{k}{k+\alpha}<\gamma'\leq\gamma$,}
$$
then for any
$$
\kappa\in \Omega^k C_{\rm b}^{0,\alpha}\left(\R^m\right)
\quad
\text{and any}
\quad
\tau\in \Omega^{n-k}C_{\rm b}^{0,\alpha\gamma}\left(\cM\right),
$$
the limit
$$
\langle f^*\kappa,\tau\rangle=
\int_{\cM} f^*\kappa\wedge\tau
:=
\lim_{\eps\to 0}
\int_{\cM} (f^\eps)^*\kappa\wedge\tau
$$
exists and does not depend of the choice of $\gamma'\in\left(\frac{k}{k+\alpha},\gamma\right]$ and a Lipschitz approximation~$f^\eps$.
\end{theorem}
\begin{remark}
\label{R4}
Under the assumptions of Theorem~\ref{T79}, it follows that if $\kappa=\sum_{|I|=k}\kappa_I\, dy_I$, then the distributional pullback satisfies
$$
f^*\kappa = \sum_{|I|=k} (\kappa_I\circ f)\, df_I
$$
in a sense that
$$
\int_{\cM} f^*\kappa\wedge\tau=
\lim_{\eps\to 0}
\sum_{|I|=k} \int_{\cM} (\kappa_I\circ f)\, df^\eps_I\wedge\tau:=
\sum_{|I|=k} \int_{\cM} (\kappa_I\circ f)\, df_I\wedge\tau.
$$
To see this, it suffices to slightly modify the proof of Theorem~\ref{T79} and apply Theorem~\ref{T34} to
$\lambda^\eps=\lambda=(\psi\circ f)\tau$.
\end{remark}

The next result is a counterpart of Corollary~\ref{T35}.
\begin{theorem}
\label{T96}
Let $\cM$ be a smooth oriented and compact $n$-dimensional Riemannian manifold with or without boundary. If
$$
f,\tilde{f}\in C_{\rm b}^{0,\gamma}(\cM;\R^m),
\quad 
1\leq k\leq n, 
\quad
\alpha\in (0,1],
\quad 
\frac{k}{k+\alpha}<\gamma\leq 1
$$
and 
$$
\kappa\in \Omega^kC_{\rm b}^{0,\alpha}\left(\R^m\right),
\qquad
\tau\in \Omega^{n-k}C_{\rm b}^{0,\alpha\gamma}\left(\cM\right),
$$
then 
\begin{equation}
\label{eq98}
\Bigg| \int_{\cM} f^*\kappa\wedge\tau\Bigg|\lesssim[\kappa]_{\alpha}
\Vert\tau\Vert_{C^{0,\alpha\gamma}}
\Vert f\Vert_{C^{0,\gamma}}^{k+\alpha}.
\end{equation}
Moreover,
for any $\beta\in \big(\frac{k(1-\gamma)}{\gamma},\alpha\big)$
\begin{equation}
\label{eq99}
\Bigg| \int_{\cM} f^*\kappa\wedge\tau - \int_{\cM} \tilde{f}^*\kappa\wedge\tau\Bigg|\lesssim
[\kappa]_{\alpha}\Vert\tau\Vert_{C^{0,\alpha\gamma}}
(\Vert f\Vert_{C^{0,\gamma}}+\Vert\tilde{f}\Vert_{C^{0,\gamma}})^{k+\beta}
\Vert f-\tilde{f}\Vert_{C^{0,\gamma}}^{\alpha-\beta}.
\end{equation}
\end{theorem}
\begin{remark}
It follows from $\gamma>\frac{k}{k+\alpha}$ that $\alpha>\frac{k(1-\gamma)}{\gamma}$, so the interval $\big(\frac{k(1-\gamma)}{\gamma},\alpha\big)$ is not empty.
\end{remark}
\begin{proof}
We will prove \eqref{eq99} only as \eqref{eq98} follows directly from \eqref{eq99}.
As before, we may assume that $\kappa=\psi dy_1\wedge\ldots\wedge dy_k$, $\psi\in C_{\rm b}^{0,\alpha}(\R^m)$.

Since inequality \eqref{eq95} is true with $\gamma'$ replaced by $\gamma$, we have 
\begin{equation}
\label{eq96}
\Vert (\psi\circ f)\tau-(\psi\circ \tilde{f})\tau\Vert_{\beta\gamma}\lesssim
[\psi]_{\alpha}\Vert\tau\Vert_{\alpha\gamma}(\Vert f\Vert_{\gamma}+\Vert \tilde{f}\Vert_{\gamma})^{\beta}
\Vert f-\tilde{f}\Vert_\gamma^{\alpha-\beta}.
\end{equation}
We applied here a rough estimate $\Vert\cdot\Vert_\infty\leq\Vert\cdot\Vert_{C^{0,\gamma}}$ to the last factor.
In particular, replacing $f$ by $0$ yields
\begin{equation}
\label{eq97}
\Vert (\psi\circ\tilde{f})\tau\Vert_{\beta\gamma}\lesssim
[\psi]_\alpha\Vert\tau\Vert_{\alpha\gamma}\Vert\tilde{f}\Vert_\gamma^\alpha.
\end{equation}
Let $\beta\in \big(\frac{k(1-\gamma)}{\gamma},\alpha\big)$. It follows that $\delta+k\gamma>k$, where $\delta=\beta\gamma$.

With the obvious change of notation, we can write a special case of Corollary~\ref{T35} as follows:

If $\lambda,\tilde{\lambda}\in \Omega^{n-k}C_{\rm b}^{0,\delta}(\cM)$ and $f_i,\tilde{f}_i\in C_{\rm b}^{0,\gamma}(\cM)$, $i=1,2,\ldots,k$, where $\delta+k\gamma>k$, then
\begin{equation}
\label{eq93}
\begin{split}
&\Bigg|\int_{\cM}\lambda\wedge df_1\wedge\ldots\wedge df_k-
\int_{\cM}\tilde{\lambda}\wedge d\tilde{f}_1\wedge\ldots\wedge d\tilde{f}_k
\Bigg|\\
&\lesssim
(\Vert\lambda-\tilde{\lambda}\Vert_\delta\Vert f\Vert_\gamma+\Vert\tilde{\lambda}\Vert_\delta\Vert f-\tilde{f}\Vert_\gamma)
(\Vert f\Vert_\gamma+\Vert\tilde{f}\Vert_\gamma)^{k-1}.
\end{split}
\end{equation}
Applying this inequality along with \eqref{eq96} and \eqref{eq97} to
$$
\lambda:=(\psi\circ f)\tau
\quad
\text{and}
\quad
\tilde{\lambda}:=(\psi\circ\tilde{f})\tau
$$
leads after simple calculations, to
\[
\begin{split}
&\Bigg|\int_{\cM}(\psi\circ f)\tau\wedge df_1\wedge\ldots\wedge df_k-
\int_{\cM}(\psi\circ \tilde{f})\tau\wedge d\tilde{f}_1\wedge\ldots\wedge d\tilde{f}_k\Bigg|\\
&\lesssim
\Vert\tau\Vert_{\alpha\gamma}[\psi]_\alpha(\Vert f\Vert_\gamma+\Vert\tilde{f}\Vert_\gamma)^{k+\beta}
\Vert f-\tilde{f}\Vert_\gamma^{\alpha-\beta}.
\end{split}
\]
The proof is complete.
\end{proof}
\begin{corollary}
\label{T100}
Let $\cM$ be a smooth oriented $n$-dimensional manifold with or without boundary. Let
$f\in C^{0,\gamma}_{\rm loc}(\cM;\R^m)$, where $\gamma\in\big(\frac{k}{k+1},1\big]$, $k\in \{1,2,\ldots,n\}$. If
$$
C^{\infty}(\cM;\R^m)\ni f^\eps\to f
\quad
\text{pointwise on a dense subset of $\cM$ as }  \eps\to 0, \text{ and} 
$$
$$
\sup_{\eps>0}\Vert f^\eps\Vert_{C_{\rm b}^{0,\gamma}(K)}<\infty
\quad
\text{for every compact set } K\subset\cM,
$$
then for any $\kappa\in\Omega^k(\R^m)$ and any $\tau\in \Omega_c^{n-k}C^{0,\gamma}(\cM)$, the limit
$$
\langle f^*\kappa,\tau\rangle=\int_{\cM} f^*\kappa\wedge\tau:=
\lim_{\eps\to 0^+}\int_{\cM} (f^\eps)^*\kappa\wedge\tau
$$
exists and satisfies
$$
\bigg|\int_{\cM} f^*\kappa\wedge\tau\bigg|\lesssim C(\kappa,U_\tau,f(U_\tau))
\Vert\tau\Vert_{C^{0,\gamma}}\Vert f\Vert^{k+1}_{C_{\rm b}^{0,\gamma}(U_\tau)},
$$
where $U_\tau\Subset\cM$ is any neighborhood of $\operatorname{supp}\tau$. 
\end{corollary}
\begin{remark}
In the case of manifolds with boundary we assume that $\tau\in \Omega_c^{n-k}C^{0,\gamma}(\cM)$ has support disjoint from the boundary, so the boundary does not play any role.
\end{remark}
\begin{proof}
Let $U_\tau'\subset U_\tau$ be a compact submanifold of $\cM$ (possibly with boundary) that contains $\operatorname{supp}\tau$ in its interior. Then, the result follows from Theorems~\ref{T79} and~\ref{T96} with $\alpha=1$, applied to $U_\tau'$. While $[\kappa]_1$ need not be finite, $[\kappa]_1$ is finite in a neighborhood of $f(U_\tau)$. We also apply Corollary~\ref{T77} to $f^\eps$.
\end{proof}
\begin{corollary}
\label{T76}
Let $\Omega\subset\R^n$ be a bounded domain with smooth boundary and let $f\in C_{\rm b}^{0,\gamma}(\Omega;\R^n)$, $\gamma\in\big(\frac{n}{n+1},1\big]$. If 
$$
\lip({\Omega};\R^n)\ni f^\eps\to f
\quad
\text{in $C_{\rm b}^{0,\gamma'}$ for some $\frac{n}{n+1}<\gamma'\leq\gamma$,}
$$
then  for any $v\in C^\infty(\R^n)$ and any $g\in C_{\rm b}^{0,\gamma}(\Omega)$ the limit
$$
\int_\Omega (v\circ f)(x)J_f(x)g(x)\, dx:=
\lim_{\eps\to 0}
\int_\Omega (v\circ f^\eps)(x)J_{f^\eps}(x)g(x)\, dx=
\lim_{\eps\to 0} \int_\Omega  (v\circ f)J_{f^\eps}(x)g(x)\, dx
$$
exists and it does not depend on the choice of $\gamma'\in \left(\frac{n}{n+1},\gamma\right]$ and a Lipschitz approximation~$f^\eps$.
\end{corollary}
\begin{remark}
Later we will prove that the result is true with $g=1$ and $v\in L^{\infty}$, under the weaker assumption that $\gamma\in\big(\frac{n-1}{n},1\big]$, see Theorem~\ref{T88}.
\end{remark}
\begin{proof}
We will apply Theorem~\ref{T79} and Remark~\ref{R4} with 
$\cM=\Omega$, $m=n=k$, $\alpha=1$, $\kappa=v\,dy_1\wedge\ldots\wedge dy_n$ and $\tau=g$ ($\tau$ as a zero form is a function).
It suffices to observe that
$$
(f^\eps)^*\kappa\wedge\tau=
(v\circ f^\eps)J_{f^\eps}g(x)\, dx_1\wedge\ldots\wedge dx_n.
$$
Since $v$ as an arbitrary smooth function, $\kappa$ does not necessarily belong to $C^{0,\alpha}=\lip$, but this is not a problem since the family $f^\eps$ takes values in a bounded subset of $\R^n$ where $v$ is Lipschitz.
\end{proof}

\section{Green, Stokes and the change of variables formula}
\label{GSC}

The aim of this section is to generalize the change of variables formula involving degree, 
the Green theorem, and the Stokes theorem, to the case of H\"older continuous mappings. 

Let us start with the classical version of the change of variables formula.

Throughout this section we assume that $\Omega\subset\R^n$ is a bounded domain.
\begin{definition}[\mbox{\cite[Definition 1.2]{FG}}] 
\label{defn12}
Let $\Omega\subset\R^n$ be a bounded domain. Assume that $f\in C^\infty(\overbar{\Omega};\R^n)$. Let $y\in \R^n\setminus f(\partial\Omega)$ be a regular value of $f$. We define the \emph{local degree of $f$ at $y$ with respect to $\Omega$} as
$$
\deg(f,\Omega,y)=\sum_{x\in f^{-1}(y)} \sgn J_f(x).
$$
Moreover, if $y\not \in f(\overbar{\Omega})$, we set $\deg(f,\Omega,y)=0$.
\end{definition}
\begin{lemma}[\mbox{\cite[Proposition 1.8]{FG}}]
\label{prop18}
Let $V$ be a connected component of $\R^n\setminus f(\partial\Omega)$ and assume $y_1,y_2\in V$ are regular values of $f$. Then $\deg(f,\Omega,y_1)=\deg(f,\Omega,y_2)$.
\end{lemma}
Lemma~\ref{prop18} allows us to define the local degree at {\em all} points $y\in\R^n\setminus f(\Omega)$, by setting $\deg(f,\Omega,y)=\deg(f,\Omega,y_1)$,
where $y_1$ is any regular value of $f$ lying in the same component of $\R^n\setminus f(\Omega)$ as $y$.

The local degree is a $C^\infty$ homotopy invariant:
\begin{lemma}[\mbox{\cite[Theorem 1.12]{FG}}]
\label{thm112}
If $H:\overbar{\Omega}\times[0,1]\to\R^n$ is a $C^\infty$ mapping such that $H(\cdot,0)=f(\cdot)$, $H(\cdot,1)=g(\cdot)$
and $y\not\in H(\partial\Omega\times[0,1])$,
then $\deg(f,\Omega,y)=\deg(g,\Omega,y)$.
\end{lemma}
Proposition \ref{thm112} allows us to extend the notion of local degree to continuous mappings:
\begin{definition}[\mbox{\cite[Definition 1.18]{FG}}]\label{defn118}
If $f\in C(\overbar{\Omega};\R^n)$ and $y\not \in f(\partial\Omega)$, we set  $\deg(f,\Omega,y)=\lim_{\eps\to 0}\deg(f_\eps,\Omega,y)$, where $f_\eps\rightrightarrows f$ is a smooth approximation of $f$ that converges uniformly on $\overbar{\Omega}$ as $\eps\to 0$.
\end{definition}
Given $y\not \in f(\partial\Omega)$,
it follows from Lemma~\ref{thm112} that there is $\eps_o>0$ such that for all $0<\eps_1,\eps_2<\eps_o$,
$\deg(f_{\eps_1},\Omega,p)=\deg(f_{\eps_2},\Omega,y)$ so the limit exists. Moreover, the degree
$\deg(f,\Omega,y)$ defined in that way does not depend on the choice of the approximation.

The next lemma easily follows from Lemma~\ref{thm112}.
\begin{lemma}
Let $\Omega\subset\R^n$ be a bounded domain and let
$f:\partial\Omega\to\R^n$ be a continuous mapping. If $y\in\R^n\setminus f(\partial\Omega)$, and $F_1,F_2:\overbar\Omega\to\R^n$ are two continuous extensions of $f$, then
$$
\deg (F_1,\Omega,y)=\deg(F_2,\Omega,y).
$$
\end{lemma}
Since the degree is independent of the choice of the continuous extension, it depends only on the values of the mapping $f$ on the boundary of the domain, and we define:
\begin{definition}
\label{D5}
Let $\Omega\subset\R^n$ be a bounded domain, $f:\partial\Omega\to\R^n$, a continuous mapping, and $y\in\R^n\setminus f(\partial\Omega)$. Then, we define the {\em winding number of $f$ around $y$} by
$$
w(f,y)=\deg(F,\Omega,y),
$$
where $F:\overbar\Omega\to\R^n$ is any continuous extension of $f$.
\end{definition}

The classical change of variables formula reads as follows, see \cite[Theorem~5.27]{FG}.
\begin{lemma}
\label{T43}
Let $\Omega\subset\R^n$ be a bounded domain with smooth boundary and let $f\in C^\infty(\overbar{\Omega};\R^n)$. If $v\in L^\infty(\R^n)$, then
\begin{equation}
\label{eq405}
\int_{\Omega}(v\circ f)J_f(x)\, dx = \int_{\R^n}v(y)w(f,y)\, dy.
\end{equation}
\end{lemma}
\begin{remark}
The function $w(f,y)$ is not defined when $y\in f(\partial\Omega)$. However, the $n$-dimensional measure of the set $f(\partial\Omega)$ equals zero, so $w(f,y)$ is defined almost everywhere.
\end{remark}

\begin{remark}
In fact, the change of variables is true for Lipschitz mappings $f$ or even for Sobolev mappings satisfying some additional conditions. It is true also for much more general domains.
However, for our purposes the above statement is sufficient. For a proof and generalizations, see Chapter~5 in \cite{FG}.
\end{remark}
Assume now that
$v=\partial u/\partial y_i$ for some $u\in W^{1,\infty}(\R^n)$. Then 
\begin{align*}
\left(\frac{\partial u}{\partial y_i}\circ f\right) J_f\, dx_1\wedge\ldots\wedge dx_n 
&=
f^*\left(\frac{\partial u}{\partial y_i}\, dy_1\wedge\ldots\wedge dy_n\right)\\
&=
f^*\left((-1)^{i-1}d(u\, dy_1\wedge\ldots\wedge\widehat{dy_i}\wedge\ldots\wedge dy_n)\right)\\
&= 
(-1)^{i-1}d\left((u\circ f) df_1\wedge\ldots\wedge\widehat{df_i}\wedge\ldots\wedge df_n)\right).
\end{align*}
The Stokes theorem (Lemma~\ref{T78}) along with Lemma~\ref{T43} yields
\begin{corollary}
\label{T44}
Let $\Omega\subset\R^n$ be a bounded domain with smooth boundary and let $f=(f_1,\ldots,f_n)\in C^\infty(\partial\Omega;\R^n)$
Then for $u\in W_{\rm loc}^{1,\infty}(\R^n)$ and $i=1,2,\ldots,n$ we have
\begin{equation}
\label{eq406}
(-1)^{i-1}\int_{\partial\Omega} (u\circ f)\,df_1\wedge\ldots\wedge\widehat{df_i}\wedge\ldots\wedge df_n=
\int_{\R^n}\frac{\partial u}{\partial y_i}(y)w(f,y)\, dy.
\end{equation}
\end{corollary}
\begin{remark}
\label{R9}
When $u=y_i$, this is a multidimensional version of Green's theorem, see also Corollary~\ref{T25}.
\end{remark}
Any form $\kappa\in \Omega^{n-1}W_{\rm loc}^{1,\infty}\big(\R^n\big)$ can be written as
$$
\kappa=\sum_{i=1}^n(-1)^{i-1}u_i\, dy_1\wedge\ldots\wedge\widehat{dy_i}\wedge\ldots\wedge dy_n.
$$
Let $U=(u_1,\ldots,u_n)$. Then \eqref{eq406} yields
\begin{equation}
\label{eq8}
\int_{\partial\Omega} f^*\kappa=\int_{\R^n}\operatorname{div} U(y)\, w(f,y)\, dy=
\int_{\R^n} w(f,y)\, d\kappa(y).
\end{equation}
This gives the following version of the Stokes theorem.
\begin{corollary}
\label{T87}
Let $\Omega\subset\R^n$ be a bounded domain with smooth boundary and let $f\in C^\infty(\partial\Omega;\R^n)$. Then for any $\kappa\in \Omega^{n-1}W_{\rm loc}^{1,\infty}\big(\R^n\big)$, we have
\begin{equation}
\label{eq108}
\int_{\partial\Omega} f^*\kappa=\int_{\R^n} w(f,y)\, d\kappa(y).
\end{equation}
\end{corollary}
\begin{remark}
Let $F:\Omega\to\R^n$ be a smooth extension of $f$. Then \eqref{eq108} takes the form of
$$
\int_{\partial\Omega} F^*\kappa=\int_{\R^n}\deg(F,\Omega,y)\, d\kappa(y),
$$
and it can be regarded as Stokes' theorem for the singular parametric surface $F:\Omega\to\R^n$, where the degree counts the self-intersections of the surface taking into account the orientation.
\end{remark}

The purpose of this section is to generalize Lemma~\ref{T43}, Corollary~\ref{T44} and Corollary~\ref{T87} to the case of H\"older continuous mappings. We will assume that $\Omega\subset\R^n$ is a bounded domain with smooth boundary and $f\in C_{\rm b}^{0,\gamma}(\partial\Omega;\R^n)$ for some $\gamma>\frac{n-1}{n}$. Under these assumptions, the $n$-dimensional measure of $f(\partial\Omega)$ equals zero and hence $w(f,\cdot)$ is defined almost everywhere.

Theorem~\ref{T83} due to Olbermann \cite{Olber}, deals with integrability properties of the function $w(f,\cdot)$. A simpler proof was discovered by De Lellis and Inauen \cite{Delellis} who also proved fractional Sobolev regularity of the function $w(f,\cdot)$. Our proof is similar to that in \cite{Delellis}. See also \cite[Theorem~1.3]{ZustGAFA}.
\begin{theorem}
\label{T83}
If $\Omega$ is a bounded domain with smooth boundary and $f\in C_{\rm b}^{0,\gamma}(\partial\Omega;\R^n)$ for some $\gamma\in\big(\frac{n-1}{n},1\big]$, then
\begin{enumerate}
\item[(a)] 
$\Vert w(f,\cdot)\Vert_{L^p(\R^n)}\lesssim\Vert f\Vert_{C^{0,\gamma}}^{\frac{n}{p}}$
for all  $1<p<\frac{n\gamma}{n-1}$,
where the implied constant depends on $n$, $p$, $\gamma$, and $\Omega$ only.
\item[(b)]
If $f_k\to f$ in $C_{\rm b}^{0,\gamma}$, then $w(f_k,\cdot)\to w(f,\cdot)$ in $L^p(\R^n)$, for all $1\leq p<\frac{n\gamma}{n-1}$.
\end{enumerate}
\end{theorem}
\begin{remark}
Since $w(f,\cdot)$ equals zero outside a bounded set, it follows that $w(f,\cdot)\in L^1(\R^n)$.
\end{remark}

In the proof we will need three lemmata. All of them are well known.
\begin{lemma}
\label{T84}
Let $\mu$ be a finite measure on a space $X$. Suppose that $f_k\to f$ $\mu$-a.e., and $\sup_k\Vert f_k\Vert_{L^q(X)}<\infty$ for some $1<q<\infty$. Then $f\in L^q(X)$ and $f_k\to f$ in $L^p(X)$ for any $1\leq p<q$.
\end{lemma}
\begin{proof}
Let $M=\sup_k\Vert f_k\Vert_{L^q}$. It follows from Fatou's lemma that $f\in L^q$ and $\Vert f\Vert_{L^q}\leq M$. It remains to prove that $f_k\to f$ in $L^p$. Suppose to the contrary that there is $\eps>0$ and a subsequence such that $\Vert f_{k_\ell}-f\Vert_{L^p}>\eps$. According to Egorov's theorem, there is $E\subset X$, such that
$$
f_{k_\ell}\rightrightarrows f \text{ uniformly on $E$}
\qquad
\text{and}
\qquad
\mu(X\setminus E)<\Big(\frac{\eps}{4M}\Big)^{\frac{pq}{q-p}}.
$$
Minkowski's and H\"older's inequalities yield
\[
\begin{split}
\Vert f_{k_\ell}-f \Vert_{L^p}
&=
\Vert (f_{k_\ell}-f)\chi_E + (f_{k_\ell}-f)\chi_{X\setminus E}\Vert_{L^p}\\
&\leq 
\Bigg(\int_E |f_{k_\ell}-f|^p\, d\mu\Bigg)^{1/p} +
\Bigg(\int_{X\setminus E} |f_{k_\ell}-f|^q\, d\mu\Bigg)^{1/q}\mu(X\setminus E)^{\frac{q-p}{pq}}\\
&\leq 
\Bigg(\int_E |f_{k_\ell}-f|^p\, d\mu\Bigg)^{1/p} +\frac{\eps}{2}<\eps,
\end{split}
\]
for sufficiently large $\ell$, which leads to a contradiction.
\end{proof}
\begin{lemma}
\label{T85}
Let $\phi\in C_c^\infty(\R^n)$. If
\begin{equation}
\label{eq100}
U(x)=\frac{1}{n\omega_n}\int_{\R^n}\frac{x-z}{|x-z|^n}\phi(z)\, dz,
\end{equation}
then $U\in C^\infty(\R^n)$ satisfies $\operatorname{div} U=\phi$. Moreover, 
\begin{equation}
\label{eq106}
\Vert DU\Vert_{L^p(\R^n)}\leq C(n,p)\Vert\phi\Vert_{L^p(\R^n)}
\quad
\text{for all $1<p<\infty$.}
\end{equation}
\end{lemma}
\begin{proof}
Let $\Gamma$ be the fundamental solution of the Laplace equation. Then $\Gamma*\phi\in C^\infty(\R^n)$ solves the Poisson equation 
$\Delta(\Gamma*\phi)=\phi$ and it follows from basic Calder\'on-Zygmund estimates (see \cite[Theorem~9.9]{GilbargT}) that
$$
\Vert D^2(\Gamma*\phi)\Vert_{L^p(\R^n)}\leq C(n,p)\Vert\Delta(\Gamma*\phi)\Vert_{L^p(\R^n)}=\Vert\phi\Vert_{L^p(\R^n)},
\quad
\text{for } 1<p<\infty.
$$
Let $U:=D(\Gamma*\phi)=(D\Gamma)*\phi\in C^\infty(\R^n)$. Clearly $U$ satisfies $\operatorname{div}U=\phi$.
Since
$$
D\Gamma(x)=\frac{1}{n\omega_n}\frac{x}{|x|^n},
$$
$U$ is defined by \eqref{eq100}. We have
$\Vert DU\Vert_{L^p}=\Vert D^2(\Gamma*\phi)\Vert_{L^p}\lesssim\Vert\phi\Vert_{L^p}$.
\end{proof}
The next result is a consequence of Morrey's inequality, see \cite[Section~4.5.3]{EG}.
\begin{lemma}
\label{T86}
If $u\in W^{1,p}_{\rm loc}(\R^n)$, $n<p<\infty$, satisfies $Du\in L^p(\bbbr^n)$, then
$$
|u(y)-u(x)|\leq C(n,p)|x-y|^{1-\frac{n}{p}}\Vert Du\Vert_{L^p(\R^n)} 
\quad
\text{for all } x,y\in\R^n.
$$
\end{lemma}
\begin{proof}[Proof of Theorem~\ref{T83}]
We will begin with the proof of (a).
Assume first that $f\in C^\infty(\partial\Omega;\R^n)$.
If $\alpha\in(0,1]$ is such that
\begin{equation}
\label{eq107}
\gamma>\frac{n-1}{n-1+\alpha},
\end{equation}
then for any $\kappa\in \Omega^{n-1}C_{\rm b}^{0,\alpha}(\R^n)$, \eqref{eq98} applied to $\cM=\partial\Omega$, yields the estimate
$$
\Bigg|\, \int_{\partial\Omega} f^*\kappa\Bigg|\lesssim
[\kappa]_{\alpha} \Vert f\Vert^{n-1+\alpha}_{C^{0,\gamma}}.
$$
Let $\phi\in C_c^\infty(\R^n)$ and let $U=(u_1,\ldots,u_n)\in C^\infty(\R^n)$ be the solution to $\operatorname{div} U=\phi$ as in Lemma~\ref{T85}. If
$$
\kappa=\sum_{i=1}^n(-1)^{i-1}u_i\, dy_1\wedge\ldots\wedge\widehat{dy_i}\wedge\ldots\wedge dy_n,
$$
then \eqref{eq8} yields
$$
\int_{\partial\Omega} f^*\kappa=\int_{\R^n}\operatorname{div} U(y)\,  w(f,y)\, dy=
\int_{R^n}\phi(y) w(f,y)\, dy,
$$
and hence
\begin{equation}
\label{eq1}
\Bigg| \int_{\R^n}\phi(y)w(f,y)\, dy\Bigg|\lesssim
[U]_{\alpha} \Vert f\Vert^{n-1+\alpha}_{C^{0,\gamma}}.
\end{equation}
Let $q$ be the H\"older conjugate to $p$. Since $1<p<\frac{n}{n-1}$, it follows that $n<q<\infty$.
Let $\alpha=1-\frac{n}{q}$ and note that \eqref{eq107} is satisfied. Lemmata~\ref{T86} and~\ref{T85} yield
$$
[U]_{C^{0,\alpha}(\R^n)}\lesssim\Vert DU\Vert_{L^q(\R^n)}\lesssim\Vert\phi\Vert_{L^q(\R^n)},
$$
and hence \eqref{eq1} implies
$$
\Bigg|\int_{\R^n}\phi(y)w(f,y)\, dy\Bigg|\lesssim\Vert\phi\Vert_{L^q}\Vert f\Vert_{C^{0,\gamma}}^{\frac{n}{p}}.
$$
Taking supremum over all $\phi$ with $\Vert\phi\Vert_{L^q}\leq 1$, proves the  estimate
$\Vert w(f,\cdot)\Vert_{L^p(\R^n)}\lesssim\Vert f\Vert_{C^{0,\gamma}}^{\frac{n}{p}}$, under the assumption that $f\in C^\infty$. That completes the proof of (a) for smooth mappings. 
The general case will be obtained by approximation. 

Thus, assume that $f\in C_{\rm b}^{0,\gamma}(\partial\Omega;\R^n)$. Let $f_k\in C^\infty(\partial\Omega;\R^n)$ be a sequence 
such that $f_k\rightrightarrows f$ uniformly and $\sup_k\Vert f_k\Vert_{C^{0,\gamma}}\lesssim\Vert f\Vert_{C^{0,\gamma}}$, see
Corollary~\ref{T71}. We have,
$$
\Vert w(f_k,\cdot)\Vert_{L^p}\lesssim\Vert f_k\Vert_{C^{0,\gamma}}^{\frac{n}{p}}\lesssim\Vert f\Vert_{C^{0,\gamma}}^{\frac{n}{p}},
$$
and it suffices to prove that $w(f_k,\cdot)\to w(f,\cdot)$ in $L^p(\R^n)$.

Since $w(f_k,y)\to w(f,y)$ pointwise for all $y\in\R^n\setminus f(\partial\Omega)$, it follows that $w(f_k,\cdot)\to w(f,\cdot)$ almost everywhere.

Fix $\tilde{p}$, such that 
$p<\tilde{p}<\frac{n\gamma}{n-1}$.
We proved the inequality in (a) for smooth mappings. Applying it to $f_k$ with $p$ replaced by $\tilde{p}$, we have
$$
\Vert w(f_k,\cdot)\Vert_{L^{\tilde{p}}}\lesssim
\Vert f_k\Vert_{C^{0,\gamma}}^{\frac{n}{\tilde{p}}}\lesssim
\Vert f\Vert_{C^{0,\gamma}}^{\frac{n}{\tilde{p}}}.
$$
Note that there is $R>0$ such that $w(f_k,y)=w(f,y)=0$ for all $k$ and all $y\in \R^n\setminus B_R$. Since $w(f_k,\cdot)\to w(f,\cdot)$ almost everywhere, and the sequence is bounded in $L^{\tilde{p}}$, it follows from Lemma~\ref{T84} that $w(f_k,\cdot)\to w(f,\cdot)$ in $L^p(B_R)$ and hence in $L^p(\R^n)$. That completes the proof of (a).

The proof of part (b) is essentially contained in estimates used in the proof of part (a). Let $\tilde{p}$, and $R$ be as above.

If $f_k\to f$ in $C_{\rm b}^{0,\gamma}$, then $w(f_k,\cdot)\to w(f,\cdot)$ almost everywhere. Since
$$
\Vert w(f_k,\cdot)\Vert_{L^{\tilde{p}}}\lesssim\sup_k\Vert f_k\Vert^{\frac{n}{\tilde{p}}}_{C^{0,{\gamma}}}<\infty,
$$
and $\tilde{p}>p$, it follows from Lemma~\ref{T84} that $w(f_k,\cdot)\to w(f,\cdot)$ in $L^p(B_R)$ and hence in $L^p(\R^n)$.
The proof is complete.
\end{proof}
Compare the next result with Corollary~\ref{T76}. 
\begin{theorem}
\label{T88}
Let $\Omega\subset\R^n$ be a bounded domain with smooth boundary and let $f\in C_{\rm b}^{0,\gamma}(\Omega;\R^n)$, $\gamma\in\big(\frac{n-1}{n},1\big]$. If 
$$
C^\infty(\overbar{\Omega};\R^n)\ni f^\eps\to f
\quad
\text{in $C^{0,\gamma'}$ for some $\frac{n-1}{n}<\gamma'\leq\gamma$,}
$$
then  for any $v\in L^\infty_{\rm loc}(\R^n)$, the limit
\begin{equation}
\label{eq109}
\int_\Omega (v\circ f)(x)J_f(x)\, dx:=
\lim_{\eps\to 0}
\int_\Omega (v\circ f^\eps)(x)J_{f^\eps}(x)\, dx
\end{equation}
exists and it does not depend on the choice of $\gamma'\in \left(\frac{n-1}{n},\gamma\right]$ and a smooth approximation $f^\eps$. Moreover,
\begin{equation}
\label{eq110}
\int_\Omega (v\circ f)(x)J_f(x)\, dx=\int_{\R^n} v(y) w(f,y)\, dy.
\end{equation}
\end{theorem}
\begin{remark}
\label{R6}
The result is true for a larger class of functions $v$ due to higher integrability of the winding number, but we consider the case $v\in L^\infty_{\rm loc}$ only in order to avoid technical difficulties. This case is sufficient for most of the applications. For related results see also 
\cite[Proposition~6]{Conti}, \cite[Lemma~2.2]{LZ},
\cite[Proposition 7.1]{LP17}, and \cite[Lemma 3.1.]{LPS21}, \cite[Theorem~1.3]{ZustGAFA}.
\end{remark}
\begin{proof}
According to Lemma~\ref{T43} and Theorem~\ref{T83}(b),
$$
\int_\Omega (v\circ f^\eps)(x)J_{f^\eps}(x)\, dx=\int_{\R^n}v(y)w(f^\eps,y)\,dy
\to \int_{\R^n} v(y)w(f,y)\, dy,
\quad
\text{as $\eps\to 0$,}
$$
because $w(f^\eps,\cdot)\to w(f,\cdot)$ in $L^1$, the functions vanish outside a certain ball, and $v\in L^\infty$ on that ball. 
We proved that the limit \eqref{eq109} satisfies \eqref{eq110} and hence it does not depend on the choice of $\gamma'$ and approximation.
\end{proof}
\begin{theorem}
\label{T89}
Let $\Omega\subset\R^n$ be a bounded domain with smooth boundary and let $f\in C_{\rm b}^{0,\gamma}(\partial\Omega;\R^n)$, $\gamma\in \big(\frac{n-1}{n},1\big]$.
If 
$$
C^\infty(\partial\Omega;\R^n)\ni f^\eps\to f
\quad
\text{in $C_{\rm b}^{0,\gamma'}$ for some $\frac{n-1}{n}<\gamma'\leq\gamma$,}
$$
Then for any $\kappa\in \Omega^{n-1}W^{1,\infty}_{\rm loc}\big(\R^n\big)$, the limit
\begin{equation}
\label{eq111}
\int_{\partial\Omega} f^*\kappa:=\lim_{\eps\to 0} \int_{\partial\Omega} (f^\eps)^*\kappa
\end{equation}
exists and it does not depend on the choice of $\gamma'\in \left(\frac{n-1}{n},\gamma\right]$ and a smooth approximation $f^\eps$. Moreover,
\begin{equation}
\label{eq112}
\int_{\partial\Omega} f^*\kappa=\int_{\R^n} w(f,y)\, d\kappa(y).
\end{equation}
\end{theorem}
\begin{remark}
Comments in Remark~\ref{R6} apply here as well.
\end{remark}
\begin{proof}
According to Corollary~\ref{T87}, \eqref{eq8}, and Theorem~\ref{T83}(b),
$$
\int_{\partial\Omega} (f^\eps)^*\kappa=
\int_{\R^n}\operatorname{div}U(y)w(f^\eps,y)\, dy\stackrel{\eps\to 0}{\longrightarrow}
\int_{\R^n}\operatorname{div}U(y)w(f,y)\, dy=
\int_{\R^n}w(f,y)\, d\kappa(y), 
$$
because $w(f^\eps,\cdot)\to w(f,\cdot)$ in $L^1$, the functions vanish outside a certain ball and $\operatorname{div}U\in L^\infty$ on that ball.
We proved that the limit \eqref{eq111} satisfies \eqref{eq112} and hence it does not depend on the choice of $\gamma'$ and approximation.
\end{proof}

Applying Theorem~\ref{T89} to 
$\kappa=(-1)^{i-1}y_i\, dy_1\wedge\ldots\wedge\widehat{dy_i}\wedge\ldots\wedge dy_n$ yields
$u=x_i$ and we immediately obtain (cf.\ \cite[Lemma~2.3]{LZ}):
\begin{corollary}
\label{T25}
If $\Omega\subset\R^n$ is a bounded domain with smooth boundary and $f\in C_{\rm b}^{0,\gamma}(\partial\Omega;\R^n)$ for some $\gamma>\frac{n-1}{n}$, then
\begin{equation}
\label{eq127}
(-1)^{i-1}\int_{\partial\Omega} f_i\, df_1\wedge\ldots\wedge\widehat{df_i}\wedge\ldots\wedge df_n=
\int_{\R^n} w(f,y)\, dy
\quad
\text{for $i=1,2,\ldots,n$.}
\end{equation}
In particular we obtain the following generalization of Green's formula
\begin{equation}
\label{eq57}
\frac{1}{n}\int_{\partial\Omega} \sum_{i=1}^n (-1)^{i-1}f_i\, df_1\wedge\ldots\wedge\widehat{df_i}\wedge\ldots\wedge df_n=
\int_{\R^n}w(f,y)\, dy.
\end{equation}
\end{corollary}
\begin{remark}
As in Theorem~\ref{T89}, the generalized integral on the left hand side of \eqref{eq127} is defined as the limit of integrals of smooth approximations
$$
\int_{\partial\Omega} f_i\, df_1\wedge\ldots\wedge\widehat{df_i}\wedge\ldots\wedge df_n=
\lim_{\eps\to 0}\int_{\partial\Omega} f_i^\eps\, df_1^\eps\wedge\ldots\wedge\widehat{df_i^\eps}\wedge\ldots\wedge df^\eps_n,
$$
where $C^\infty(\partial\Omega;\R^n)\ni f^\eps\to f$ in $C_{\rm b}^{0,\gamma'}$ for some $\gamma'\in (\frac{n-1}{n},\gamma]$.
\end{remark}
\begin{remark}
Note that the right hand side of \eqref{eq57} can be interpreted as the oriented volume enclosed by $f:\partial\Omega\to\R^n$. In particular, if $\partial\Omega$ is connected and $f$ is an embedding, the right hand side of \eqref{eq57} equals $\pm 1$ times the volume enclosed by the surface $f(\partial\Omega)$.
\end{remark}
The next result follows from  Theorem~\ref{T88} applied to $v=1$ and from  Corollary~\ref{T25}.
\begin{corollary}
\label{T93}
If $\Omega\subset\R^n$ is a bounded domain with smooth boundary and $f\in C_{\rm b}^{0,\gamma}(\Omega;\R^n)$ for some $\gamma>\frac{n-1}{n}$, then
$$
\int_\Omega df_1\wedge\ldots\wedge df_n=
\frac{1}{n}\int_{\partial\Omega} \sum_{i=1}^n (-1)^{i-1}f_i\, df_1\wedge\ldots\wedge\widehat{df_i}\wedge\ldots\wedge df_n=
\int_{\R^n}w(f,y)\, dy.
$$
\end{corollary}
The special case when $n=2$ and $\Omega=\bbbb^2$, so $\partial\Omega=\bbbs^1$, was originally proved in \cite[Lemma~2.3]{LZ}.
We will write $\gamma$ instead of $f$ to be consistent with notation used in Section~\ref{S7}.
\begin{corollary}
\label{T91}
If $\gamma=(\gamma^x,\gamma^y):\mathbb{S}^1\to\mathbb{R}^2$ is $\alpha$-H\"older continuous for some $\alpha\in (\frac{1}{2},1]$, then
\begin{equation}
\label{eq187}
\frac{1}{2}\int_{\mathbb{S}^1}\gamma^x d\gamma^y-\gamma^y d\gamma^x=\int_{\mathbb{R}^2}w(\gamma,z)\, dz.
\end{equation}
\end{corollary}
We also have
\begin{corollary}
\label{T94}
If $f=(f^x,f^y):\bbbb^2\to\R^2$ is $\alpha$-H\"older continuous for some $\alpha\in (\frac{1}{2},1]$, then
$$
\int_{\bbbb^2} df^x\wedge df^y=\int_{\R^2} w(f,z)\, dz.
$$
\end{corollary}
Both corollaries will be used in Section~\ref{S7}.

\section{Heisenberg groups}
\label{S7}
\subsection{Preliminaries}
In this section we collect basic definitions and facts from the theory of the Heisenberg groups.

\begin{definition}
The {\em Heisenberg group}  is a Lie group
$\bbbh^n=\bbbc^n\times\bbbr=\bbbr^{2n+1}$ equipped with the group law
\begin{align*}
(z,t)*(z',t')
&=
\left(z+z',t+t'+2\, {\rm Im}\,  \left(\sum_{j=1}^n z_j
\overline{z_j'}\right)\right)\\
&=
\left(x+x',y+y',t+t'+2\sum_{j=1}^n(y_jx_j'-x_jy_j')\right).
\end{align*}
Here and in what follows we will denote coordinates in $\R^{2n+1}$ by
$(x_1,y_1,\ldots,x_n,y_n,t)$.
Also, $t$ will be called the \emph{height} of the point $(z,t)$.
\end{definition}

Note also that if we identify elements $z$ and $z'$ with vectors in $\bbbr^{2n}$, then the
product in the Heisenberg group can be written in terms of the standard symplectic form, see \eqref{eq128}
$$
(z,t)*(z',t')=(z+z',t+t'-2\upomega(z,z')).
$$
A basis of left invariant vector fields is given 
at any point $(x_1,y_1,\dots,x_n,y_n,t) \in \bbbh^n$, by
$$
X_j=\frac{\partial}{\partial x_j}+ 2y_j\, \frac{\partial}{\partial t},
\quad
Y_j=\frac{\partial}{\partial y_j}- 2x_j\, \frac{\partial}{\partial t},
\quad
T=\frac{\partial}{\partial t}
$$
for $j=1,2,\dots,n$.
These left invariant vector fields determine the Lie algebra $\mathfrak{h}^n$. 
It is not hard to see that
$[X_j,Y_j]=-4T$ for $j=1\dots,n$
and all other commutators vanish.

\begin{definition}
The {\em Horizontal distribution} in the Heisenberg group
is a subbundle of the tangent bundle defined by 
$$
H_p \bbbh^n = \text{span} \left\{ X_1(p),Y_1(p),\dots,X_n(p),Y_n(p) \right\} \quad \text{for all } p \in \bbbh^n.
$$
\end{definition}
This is a smooth distribution of 
$2n$-dimensional subspaces in the 
$(2n+1)$-dimensional tangent space $T_p \bbbh^n = T_p \bbbr^{2n+1}$.

The next lemma is straightforward. It shows that the horizontal distribution in $\bbbh^n$ is nothing else but a standard contact structure on $\R^{2n+1}$.
\begin{lemma}
\label{T38}
The horizontal distribution $H\bbbh^n$ is the kernel of a standard contact form
$$
\upalpha = dt + 2 \sum_{j=1}^n (x_j \, dy_j - y_j \, dx_j).
$$
i.e. $H_p \bbbh^n = \operatorname{ker}\upalpha(p) \subset T_p\R^{2n+1}$ for all $p\in \R^{2n+1}$
\end{lemma}

\begin{definition}
An absolutely continuous curve $\gamma:[a,b] \to \R^{2n+1}$ 
is a \emph{horizontal curve} if it is almost everywhere tangent to the horizontal distribution i.e.,
$$
\gamma'(s) \in H_{\gamma(s)} \bbbh^n \quad \text{for almost every } s \in [a,b].
$$
\end{definition}
It easily follows from Lemma~\ref{T38} that an absolutely continuous curve
$$
\gamma(s)= (\gamma^{x_1}(s),\gamma^{y_1}(s),\dots,\gamma^{x_n}(s),\gamma^{y_n}(s),\gamma^t(s))
$$
is horizontal if and only if
\begin{equation}
\label{eq59}
(\gamma^t)'(s) = 2\sum_{j=1}^n (\gamma^{y_j}(s)(\gamma^{x_j})'(s) - \gamma^{x_j}(s)(\gamma^{y_j})'(s)) \quad \text{ for almost every } s \in [a,b].
\end{equation}
Let $\gamma_j=(\gamma^{x_j},\gamma^{y_j})$ be the projection of $\gamma$ onto the $x_jy_j$-plane. Integrating \eqref{eq59} yields the following formula for the change of height $\gamma^t$ along the curve $\gamma$:
\begin{equation}
\label{eq60}
\begin{split}
\gamma^t(b)-\gamma^t(a)
&=
2\sum_{j=1}^n \int_a^b \gamma^{y_j}(s)(\gamma^{x_j}(s))'-\gamma^{x_j}(s)(\gamma^{y_j}(s))'\, ds\\
&=
-4\sum_{j=1}^n \Bigg(\frac{1}{2}\int_{\gamma_j} x_j\, dy_j-y_j\, dx_j\Bigg).
\end{split}    
\end{equation}
This formula has a nice geometric interpretation.

If $\gamma$ is a closed horizontal curve, $\gamma(a)=\gamma(b)$, then \eqref{eq60} along with Green's formula that is true for absolutely continuous curves (see, \eqref{eq187}), give
$$
\sum_{j=1}^n\int_{\R^2} w(\gamma_j,y)\, dy=0.
$$

That means, the sum of the oriented areas enclosed by the curves $\gamma_j$ equals zero.
\begin{definition}
\label{d6}
With an absolutely continuous curve $\eta:[a,b]\to\R^2$, that is not necessarily closed, we associate a closed curve $\bar{\eta}:[a-1,b+1]\to\R^2$ defined by
$$
\bar{\eta}(s)=
\begin{cases}
(s-a+1)\eta(a) & \text{if $a-1\leq s\leq a$,}\\
\eta(s)        & \text{if $a\leq s\leq b$,}\\
(b+1-s)\eta(b) & \text{if $b\leq s\leq b+1$.}
\end{cases}
$$
That is, we first follow along a straight line from the origin $0$ to $\eta(a)$, then, we follow the curve $\eta$ and then, we follow a straight line from $\eta(b)$ back to the origin $0$. 
\end{definition}
It is easy to see that
\begin{equation}
\label{eq115}
\frac{1}{2}\int_\eta x\, dy-y\, dx =
\frac{1}{2}\int_{\bar{\eta}} x\, dy-y\, dx=
\int_{\R^2} w(\bar{\eta}, y)\, dy
\end{equation}
equals the oriented area enclosed by the closed curve $\bar{\eta}$. This generalizes Green's formula to the case of non-closed curves.

Now, if $\gamma:[a,b]\to\R^{2n+1}$ is a horizontal curve that is not necessarily closed, then \eqref{eq60} and \eqref{eq115} yield a formula for the change of height
\begin{equation}
\label{eq62}    
\gamma^t(b)-\gamma^t(a)=-4\sum_{j=1}^n \int_{\R^2} w(\bar{\gamma}_j,y)\, dy.
\end{equation}

While the above results are straightforward consequences of the definition of a horizontal curve, we will see later that most of these results generalize to the case of $\alpha$-H\"older continuous curves $\gamma:[a,b]\to\bbbh^n$, provided $\alpha>1/2$.

Every Lie group can be equipped with a left invariant Riemannian metric. We equip the Heisenberg
group with the left invariant metric such that the vectors
$X_1(p), Y_1(p),\ldots,X_n(p),Y_n(p),T(p)$ are orthonormal at every point $p\in\bbbh^n$.
We denote this left invariant metric on $\bbbh^n$ by $g$.

\begin{lemma}
\label{T39}
Every two points in $\bbbh^n$ can be connected by a horizontal curve.
\end{lemma}
Indeed, it is not difficult to construct directly a piecewise smooth horizontal curve connecting any two points, or one can use an explicit construction of a geodesic. For a construction of a geodesic see e.g. \cite{HajlZ} and references therein.

\begin{definition}
Let $\gamma:[a,b] \to \bbbh^n$ be a horizontal curve and write
$$
\gamma'(t) = \sum_{j=1}^n a_j(t)X_j(\gamma(t)) + b_j(t)Y_j(\gamma(t)) \quad \text{ a.e. } t\in [a,b].
$$
The \emph{horizontal length} of $\gamma$
is defined as the length of $\gamma$ with  respect to the metric $g$
$$
\ell_H(\gamma) := \int_a^b |\gamma'(t)|_H \, dt
$$
where
$$
|\gamma'(t)|_H = \sqrt{\langle \gamma'(t),\gamma'(t) \rangle_g} = \sqrt{\sum_{j=1}^n (a_j(t))^2 + (b_j(t))^2}\, .
$$
\end{definition}
\begin{definition}
The \emph{Carnot-Carath\'{e}odory metric} in $\bbbh^n$
is defined for all $p,q\in \bbbh^n$ by
$$
d_{cc}(p,q) = \inf_{\gamma} \{ \ell_H(\gamma) \},
$$
where the infimum is taken over all horizontal curves
connecting $p$ and $q$.
\end{definition}
It follows Lemma~\ref{T39} that $d_{cc}(p,q)<\infty$ for all $p,q\in\bbbh^n$ and hence $d_{cc}$ is a metric in $\bbbh^n$.

The Carnot-Carath\'eodory metric is very far from being Euclidean. In some directions it is Euclidean, but in some other directions it is the square root of the Euclidean metric. In general, we have the following well known comparison of the metrics:
\begin{lemma}
\label{T40}
If $K \subset \bbbr^{2n+1}$ is compact,
then there is a constant $C = C(K) \geq 1$
such that
$$
C^{-1}|p-q| \leq d_{cc}(p,q) \leq C|p-q|^{1/2} 
\quad 
\text{for all $p,q \in K$.}
$$
\end{lemma}
The lower bound directly follows from the definition of the $d_{cc}$ metric, while the upper bound can be concluded from Lemma~\ref{T41} and~\eqref{eq28}.

While the metric $d_{cc}$ is difficult to evaluate, there is a bi-Lipschitz equivalent metric that can be computed explicitly.

\begin{definition}
For any $(z,t) \in \bbbh^n$, the \emph{Kor\'{a}nyi norm} $\Vert \cdot \Vert_K$ is
$$
\Vert (z,t) \Vert_K := (|z|^4+t^2)^{1/4}
$$
and the \emph{Kor\'{a}nyi metric} $d_K$ is
$$
d_K(p,q) := \Vert q^{-1} * p \Vert_K.
$$
\end{definition}
While it is non-obvious, one can prove that 
$d_K$ satisfies the triangle inequality and therefore it is a metric, see \cite[p.~320]{KR}.
Moreover, it is bi-Lipschitz equivalent to the Carnot-Carath\'eodory metric,  
\cite[Section~2.1]{CDPT}:
\begin{lemma}
\label{T41}
$d_K$ is a metric. Moreover, the metrics
$d_{cc}$ and $d_K$ are bi-Lipschitz equivalent.
That is, there is a constant $C \geq 1$ such that
$$
C^{-1}d_K(p,q) \leq d_{cc}(p,q) \leq Cd_K(p,q) 
\quad 
\text{for all $p,q\in \bbbh^n$.}
$$
\end{lemma}
In what follows, we will regard $\bbbh^n$ as a metric space with the Kor\'anyi metric $d_K$. 

One may easily check that for $p=(z,t)$ and $q=(z',t')$ we have
\begin{equation}
\begin{split}
\label{eq28}
d_K(p,q) = \Vert q^{-1}*p \Vert_K &= \left( |z-z'|^4 + \Bigg|t-t'+2\sum_{j=1}^n (x_j'y_j - x_jy_j')\Bigg|^2\, \right)^{1/4}\\
&\approx |z-z'|+\Bigg|t-t'+2\sum_{j=1}^n (x_j'y_j - x_jy_j')\Bigg|^{1/2}\, .
\end{split}
\end{equation}
It will also be convenient to use notation
\begin{equation}
\label{eq32}
d_K(p,q)=(|\pi(p)-\pi(q)|^4+|\varphi(p,q)|^2)^{1/4},
\end{equation}
where $\pi:\R^{2n+1}\to\R^{2n}$, $\pi(z,t)=z$ is the orthogonal projection and
\begin{equation}
\label{eq42}
\varphi(p,q)=t-t'+2\sum_{j=1}^n (x_j'y_j - x_jy_j').
\end{equation}
The Heisenberg group is equipped with dilations $\delta_r:\bbbh^n\to\bbbh^n$, $r\geq 0$, defined by
\begin{equation}
\label{eq147}
\delta_r(x,t)=(rx,r^2t).
\end{equation}
For $r>0$, the dilationas are group authomorphisms, but this fact will play no role here.
The following obvious observation is very useful
\begin{equation}
\label{eq142}
d_K(\delta_r p,\delta_r q)=rd_K(p,q)
\quad
\text{for all } p,q\in\bbbh^n \text{ and } r\geq 0.
\end{equation}

Note that $f\in C^{0,\gamma}(\Omega;\bbbh^n)$, $\Omega\subset\R^m$, means that 
\begin{equation}
\label{eq41}
[f]_{\gamma}=
\sup\left\{\frac{d_K(f(x),f(y))}{|x-y|^\gamma}:\, x,y\in\Omega,\ x\neq y\right\}<\infty.
\end{equation}
Recall that $C_{\rm b}^{0,\gamma}(\Omega;\bbbh^n)$ stands for a subspace of $C^{0,\gamma}(\Omega;\bbbh^n)$ that consists of mappings that are bounded as mappings into $\R^{2n+1}$.

If $f:\Omega\to\bbbh^n$, $\Omega\subset\R^m$, is Lipschitz continuous, then by Lemma~\ref{T40}, $f:\Omega\to\R^{2n+1}$ is locally Lipschitz continuous and hence differentiable a.e. It turns out that at every point $p$ of differentiability of $f$ (and hence almost everywhere), the derivative of $f$ maps the tangent space $T_p\R^m$ to the horizontal space  $H_{f(p)}\bbbh^n$, which is equivalent to the fact that $f^*\upalpha(p)=0$.
This is a consequence of the following more general result, see \cite[Proposition~8.1]{BHW}.
\begin{proposition}
\label{T42}
Suppose that a mapping
$f=(f^{x_1},f^{y_1},\ldots,f^{x_n},f^{y_n},f^t):\Omega\to\bbbh^n$, defined on an open set $\Omega\subset\R^m$
is of class $C^{0,\frac{1}{2}+}_{\rm loc}(\Omega;\bbbh^n)$. If the components
$f^{x_j}, f^{y_j}$ are differentiable at $p_0\in\Omega$ for $j=1,2,\ldots,n$, then
the last component $f^t$ is also differentiable at $p_0$ and
$$
Df^t(p_0) = 2\sum_{j=1}^n
\left( f^{y_j}(p_0)Df^{x_j}(p_0)-f^{x_j}(p_0)Df^{y_j}(p_0)\right)\, ,
$$
i.e. the image of the derivative $Df(p_0)$ lies in the horizontal space $H_{f(p_0)}\bbbh^n$.
\end{proposition}
\begin{remark}
Recall that the class $C^{0,\frac{1}{2}+}$ was defined in Definition~\ref{d1}.
\end{remark}
\begin{proof}
From the assumptions about $f$, there is a 
modulus of continuity $\omega$ such that
for all $p$ in a neighborhood of $p_0$ we have
$d_{\rm{K}}(f(p),f(p_0))\leq |p-p_0|^{1/2}\omega(|p-p_0|)$.
Indeed, if $E$ is a compact closure of a neighborhood of $p_0$ and 
$$
\omega_1(t):=\sup\Bigg\{\frac{d_K(f(p),f(q))}{|p-q|^{1/2}}:\, p,q\in E,\ 0<|p-q|\leq t\Bigg\},
$$
then existence of a modulus of continuity $\omega$ follows from Lemma~\ref{T98}.

Therefore, \eqref{eq28} yields
\[
\begin{split}
&\left|
f^t(p)-f^t(p_0)-2\sum_{j=1}^n \big(f^{y_j}(p_0)(f^{x_j}(p)-f^{x_j}(p_0))-f^{x_j}(p_0)(f^{y_j}(p)-f^{y_j}(p_0))\big)
\right|\\
&=
\left|f^t(p)-f^t(p_0) + 2\sum_{j=1}^n
\big( f^{x_j}(p_0)f^{y_j}(p) - f^{x_j}(p)f^{y_j}(p_0)\big)\right|\\ 
& \leq 
d_K(f(p),f(p_0))^2\leq
|p-p_0|\omega^2(|p-p_0|),
\end{split}
\]
and hence
\begin{eqnarray*}
\lefteqn{\left| f^t(p)-f^t(p_0) - 
2\sum_{j=1}^n \big(f^{y_j}(p_0)Df^{x_j}(p_0) - f^{x_j}(p_0)Df^{y_j}(p_0)\big)(p-p_0)\right|} \\
& \leq &
|p-p_0|\omega^2(|p-p_0|) \\ 
& + &
2\sum_{j=1}^n |f^{y_j}(p_0)|\, |f^{x_j}(p)-f^{x_j}(p_0)-Df^{x_j}(p_0)(p-p_0)|\\
& + &
2\sum_{j=1}^n |f^{x_j}(p_0)|\, |f^{y_j}(p)-f^{y_j}(p_0)-Df^{y_j}(p_0)(p-p_0)|\\
& = &
o(|p-p_0|),
\end{eqnarray*}
because the functions $f^{x_j}$ and $f^{y_j}$ are differentiable at $p_0$. The proof is complete.
\end{proof}
It follows from Proposition~\ref{T42} and the Rademacher theorem, that at almost every point, the derivative of a Lipschitz map $f:\Omega\to\bbbh^n$, $\Omega\subset\bbbr^m$, maps the tangent space to a horizontal subspace, and hence $\operatorname{rank} Df\leq 2n$ almost everywhere. However, a stronger result is true.
\begin{proposition}
\label{T110}
If $f:\Omega\to\bbbh^n$, $\Omega\subset\R^m$ is Lipschitz continuous, then $\operatorname{rank} Df\leq n$ a.e.
\end{proposition}
\begin{proof}
Components of $f$ are locally Lipschitz continuous and
$df^t=2\sum_{j=1}^n(f^{y_j}df^{x_j}-f^{x_j}df^{y_j})$.
Taking the distributional exterior derivative yields
$\sum_{j=1}^n df^{x_j}\wedge df^{y_j}=0$.
That means $F^*\upomega=0$, where $\upomega$ is the standard symplectic form and
$F=(f^{x_1}, f^{y_1},\ldots,f^{x_n},f^{y_n}):\Omega\to\R^{2n}$.
Thus $DF$ maps $T_p\R^m$ into an isotropic subspace of $T_{f(p)}\R^{2n}$ so $\rank DF\leq n$ by Lemma~\ref{T62}. Since $Df^t$ is a linear combination of $Df^{x_j}$ and $Df^{y_j}$, $j=1,2,\ldots,n$, it follows that $\rank Df\leq n$.
\end{proof}

\section{H\"older continuous mappings into the Heisenberg group}
\label{HM}
\subsection{A characterization of H\"older continuous mappings into \texorpdfstring{$\bbbh^n$}{Hn}.}
\label{HM1}
A smooth mapping, $f:\Omega\to\R^{2n+1}$, defined on an open set $\Omega\subset\R^m$, is horizontal as a mapping into $\bbbh^n$, if and only if $f^*\upalpha=0$. The next result, generalizes this characterization to the case of H\"older continuous mappings.

Recall that $\pi:\R^{2n+1}\to\R^{2n}$, $\pi(z,t)=z$ is the projection onto the first $2n$ components.

Note that if $f:\Omega\to\R^{2n+1}$ is $\gamma$-H\"older continuous, then by Lemma~\ref{T40}, it is only (locally) $\gamma/2$-H\"older continuous as a mapping into $\bbbh^n$. In fact, $\gamma/2$-H\"older continuity is the best estimate we can get: the identity map $\operatorname{id}:\R^{2n+1}\to\R^{2n+1}$ is Lipschitz continuous ($\gamma=1$), but $\operatorname{id}:\R^{2n+1}\to\bbbh^n$ is only locally $1/2$-H\"older continuous, see also \cite[Proposition~3.1]{HMR}.

The next result characterizes maps $f:\Omega\to\R^{2n+1}$ that are $\gamma$-H\"older continuous as maps into $\bbbh^n$.
\begin{theorem}
\label{T64}
Let $f:\Omega\to\bbbh^n$ be a continuous map defined on an open set $\Omega\subset\R^m$, and let $1/2<\gamma\leq 1$. Then $f\in C^{0,\gamma}_{\rm loc}(\Omega;\bbbh^n)$ if and only if both of the conditions are satisfied
\begin{itemize}
\item[(a)] $\pi\circ f\in C_{\rm loc}^{0,\gamma}(\Omega;\R^{2n})$,
\item[(b)] $f^*_\eps\upalpha\rightrightarrows 0$ uniformly as $\eps\to 0^+$ on compact subsets of $\Omega$, where $f_\eps=f*\eta_\eps$ is a standard $\eps$-mollification of $f$ (Definition~\ref{d2}).
\end{itemize}
\end{theorem}
Both of the implications in Theorem~\ref{T64} will be proved as separate results. In fact these results will provide stronger estimates. The implication $\Rightarrow$ in Theorem~\ref{T64} follows from Lemmata~\ref{T40} and~\ref{T41} (this proves (a)) and from the following result \cite[Lemma~2.2]{HS23}: 
\begin{theorem}
\label{T45}
Suppose that
$f\in C^{0,\gamma}(\Omega;\bbbh^n)$, 
where  $\Omega\subset\R^m$ is open and
$0<\gamma\leq 1$.  If $B^m(x_o,2r)\subset\Omega$, then
$$
\Vert f^*_\eps\upalpha\Vert_{L^\infty(B^m(x_o,r))}\leq 
4^\gamma\Vert \nabla\eta\Vert_1[f]_{\gamma,2\eps,B^m(x_o,2r)}^2\,\eps^{2\gamma-1}
\quad
\text{for all $0<\eps<r$,}
$$
where $f_\eps=f*\eta_\eps$ is a standard $\eps$-mollification of $f$.
\end{theorem}

The implication $\Leftarrow$ in Theorem~\ref{T64} follows from the next result. 
Note that by Proposition~\ref{T3}, condition (a) in Theorem~\ref{T64} implies that
$$
\sup_{\eps>0}\, [\pi\circ f_\eps]_{C^{0,\gamma}(K;\R^{2n})}<\infty
\quad
\text{if $K\subset\Omega$ is compact.}
$$
\begin{theorem}
\label{T65}
Suppose that $f:\Omega\to\R^{2n+1}$ is a continuous mapping defined on an open set $\Omega\subset\R^{m}$.
Let $1/2<\gamma\leq 1$.
Suppose that $f_i\in C^\infty(\Omega;\R^{2n+1})$ is a sequence of smooth mappings such that for every compact set $K\subset\Omega$:
\begin{itemize}
\item[(a)] $f_i\rightrightarrows f$ and $f_i^*\upalpha\rightrightarrows 0$ uniformly on $K$, 
\item[(b)] 
$
\sup_{i\in\bbbn}\, [\pi\circ f_i]_{C^{0,\gamma}(K;\R^{2n})}<\infty.
$
\end{itemize}
Then $f\in  C_{\rm loc}^{0,\gamma}(\Omega;\bbbh^n)$.
\end{theorem}
\begin{proof}[Proof of Theorem~\ref{T45}]
Recall that $\varphi(p,q)$ was defined in \eqref{eq42}. 
Let $p\in B^m(x_o,r)$ and let $\eps\in (0,r)$.
In what follows we will identify $(f^*_\eps\upalpha)(p)$ with the vector (of equal length):
\begin{equation*}
\begin{split}
&(f^*_\eps\upalpha)(p)=\nabla f^t_\eps(p)+2\sum_{j=1}^n (f_\eps^{x_j}(p)\nabla f_\eps^{y_j}(p)-f_\eps^{y_j}(p)\nabla f_\eps^{x_j}(p))=\\
&
\eps^{-m-1}\int_{B_\eps}
\left( \big(f^t(p-z)-f_\eps^t(p)+
2\sum_{j=1}^n(f_\eps^{x_j}(p)f^{y_j}(p-z)-f_\eps^{y_j}(p)f^{x_j}(p-z)\big)\right)\nabla\eta
\left(\frac{z}{\eps}\right)\, dz\\
&=
\eps^{-m-1}\int_{B_\eps} \varphi(f(p-z),f_\eps(p))\nabla\eta\left(\frac{z}{\eps}\right)\, dz\\
&=
\eps^{-m-1}\int_{B_\eps}\int_{B_\eps} \varphi(f(p-z),f(p-w))\eta_\eps(w)\nabla\eta\left(\frac{z}{\eps}\right)\, dw\, dz.
\end{split}    
\end{equation*}
In the second equality we used the fact that
$\int_{B_\eps}f^t_\eps(p)\nabla\eta\left(\frac{z}{\eps}\right)\, dz=0$.
Easy verification of the last equality is left to the reader.
Note that \eqref{eq32} yields
$$
|\varphi(f(p-z),f(p-w))|\leq d_K(f(p-z),f(p-w))^2\leq [f]_{\gamma,2\eps,B^m(x_o,2r)}^2(2\eps)^{2\gamma}.
$$
Therefore,
$$
|(f_\eps^*\upalpha)(p)|\leq 
[f]_{\gamma,2\eps,B^m(x_o,2r)}^2(2\eps)^{2\gamma}\eps^{-1}\eps^{-m}\int_{B_\eps}\left|\nabla\eta\left(\frac{z}{\eps}\right)\right|\, dz=
4^\gamma\Vert\nabla\eta\Vert_1[f]^2_{\gamma,2\eps,B^m(x_o,2r)}\eps^{2\gamma-1}.
$$
The proof is complete.
\end{proof}

\begin{proof}[Proof of Theorem~\ref{T65}]
Given a closed ball $\bar{B}\Subset\Omega$, it suffices to prove that there is $C>0$ 
such that for any $p,q\in\bar{B}$, $p\neq q$, there is 
$i_{pq}\in\bbbn$ such that
\begin{equation}
\label{eq43}
d_K(f_i(q),f_i(p))\leq C|q-p|^\gamma
\quad
\text{for all $i\geq i_{pq}$.}
\end{equation}
Note that in estimate \eqref{eq43}, $i_{pq}$ is allowed to depend on the choice of $p$ and $q$. However, the constant $C$ is independent of $p,q\in\bar{B}$, but it depends on $\bar{B}$.

Once we prove \eqref{eq43}, the $\gamma$-H\"older continuity of $f:\bar{B}\to \bbbh^n$ will follow upon passing to the limit as $i\to\infty$.

Recall that (see \eqref{eq32})
\begin{equation}
\label{eq45}
d_K(f_i(q),f_i(p))\approx |\pi(f_i(q))-\pi(f_i(p))|+
|\varphi(f_i(q),f_i(p))|^{1/2}.
\end{equation}
Clearly,
\begin{equation}
\label{eq46}
|\pi(f_i(q))-\pi(f_i(p))|\leq  C|q-p|^\gamma
\quad
\text{for all $p,q\in\bar{B}$ and all $i$},
\end{equation}
by the assumption (b).
It remains to estimate $|\varphi(f_i(q),f_i(p))|$. With the notation
$$
f_i=(f_i^{x_1},f_i^{y_1},\ldots,f_i^{x_n},f_i^{y_n},f_i^t)
\quad
\text{and}
\quad
p_s=p+s(q-p),
$$
we have
\begin{equation*}
\begin{split}
&\varphi(f_i(q),f_i(p)) 
=
f_i^t(q)-f_i^t(p)+2\sum_{j=1}^n\left(f_i^{x_j}(p)f_i^{y_j}(q)-f_i^{x_j}(q)f_i^{y_j}(p)\right)\\
&=
f_i^t(q)-f_i^t(p)+2\sum_{j=1}^n \left(f_i^{x_j}(p)(f_i^{y_j}(q)-f_i^{y_j}(p))-f_i^{y_j}(p)(f_i^{x_j}(q)-f_i^{x_j}(p))\right)\\
&=
\int_0^1 df_i^t(p_s)(q-p)\, ds+
2\sum_{j=1}^n \int_0^1 (f_i^{x_j}(p)df_i^{y_j}(p_s)-f_i^{y_j}(p)df_i^{x_j}(p_s))(q-p)\, ds=:\heartsuit.
\end{split}    
\end{equation*}
Since
$$
(f_i^*\upalpha)(p_s)=df_i^t(p_s)+2\sum_{j=1}^n (f_i^{x_j}(p_s)df_i^{y_j}(p_s)-f_i^{y_j}(p_s)df_i^{x_j}(p_s)),
$$
we have
\begin{equation*}
\begin{split}
\heartsuit
&=
\int_0^1 (f_i^*\upalpha)(p_s)(q-p)\, ds\\ 
&+
2\sum_{j=1}^n\int_0^1\big((f_i^{x_j}(p)-f_i^{x_j}(p_s))df_i^{y_j}(p_s)-(f^{y_j}_i(p)-f_i^{y_j}(p_s))df_i^{x_j}(p_s)\big)(q-p)\, ds.
\end{split}
\end{equation*}
Therefore,
\begin{equation}
\label{eq68}
\begin{split}
|\varphi(f_i(q),f_i(p))|
&\leq
\sup_{\bar{B}} |f_i^*\upalpha|\, |q-p|
+
2\sum_{j=1}^n\Bigg|\int_0^1 (f_i^{x_j}(p)-f_i^{x_j}(p_s))\frac{d}{ds}f_i^{y_j}(p_s)\, ds\Bigg|\\
&+
2\sum_{j=1}^n\Bigg|\int_0^1 (f_i^{y_j}(p)-f_i^{y_j}(p_s))\frac{d}{ds}f_i^{x_j}(p_s)\, ds\Bigg|.
\end{split}    
\end{equation}
The integrals in the last two sums can be estimated with the help of inequality \eqref{eq113} in Corollary~\ref{T75}.
To this end, we need to estimate the semi-norm $[\,\cdot\,]_{\gamma}$ i.e., the constant in the $\gamma$-H\"older continuity of
\begin{equation}
\label{eq44}
[0,1]\ni s\mapsto f_i^{x_j}(p_s)
\quad
\text{and}
\quad
[0,1]\ni s\mapsto f_i^{y_j}(p_s).
\end{equation}
Using \eqref{eq46}, we obtain that 
$$
|f_i^{x_j}(p_{s_1})-f_i^{x_j}(p_{s_2})|\lesssim 
|p_{s_1}-p_{s_2}|^\gamma=|s_1-s_2|^\gamma|q-p|^\gamma
\quad
\text{so}
\quad[f_i^{x_j}(p_s)]_{\gamma}\lesssim |q-p|^\gamma.
$$
Similar estimate holds for $f_i^{y_j}$. Since $\gamma>1/2$,  \eqref{eq113} yields 
\begin{equation}
\label{eq47}
|\varphi(f_i(q),f_i(p))|\leq
\sup_{\bar{B}} |f_i^*\upalpha|\, |q-p| +C|q-p|^{2\gamma}.
\end{equation}
Applying uniform convergence $f_i^*\upalpha\rightrightarrows 0$ on $\bar{B}$, 
given $p,q\in\bar{B}$, $p\neq q$, we can find $i_{pq}$ such that
\begin{equation}
\label{eq48}
\sup_{\bar{B}} |f_i^*\upalpha|\leq |q-p|^{2\gamma-1}
\quad
\text{for $i\geq i_{pq}$.}
\end{equation}
Note that $i_{pq}$ depends on the choice of $p$ and $q$.
In concert, \eqref{eq45}, \eqref{eq46}, \eqref{eq47} and \eqref{eq48} complete the proof of \eqref{eq43}.
\end{proof}

The next result, Theorem~\ref{T99}, characterizes H\"older continuous maps into $\bbbh^n$ as H\"older continuous maps into $\R^{2n+1}$ for which the pullback of the contact form $f^*\upalpha=0$ equals zero in the distributional sense.
We will precede the statement with some constructions needed in the proof.

Let $\kappa\in\Omega^k(\R^N)$, $k\in\{1,2,\ldots,m\}$ be a smooth $k$-form. If $\Omega\subset\R^m$ is open and $f\in C_{\rm loc}^{0,\gamma}(\Omega;\R^N)$, $\frac{k}{k+1}<\gamma\leq 1$, then the pullback $f^*\kappa$ is a well defined distribution (called also a $k$-current)
\begin{equation}
\label{eq130}
\langle f^*\kappa,\phi\rangle=\int_\Omega f^*\kappa\wedge\phi
\quad
\text{acting on } \phi\in \Omega_c^{m-k}(\Omega),
\end{equation}
see Corollary~\ref{T100}. In particular if $f\in C^{0,\gamma}_{\rm loc}(\Omega;\R^{2n+1})$, $\frac{1}{2}<\gamma\leq 1$, then $f^*\upalpha$ is a well defined distribution acting on $\Omega_c^{m-1}(\Omega)$.

If $f\in C^{0,\gamma}_{\rm loc}(\Omega;\bbbh^n)$, $\frac{1}{2}<\gamma\leq 1$, then 
$f\in C^{0,\gamma}_{\rm loc}(\Omega;\R^{2n+1})$ and $f^*\upalpha=0$ in the distributional sense \eqref{eq130}, because of Theorem~\ref{T64}(b).
The main objective of Theorem~\ref{T99} is the converse implication that $f^*\upalpha=0$ (which is seemingly weaker than the condition (b) in Theorem~\ref{T64}) along with the H\"older continuity into $\R^{2n+1}$, implies that $f\in C^{0,\gamma}_{\rm loc}(\Omega;\bbbh^n)$.

We also need to explain the notion of approximation of the distribution $f^*\upalpha$ by mollification.
If $\tau\in\Omega^1(\R^m)$, $\tau(p)=\sum_{j=1}^m\tau_j(p)\, dp_j$, then we define
$$
\tau_\eps(p):=\sum_{j=1}^m(\tau_j*\eta_\eps)(p)\, dp_j.
$$
It is easy to check that 
\[
\begin{split}
\tau_\eps(p)
&=
\sum_{j=1}^m (-1)^{j-1}\bigg(\int_{\R^m}
\eta_{\eps}(p-z)\tau(z)\wedge dz_1\wedge\ldots\wedge\widehat{dz_j}\wedge\ldots\wedge dz_m\bigg)\, dp_j\\
&=
\sum_{j=1}^m \bigg(\int_{\R^m}
\eta_{\eps}(p-z)\tau(z)\wedge (*dz_j)\bigg)\, dp_j.
\end{split} 
\]
We used here the Hodge star operator to abbreviate notation
$$
(-1)^{j-1}dz_1\wedge\ldots\wedge\widehat{dz_j}\wedge\ldots\wedge dz_m=*dz_j.
$$
Therefore, if $f\in C^{0,\gamma}_{\rm loc}(\Omega;\R^N)$, $\frac{1}{2}<\gamma\leq 1$ and $\kappa\in \Omega^1(\R^N)$, we define
\begin{equation}
\label{eq131}
(f^*\kappa)_\eps(p):=\sum_{j=1}^m \bigg(\int_{\R^m}
\eta_{\eps}(p-z)(f^*\kappa)(z)\wedge (*dz_j)\bigg)\, dp_j.
\end{equation}
Here, the integral in understood in the distributional sense and is defined through the approximations $f_\delta=f*\eta_\delta\to f$ as $\delta\to 0$.
Clearly, \eqref{eq131} is well defined for
$$
p\in \Omega_\eps:=\{p\in\Omega:\, \dist(p,\partial\Omega)>\eps\},
$$
because this guarantees that
$$
z\mapsto\eta_\eps(p-z)(*dz_j)=
(-1)^{j-1}\eta_\eps(p-z)dz_1\wedge\ldots\wedge\widehat{dz_j}\wedge\ldots\wedge dz_m\in \Omega_c^{m-1}(\Omega).
$$
Note that in the case of the contact form we have
$$
(f^*\upalpha)_\eps(p)=
df^t_\eps(p)+2\sum_{i=1}^n \big(f^{x_i}df^{y_i}-f^{y_i}df^{x_i}\big)_\eps(p),
$$
because $(df^t)_\eps=df^t_\eps$. Indeed, 
\[
\begin{split}
(df^t)_\eps(p)
&=
\lim_{\delta\to 0} \sum_{j=1}^m \bigg(\int_{\R^m}
\eta_{\eps}(p-z)d(f^t*\eta_\delta)(z)\wedge (*dz_j)\bigg)\, dp_j.\\
&=
\lim_{\delta\to 0} \sum_{j=1}^m \partial_j \big(f^t*\eta_\delta\big)*\eta_\eps(p)\, dp_j=
\lim_{\delta\to 0} \sum_{j=1}^m \partial_j \big(f^t*\eta_\eps\big)*\eta_\delta(p)\, dp_j=
df_\eps^t(p).
\end{split}
\]
The proof of the following result employs arguments similar to those used in the proof of easier (affirmative) direction of the Onsager conjecture \cite{CET94}.
\begin{theorem}
\label{T99}
Let $\frac{1}{2}<\gamma\leq 1$, let $\Omega\subset\R^m$ be open, and let $f\in C^{0,\gamma}_{\rm loc}(\Omega;\R^{2n+1})$. Then $f\in C^{0,\gamma}_{\rm loc}(\Omega;\bbbh^n)$ if and only if $f^*\upalpha=0$ in the distributional sense i.e., 
$$
\langle f^*\upalpha,\phi\rangle = \int_\Omega f^*\upalpha\wedge\phi=0
\quad
\text{for all } \phi\in\Omega_c^{m-1}(\Omega).
$$
\end{theorem}

\begin{remark}
Let us stress that for H\"older maps, if we replace $f^*\upalpha=0$ in the distributional sense, by the condition $f^*\upalpha=0$ in the pointwise a.e. sense, we get nothing. Indeed, for any $0<\gamma<1$, Salem \cite{salem}, constructed a surjective homeomorphism $g\in C^{0,\gamma}(\R;\R)$ that is differentiable a.e. and satisfies $g'=0$ a.e. Thus,
$$
f(x_1,y_1,\ldots,x_n,y_n,t)=(g(x_1),g(y_1),\ldots,g(x_n),g(y_n),g(t)),
\quad
f\in C^{0,\gamma}(\R^{2n+1};\R^{2n+1})
$$
is a surjective homeomorphism that is differentiable a.e. and satisfies $Df=0$ a.e. Hence, $f^*\upalpha=0$ a.e. On the other hand $f\not\in C^{0,\frac{1}{2}}_{\rm loc}(\R^{2n+1};\bbbh^n)$, because $g$ is not locally Lipschitz continuous and $d_K\big((z,t),(z,t')\big)=|t-t'|^{1/2}$. 

Another, more complicated construction \cite[Theorem~3.2]{HMR}, gives a continuous function $u:\R^2\to\R$ that is differentiable a.e., and such that the tangent spaces to the graph of $u$ (the tangent spaces exist a.e.) are horizontal a.e. That is, the parametrization of the graph $f(x,y):=(x,y,u(x,y))$ satisfies $f^*\upalpha=0$ a.e. Moreover, $u$ can have modulus of continuity arbitrarily close to the Lipschitz one.

See also \cite{CDLH,CHT,Cerny,FMCO18,Hencl,LiuM} for rather difficult constructions of homeomorphisms with Sobolev regularity that have derivatives of low rank. Some of these examples were obtained using convex integration.
\end{remark}

\begin{proof}[Proof of \Cref{T99}]
The implication $\Rightarrow$, as previously explained, follows straightforwardly from Theorem~\ref{T64}(b). Therefore, it remains to prove the implications $\Leftarrow$.
In view of Theorem~\ref{T64} it suffices to show that
$f^*_\eps\upalpha=(f*\eta_\eps)^*\upalpha\rightrightarrows 0$ as $\eps\to 0$, uniformly on every
$\Omega'\Subset\Omega$.

Fix $\Omega'\Subset\Omega$ and take $0<\eps<\frac{1}{2}\dist(\Omega',\partial\Omega)$. We can
find $\Omega'\Subset\Omega''\Subset\Omega$ such that $\dist(\Omega',\partial\Omega'')>\eps$.
We fix such an $\Omega''$.

Since $f^*\upalpha=0$, it follows that $(f^*\upalpha)_\eps(p)=0$ for all $p\in\Omega'$. Therefore, for $p\in\Omega'$ we have
\begin{equation}
\label{eq132}
\begin{split}
&
(f_\eps^*\upalpha)(p)
= 
(f_\eps^*\upalpha)(p)- (f^*\upalpha)_\eps(p)\\
&=
2\sum_{i=1}^n
\big(f_\eps^{x_i}(p)df_{\eps}^{y_i}(p)-(f^{x_i}df^{y_i})_\eps(p)\big)
-
2\sum_{i=1}^n
\big(f_\eps^{y_i}(p)df_{\eps}^{x_i}(p)-(f^{y_i}df^{x_i})_\eps(p)\big).
\end{split}
\end{equation}
Note that the terms involving $f^t$ canceled out, because $(df^t)_\eps(p)=df_\eps^t(p)$. 
We will estimate the terms in the first sum only, because the estimates of the second sum are essentially the same. We have
\[
\begin{split}
&
f_\eps^{x_i}(p)df_{\eps}^{y_i}(p)-(f^{x_i}df^{y_i})_\eps(p)\\
&=
\big((f_\eps^{x_i}(p)-f^{x_i}(p))df_{\eps}^{y_i}(p)\big)+
\big((f^{x_i}(p)df_{\eps}^{y_i}(p)-(f^{x_i}df^{y_i})_\eps(p)\big)=A(p)+B(p).
\end{split}
\]
The following estimate of $A(p)$ easily follows from the definition of convolution (see \eqref{eq129} and \eqref{eq5})
\[
\begin{split}
|A(p)|
&=
|f_\eps^{x_i}(p)-f^{x_i}(p)|\, |df_{\eps}^{y_i}(p)|\\
&\leq
\eps^\gamma [f^{x_i}]_{\gamma,B(p,\eps)}
\cdot
\eps^{\gamma-1}[f^{y_i}]_{\gamma,B(p,\eps)}\Vert\nabla\eta\Vert_1
\lesssim
\eps^{2\gamma-1}[f]^2_{\gamma,B(p,\eps)}.
\end{split}
\]
It remains to estimate $B(p)$. We have
\[
\begin{split}
&
B(p)=
f^{x_i}(p)df_{\eps}^{y_i}(p)-(f^{x_i}df^{y_i})_\eps(p)
=
f^{x_i}(p)(df^{y_i})_\eps(p)-(f^{x_i}df^{y_i})_\eps(p)\\
&=
\sum_{j=1}^m
\bigg(\int_{B(p,\eps)}
\eta_\eps(p-z)
\big(f^{x_i}(p)df^{y_i}(z)-f^{x_i}(z)df^{y_i}(z)\big)
\wedge (*dz_j)\bigg)\, dp_j.
\end{split}
\]
We could restrict this (generalized, distributional) integral to $B(p,\eps)$, because
$\eta_\eps(p-\cdot)$ is supported in $B(p,\eps)$. For simplicity of notation assume that $p=0$. We will apply a change of variables to rescale the integral to $\bbbb^m=B(0,1)$. Note that the (distributional) exterior derivative with respect to $z$ satisfies
$$
df^{y_i}(\eps z)=\eps^{-1}d(f^{y_i}(\eps z)-f^{y_i}(0)). 
$$
Therefore, for $p=0$ we have
$$
|B(0)|\leq
\sum_{j=1}^m \eps^{-1}
\bigg|\int_{\bbbb^m}\underbrace{\big(f^{x_i}(0)-f^{x_1}(\eps z)\big)\eta(-z)(*dz_j)}_{\lambda}
\wedge d\underbrace{\big(f^{y_i}(\eps z)-f^{y_i}(0)\big)}_{\gamma_1}\bigg|
$$
Now Corollary~\ref{T35} with $k=1$, $\ell=m-1$, $\ell_1=1$ and $\alpha=\beta=\gamma$ yields
\[
|B(0)|\lesssim\eps^{-1}\Vert\lambda\Vert_{C^{0,\gamma}}\Vert\gamma_1\Vert_{C^{0,\gamma}}
\lesssim
\eps^{-1}\cdot\eps^\gamma [f^{x_i}]_{\gamma,B_\eps}\cdot\eps^\gamma[f^{y_i}]_{\gamma,B_\eps}
\leq
\eps^{2\gamma-1}[f]^2_{\gamma,B_\eps}.
\]
We used there the fact that if $g(0)=0$, then 
$\Vert g\Vert_{C_{\rm b}^{0,\gamma}(\bbbb^m)}\leq 2[g]_{\gamma,\bbbb^m}$, and we used Lemma~\ref{T27} to take care of $\eta(-z)$ in the estimate of $\lambda$.

For any $p\in\Omega'$ the above estimate reads as
$|B(p)|\lesssim\eps^{2\gamma-1}[f]^2_{\gamma,B(p,\eps)}$.
In concert, the estimates for $A(p)$ and $B(p)$ and similar estimates for the second sum in \eqref{eq132} yield
$$
\big|(f_\eps^*\upalpha)(p)\big|\lesssim
\eps^{2\gamma-1}[f]^2_{\gamma,B(p,\eps)}\leq
\eps^{2\gamma-1}[f]^2_{\gamma,\Omega''}
\quad
\text{whenever } p\in\Omega'.
$$
Therefore $f^*_\eps\rightrightarrows 0$ uniformly on $\Omega'$ as $\eps\to 0$. The proof is complete.
\end{proof}

\subsection{H\"older continuous horizontal curves}
\label{HM2}
If $\gamma$ is an $\alpha$-H\"older continuous curve defined on $[0,1]$ with values in $\R^{2n+1}$ or $\bbbh^{n}$, then we can assume that $\gamma$ is defined on $\R$ since we can extend it to a curve that is constant in $(-\infty,0]$ and in $[1,\infty)$. Clearly, the extended curve is also $\alpha$-H\"older continuous. The extension of $\gamma$ to $\R$ allows us to define the $\eps$-mollification $\gamma_\eps(t)$ for all $t\in [0,1]$.
\begin{theorem}
\label{T67}
Let $\gamma:[0,1]\to\R^{2n+1}$ be a curve and $\alpha\in (\frac{1}{2},1]$. Then $\gamma\in C^{0,\alpha}([0,1];\bbbh^n)$ if and only if
$\pi\circ\gamma\in C^{0,\alpha}([0,1];\R^{2n})$, and
\begin{equation}
\label{eq116}
\gamma^t(b)-\gamma^t(a)=-2\sum_{j=1}^n \int_a^b\gamma^{x_j}d\gamma^{y_j}-\gamma^{y_j}d\gamma^{x_j}
\quad
\text{for all } 0\leq a<b\leq 1,
\end{equation}
where the integral in \eqref{eq116} is understood in the sense of \eqref{eq50}. Moreover, 
$$
d_K(\gamma(b),\gamma(a))\lesssim_{n,\alpha}|b-a|^\alpha[\pi\circ\gamma]_\alpha
\quad\text{for all } 0\leq a<b\leq 1.
$$
\end{theorem}
\begin{remark}
Identity \eqref{eq116} is an extension of \eqref{eq60} beyond the class of differentiable curves. The result completely fails if $\alpha\leq\frac{1}{2}$ since according to \Cref{T40}, any smooth curve $\gamma:[a,b]\to\R^{2n+1}$ is $C^{0,1/2}$-H\"older continuous as a mapping into $\bbbh^n$.
\end{remark}
\begin{remark}
The implication from left to right was proved in \cite[Lemma~3.1]{LZ}.
\end{remark}
\begin{proof}
Assume that $\gamma\in C^{0,\alpha}([0,1];\bbbh^n)$. Clearly, $\pi\circ\gamma\in C^{0,\alpha}([0,1];\R^{2n})$.
Let $\gamma_\eps$ be an $\eps$-mollification of $\gamma$. By \Cref{T64}(b),
$$
\gamma_\eps^*\upalpha=
d\gamma_\eps^t+2\sum_{j=1}^n 
(\gamma_\eps^{x_j}d\gamma_\eps^{y_j}-\gamma_\eps^{y_j}d\gamma_\eps^{x_j})\rightrightarrows 0
\quad
\text{uniformly on $[0,1]$.}
$$
Integrating this form on $[a,b]$, letting $\eps\to 0$ and applying \Cref{T66}, yields \eqref{eq116}.

Suppose now that $\gamma:[0,1]\to \R^{2n+1}$ satisfies
$\pi\circ\gamma\in C^{0,\alpha}$ and \eqref{eq116}. We will prove that $\gamma\in C^{0,\alpha}([0,1];\bbbh^n)$. In order to do so, we need to estimate (cf.\ \eqref{eq32})
$$
d_K(\gamma(b),\gamma(a))\approx
|\pi(\gamma(b))-\pi(\gamma(a))|+|\varphi(\gamma(b),\gamma(a))|^{1/2}.
$$
Clearly,
\begin{equation}
\label{eq117}
|\pi(\gamma(b))-\pi(\gamma(a))|\leq |b-a|^\alpha[\pi\circ\gamma]_\alpha,
\end{equation}
so it remains to estimate 
\begin{equation}
\label{eq118}
\varphi(\gamma(b),\gamma(a))=\gamma^t(b)-\gamma^t(a)+
2\sum_{j=1}^n (\gamma^{x_j}(a)\gamma^{y_j}(b)-\gamma^{x_j}(b)\gamma^{y_j}(a)).
\end{equation}
If we rewrite \eqref{eq116} as
$$
0=\gamma^t(b)-\gamma^t(a)-2\sum_{j=1}^n \int_a^b \gamma^{y_j}d\gamma^{x_j}-\gamma^{x_j}d\gamma^{y_j},
$$
and subtract it from \eqref{eq118}, we have
\begin{equation}
\label{eq119}
\varphi(\gamma(b),\gamma(a))=
2\sum_{j=1}^n \int_a^b (\gamma^{y_j}-\gamma^{y_j}(a))d\gamma^{x_j}-(\gamma^{x_j}-\gamma^{x_j}(a))d\gamma^{y_j},
\end{equation}
and \eqref{eq113} in \Cref{T75} yield
$$
|\varphi(\gamma(b),\gamma(a))|\lesssim_{n,\alpha} |b-a|^{2\alpha}[\pi\circ\gamma]^2_{\alpha},
$$
which together with \eqref{eq117} implies that
$$
d_K(\gamma(b),\gamma(a))\lesssim_{n,\alpha} |b-a|^\alpha[\pi\circ\gamma]_\alpha.
$$
We used here \eqref{eq113} which was proved for Lipschitz functions, but it applies to \eqref{eq119}, because the integral in \eqref{eq119} is defined through smooth approximations.
\end{proof}
\begin{corollary}
\label{T90}
Let $\gamma:[a,b]\to\bbbh^n$ be an $\alpha$-H\"older continuous curve for some $\alpha\in(\frac{1}{2},1]$. Then the change in height along $\gamma$ equals the sum of the oriented areas enclosed by the closed curves $\bar{\gamma}_j$ (see \Cref{d6}):
$$
\gamma^t(b)-\gamma^t(a)=-4\sum_{j=1}^n\int_{\R^2} w(\bar{\gamma}_j,z)\, dz.
$$
In particular, if $\gamma$ is closed i.e., $\gamma(a)=\gamma(b)$, the sum or areas enclosed by the closed curves $\gamma_j$ equals zero
$$
\sum_{j=1}^n \int_{\R^2} w(\gamma_j,z)\, dz=0.
$$
\end{corollary}
\begin{proof}
\Cref{T67} and \Cref{T91} yield (cf. \eqref{eq115} and \eqref{eq62})
\[
\begin{split}
\gamma^t(b)-\gamma^t(a)
&=
-2\sum_{j=1}^n \int_a^b \gamma^{x_j}d\gamma^{y_j}-\gamma^{y_j}d\gamma^{x_j}\\
&=
-2\sum_{j=1}^n \int_{a-1}^{b+1} \bar{\gamma}^{x_j}d\bar{\gamma}^{y_j}-\bar{\gamma}^{y_j}d\bar{\gamma}^{x_j}
=
-4\sum_{j=1}^n\int_{\R^2} w(\bar{\gamma}_j,z)\, dz.
\end{split}
\]
\end{proof}

The next result shows that the $x,y$ components of a H\"older map uniquely determine the map up to a vertical translation. 
\begin{corollary}
\label{T92}
Let $\Omega\subset\R^m$ be a domain and let $f,g\in C^{0,\alpha}_{\rm loc}(\Omega;\bbbh^n)$, $\alpha\in (\frac{1}{2},1]$. If $\pi\circ f=\pi\circ g$ i.e., 
$f^{x_j}=g^{x_j}$ and $f^{y_j}=g^{y_j}$ for $j=1,2,\ldots,n$, then there is a constant $c\in\R$ such that
\begin{equation}
\label{eq120}
g^t(x)=f^t(x)+c
\quad
\text{for all } x\in\Omega.
\end{equation}
\end{corollary}
\begin{proof}
Fix $x_o\in\Omega$. Any $x\in\Omega$ can be connected with $x_o$ by a smooth curve $\eta:[0,1]\to\Omega$, $\eta(0)=x_o$, $\eta(1)=x$. The curves $\gamma_f:=f\circ\eta$ and $\gamma_g:=g\circ\eta$ are $\alpha$-H\"older continuous. Observe that $\gamma_f^{x_j}=\gamma_g^{x_j}$ and $\gamma_f^{y_j}=\gamma_g^{y_j}$ and we will simply write $\gamma^{x_j}$ and $\gamma^{y_j}$. Now, \eqref{eq116} yields
$$
f^t(x)-f^t(x_o)=\gamma_f^t(1)-\gamma_f^t(0)=
-2\sum_{j=1}^n \int_0^1 \gamma^{x_j}d\gamma^{y_j}-\gamma^{y_j}d\gamma^{x_j}
=g^t(x)-g^t(x_o),
$$
so \eqref{eq120} is satisfied with $c:=g^t(x_o)-f^t(x_o)$.
\end{proof}

\subsection{Lifting of H\"older continuous mappings}
\label{HM3}
Let $\Omega\subset\R^m$ be an open domain. A smooth map $F:\Omega\to\R^{2n+1}$ is called {\em Legendrian} or {\em horizontal} if $F^*\upalpha=0$ i.e., if $F$ maps vectors tangent to $\Omega$ to horizontal vectors in $\bbbh^n$. We have seen in \Cref{T42} that every smooth $F$ that belongs to $C^{0,\frac{1}{2}+}_{\rm loc}(\Omega;\bbbh^n)$ is Legendrian.

Let $f:\Omega\to\R^{2n}$, $\Omega\subset \R^m$, be a smooth mapping. A natural question is to find conditions for the existence of a smooth function $\tau:\Omega\to\R$ such that the mapping
$F=(f,\tau):\Omega\to\R^{2n+1}$ is Legendrian. The mapping $F$ is called a \emph{Legendrian} or \emph{horizontal lift} of $f$. 
Note that if a Legendrian lift exists, $\tau$ is uniquely determined up to an additive constant.
As we will see this problem has a very simple solution.

Let
$$
\beta=\sum_{j=1}^n (x_jd y_j-y_jdx_j)
\quad
\text{so}
\quad
\upalpha=dt+2\beta.
$$
Then, the condition $F^*\upalpha=0$ is equivalent to $d\tau+2f^*\beta=0$, so existence of $\tau$ is equivalent to the condition that the from $f^*\beta$ is exact. In particular $f^*\beta$ must be closed, which is equivalent to
$$
f^*\upomega=\frac{1}{2}f^*(d\beta)=\frac{1}{2}df^*\beta=0.
$$
A smooth map $f:\Omega\to\R^{2n}$ satisfying $f^*\upomega=0$ is called \emph{Lagrangian}. Thus, a necessary condition for the existence of a Legendrian lift is that $f$ is Lagrangian. However, if $\Omega$ is simply connected, this condition is also sufficient. We proved
\begin{proposition}
\label{T68}
Let $f:\Omega\to\R^{2n}$ be a smooth map defined on a domain $\Omega\subset\R^m$. Then, the following conditions are equivalent.
\begin{itemize}
\item[(a)] There is a horizontal lift $F=(f,\tau):\Omega\to\R^{2n+1}$ of $f$,
\item[(b)] The form $f^*\beta$ is exact.
\end{itemize}
If in addition the domain $\Omega$ is simply connected, we have one more equivalent condition:
\begin{itemize}
\item[(c)] $f^*\upomega=0$.
\end{itemize}
\end{proposition}
This result is well known and our purpose is to generalize it to the case of H\"older continuous maps $f\in C^{0,\alpha}_{\rm loc}(\Omega;\R^{2n})$, $1/2<\alpha\leq 1$.

Recall that if $f$ is smooth, then the $1$-form $f^*\beta$ is exact if and only if
\begin{equation}
\label{eq63}
\int_\eta f^*\beta=\int_{\Sph^1}\eta^*(f^*\beta)=0
\quad
\text{for every piecewise smooth curve $\eta:\Sph^1\to\Omega$.}
\end{equation}
Indeed, if $f^*\beta=-\frac{1}{2}d\tau$, then $\eta^*(f^*\beta)=-\frac{1}{2}d(\tau\circ\eta)$ so the integral at \eqref{eq63} equals zero. On the other hand equality \eqref{eq63}
shows that if $\eta_1,\eta_2:[0,1]\to\Omega$ are piecewise smooth curves with equal endpoints
$\eta_1(0)=\eta_2(0)$, $\eta_1(1)=\eta_2(1)$, then
$$
\int_{\eta_1} f^*\beta=\int_{\eta_2} f^*\beta,
\quad
\text{and we can define}
\quad
\tau(x):= -2\int_{\eta_x} f^*\beta,
$$
where $\eta_x$ is any piecewise smooth curve connecting a fixed point $x_o\in\Omega$ with $x\in\Omega$. Clearly, $f^*\beta=-\frac{1}{2}d\tau$.

Since we can identify $\eta:\Sph^1\to\Omega$ with a curve $\eta:[0,1]\to\Omega$ such that $\eta(0)=\eta(1)$, the reasoning presented above shows that
condition (b) in \Cref{T68} is equivalent to:

For every piecewise smooth closed curve $\eta:[0,1]\to\Omega$, $\eta(0)=\eta(1)$, we have
$$
\sum_{j=1}^n\int_0^1  \gamma^{x_j}d\gamma^{y_j}-\gamma^{y_j}d\gamma^{x_j}=0,
\quad
\text{where } \gamma=f\circ\eta:[0,1]\to\R^{2n}.
$$
Note that the integrals in the above formula are well defined if $f\in C^{0,\alpha}_{\rm loc}(\Omega;\R^{2n})$, $\alpha\in (\frac{1}{2},1]$, see \Cref{T66}. 
\begin{theorem}
\label{T69}
Let $f\in C^{0,\alpha}_{\rm loc}(\Omega,\R^{2n})$, where $\Omega\subset\R^m$ is an open domain and $\alpha\in(\frac{1}{2},1]$. Then, the following conditions are equivalent
\begin{itemize}
\item[(a)] There is a function $\tau:\Omega\to\R$ such that
$F=(f,\tau)\in C^{0,\alpha}_{\rm loc}(\Omega;\bbbh^n)$,
\item[(b)] For every piecewise smooth closed curve 
$\eta:[0,1]\to \Omega$, $\eta(0)=\eta(1)$,
\begin{equation}
\label{eq65}
\sum_{j=1}^n\int_0^1  \gamma^{x_j}d\gamma^{y_j}-\gamma^{y_j}d\gamma^{x_j}=0,
\quad
\text{where } \gamma=f\circ\eta:[0,1]\to\R^{2n},
\end{equation}
and the integral is understood as in \Cref{T66}.
\item[(c)] For every piecewise smooth closed curve 
$\eta:[0,1]\to \Omega$, $\eta(0)=\eta(1)$,
$$
\sum_{j=1}^n\int_{\R^2}  w(\gamma_j,z)\, dz=0,
\quad
\text{where } \gamma_j=(f^{x_j},f^{y_j})\circ\eta:[0,1]\to\R^{2}.
$$
\end{itemize}
\end{theorem}
\begin{proof}
The result is an easy consequence of \Cref{T67}. 
First observe that the equivalence of conditions (b) and (c) follows from \Cref{T91} and it remains to prove equivalence of (a) and (b).

\noindent 
(a)$\Rightarrow$(b). Since $f$ has a horizontal lift $F=(f,\tau)$, $\Gamma=F\circ\eta$ is a horizontal lift of $\gamma=f\circ\eta$. Since the curve $\Gamma$ is closed, \eqref{eq116} implies \eqref{eq65}.

\noindent 
(b)$\Rightarrow$(a).
Fix $x_o\in \Omega$. Any $x\in\Omega$ can be connected by a piecewise smooth curve $\eta:[0,1]\to\Omega$, $\eta(0)=x_o$, $\eta(1)=x$
and we define 
$$
\tau(x):=-2\sum_{j=1}^n\int_0^1\gamma^{x_j}d\gamma^{y_j}-\gamma^{y_j}d\gamma^{x_j},
\quad
\text{where } \gamma=f\circ\eta:[0,1]\to\Omega.
$$
Condition \eqref{eq65} guarantees that the integral that defines $\tau$ does not depend on the choice of the curve $\eta$. Since $\eta|_{[0,t]}$ connects $x_o$ to $\eta(t)$, it follows that
$$
\tau(\eta(t))=-2\sum_{j=1}^n\int_0^t\gamma^{x_j}d\gamma^{y_j}-\gamma^{y_j}d\gamma^{x_j},
$$
and hence for $0\leq a<b\leq 1$ we have
$$
\tau(\eta(b))-\tau(\eta(a))=-2\sum_{j=1}^n\int_a^b\gamma^{x_j}d\gamma^{y_j}-\gamma^{y_j}d\gamma^{x_j}.
$$
\Cref{T67} implies that the curve $(f,\tau)\circ\eta:[0,1]\to\R^{2n+1}$ belongs to $C^{0,\alpha}([0,1];\bbbh^n)$. The H\"older continuity comes with estimates in terms of H\"older continuity of $f$. 
This true for every piecewise smooth $\eta$. In particular it shows that $(f,\tau)$ is $\alpha$-H\"older continuous on any interval in $\Omega$ and it follows that $(f,\tau)\in C^{0,\alpha}_{\rm loc}(\Omega;\bbbh^n)$. \end{proof}

\subsection{Pullbacks of differential forms by H\"older mappings}
\label{S3}
In \Cref{T8} and \Cref{T46},
$f_\eps=f*\eta_\eps$ stands for the  $\eps$-mollification of $f$ as described in \Cref{d2}.
Recall that the quantities $[f]_{\gamma,\eps,A}$ and $[f]_{\gamma,\eps}$ were defined in \Cref{d1}.
 
Let us start with a simple estimate for a pullback for a $k$-form under an approximation of a H\"older continuous map.
\begin{lemma}
\label{T8}
Let $f\in C^{0,\gamma}(\Omega;\bbbr^d)$, where $\Omega\subset\R^m$ is open, $B(x_o,2r)\subset\Omega$, and $\gamma\in(0,1]$.  
If $\kappa\in \Omega^k L^\infty(\bbbr^d)$, then 
$$
\Vert f_\eps^*\kappa\Vert_{L^\infty(B^m(x_o,r))}\lesssim \Vert\kappa\Vert_{\infty}[f]_{\gamma,\eps,B^m(x_o,2r)}^k \eps^{-k(1-\gamma)}
\quad
\text{for all $0<\eps<r$,}
$$
where the constant in the inequality depends on $m$, $d$, $k$, and $\eta$ only.
\end{lemma}
\begin{proof}
Let $\kappa=\sum_{|I|=k}\psi_I\, dy_I=\sum_{|I|=k}\psi_I\, dy_{i_1}\wedge\ldots\wedge dy_{i_k}$. Then,
$$
f_\eps^*\kappa = \sum_{|I|=k}(\psi_I\circ f_\eps)\, df_\eps^{i_1}\wedge\ldots\wedge df_\eps^{i_k},
\quad
\Vert f_\eps^*\kappa\Vert_\infty\lesssim \Vert\kappa\Vert_\infty\Vert Df_\eps\Vert_\infty^k.
$$
Recall that (see \eqref{eq5}) 
$$
D f_\eps(x) = \eps^{-m-1}\int_{B^m(0,\eps)} (f(x-z)-f(x))D\eta\left(\frac{z}{\eps}\right)\, dz
$$
so for $0<\eps<r$ we have
\begin{equation}
\label{eq163}
\Vert D f_\eps\Vert_{L^\infty(B(x_o,r))}\leq
\eps^{-(1-\gamma)} [f]_{\gamma,\eps,B^m(x_o,2r)}\Vert D\eta\Vert_1.
\end{equation}
and the lemma follows.
\end{proof}
If $f\in C_{\rm b}^{0,\gamma}(\Omega;\bbbh^n)$, $\Omega\subset\R^m$ open, then by Lemmata~\ref{T40} and~\ref{T41} 
\begin{equation}
\label{eq30}
|f(p)-f(q)|\leq C(n,\Vert f\Vert_\infty) d_K(f(p),f(q))\leq C(n,\Vert f\Vert_\infty) [f]_{\gamma}\, |p-q|^\gamma.
\end{equation}
Therefore, \Cref{T8} yields
\begin{corollary}
\label{T46}
Let $f\in C^{0,\gamma}(\Omega;\bbbh^n)$, where $\Omega\subset\R^m$ is open, $B^m(x_o,2r)\subset\Omega$, and $\gamma\in (0,1]$.
If $\kappa\in \Omega^kL^\infty(\bbbr^{2n+1})$, then 
$$
\Vert f_\eps^*\kappa\Vert_{L^\infty(B^m(x_o,r))}\lesssim \Vert\kappa\Vert_{\infty}[f]_{\gamma,\eps,B^m(x_o,2r)}^k \eps^{-k(1-\gamma)}
\quad
\text{for all $0<\eps<r$,}
$$
where the constant in the inequality depends on $k$, $m$, $n$, $\eta$ and $\Vert f\Vert_\infty$ only.
\end{corollary}
For $0<\gamma<1$, the estimate in \Cref{T46} blows up as $\eps\to 0$, but the speed of the blow-up depends the exponent of H\"older continuity: 
the blow-up is slower if $\gamma$ is closer to $1$.
On the other hand, the next result shows that in the case of a pullback of the contact form, we get a much better estimate. In particular if $\gamma>1/2$, there is no blow-up and the estimate converges to zero (see also \Cref{T45}). Thus, is some generalized sense, H\"older continuous mappings with $\gamma>1/2$
preserve horizontality. This fact will play a crucial role in the remaining part of the paper.

\begin{corollary}
\label{T50}
Let $\cM$ be a smooth oriented and compact $m$-dimensional Riemannian manifold with or without boundary. If $f\in C^{0,\gamma}(\cM;\bbbh^n)$, $\gamma\in (0,1]$, then there is a smooth approximation 
$$
C^\infty(\overbar{\cM};\R^{2n+1})\ni f_\eps \xrightarrow{\eps\to 0} f
\quad
\text{in}
\quad
C_{\rm b}^{0,\gamma'}(\cM;\R^{2n+1})\ 
\text{for all $0<\gamma'<\gamma$}
$$
such that
$$
\Vert f_\eps^*\kappa\Vert_\infty \lesssim \eps^{-k(1-\gamma)}[f]^k_{\gamma,\eps}
\quad
\text{for all}\ 
\kappa\in\Omega^k \left(\R^{2n+1}\right),\ 1\leq k\leq 2n+1,
$$
and
$$
\Vert f_\eps^*\upalpha\Vert_\infty \lesssim \eps^{2\gamma-1}[f]^2_{\gamma,\eps},
\quad
\text{when $\upalpha$ is the standard contact form.}
$$
Here, we use notation $[f]_{\gamma,\eps}=[f]_{\gamma,\eps,\cM}$. 
\end{corollary}
\begin{remark}
We do not need to assume that $\kappa$ is bounded (cf.\ \Cref{T46}), because $f$ is bounded and $\kappa$ is bounded in a neighborhood of the image of $f$. Note also that the implied constants in the estimates in \Cref{T50} depend on $\Vert f\Vert_\infty$ 
\end{remark}
\begin{remark}
The construction of $f_\eps$ is similar to the construction of an approximation in \Cref{T71}.
\end{remark}
\begin{proof}
Assume first that $\partial\cM=\varnothing$.
Using the Whitney embedding theorem we may assume that $\cM$ is smoothly embedded into $\R^{2m+1}$. The new Riemannian metric on $\cM$ inherited from the embedding into $\R^{2m+1}$ will be comparable to the original one due to compactness on $\cM$.

According to the tubular neighborhood theorem, there is an open set $\Omega\subset\R^{2m+1}$ containing $\cM$ such that the nearest point projection $\pi:\Omega\to\cM$ is uniquely defined, smooth and Lipschitz. Then
$$
\tilde{f}=f\circ\pi\in C^{0,\gamma}(\Omega,\bbbh^n).
$$
Let $\Omega'\Subset\Omega$ be an open set that contains $\cM$ and let
$\eps_o=\dist(\Omega',\partial\Omega)$. Let $\tilde{f}_\eps$ be the standard approximation by mollification. It is well defined on $\Omega'$ (and hence on $\cM$) for $0<\eps<\eps_o$.
Since $\kappa$ is bounded in a neighborhood of $\tilde{f}(\Omega)=f(\cM)$,
\Cref{T46} yields
$$
\Vert\tilde{f}_\eps^*\kappa\Vert_{L^{\infty}(\cM)}\leq \Vert\tilde{f}_\eps^*\kappa\Vert_{L^{\infty}(\Omega')}\lesssim\eps^{-k(1-\gamma)}[\tilde{f}]^k_{\gamma,\eps,\Omega}
\lesssim \eps^{-k(1-\gamma)}[f]^k_{\gamma,\eps}
\quad
\text{for } 0<\eps<\frac{\eps_o}{2}\,.
$$
Similarly, \Cref{T45} yields
$$
\Vert\tilde{f}_\eps^*\upalpha\Vert_{L^{\infty}(\cM)}\lesssim\eps^{2\gamma-1}[f]^2_{\gamma,\eps},
$$
Taking $f_\eps:=\tilde{f}_\eps|_{\cM}$ for $0<\eps<\frac{1}{2}\eps_o$, yields the result since according to \Cref{T3}, for all $0<\gamma'<\gamma$ we have $\tilde{f}_\eps\to\tilde{f}$ in $C^{0,\gamma'}(\Omega',\R^{2n+1})$ as $\eps\to 0$.

If $\partial\cM\neq\varnothing$, we glue two copies of $\cM$ along the boundary to obtain a closed manifold $\widetilde{\cM}$, extend $f$ to a mapping in $C^{0,\gamma}(\widetilde{\cM},\bbbh^n)$ (by reflection) and we 
apply the above argument to $\widetilde{\cM}$.
\end{proof}

As an application of \Cref{T50} we will prove the following important result.
\begin{theorem}
\label{T47}
Let $m>n$ be positive integers. Let $\cM$ be a smooth oriented and compact $m$-dimensional manifold with or without boundary, and let
$f\in C^{0,\gamma}(\cM;\bbbh^n)$, where $\gamma\in(\frac{k}{k+1},1]$, $n<k\leq m$.
Then, for any $\kappa\in \Omega^k(\R^{2n+1})$ we have
\begin{equation}
\label{eq33}
\int_{\cM} f^*\kappa\wedge\tau = 0
\quad
\text{for all}\
\tau\in \Omega^{m-k}C_{\rm b}^{0,\gamma}\left(\cM\right).
\end{equation}
\end{theorem}
\begin{remark}\label{R47}
The `integral' in \eqref{eq33} is understood in the distributional sense 
as a limit of smooth approximations, see~\eqref{eq121}.
In the case in which $f$, $\kappa$ and $\tau$ are smooth, equality \eqref{eq33} implies that $f^*\kappa=0$ for all smooth $k$ forms $\kappa$ (cf.\ \cite[Lemma~4.3]{HST14}) which, in turn, is equivalent to $\operatorname{rank} Df<k$. Thus, in some sense \Cref{T47} says that under the given assumptions, a generalized rank of the derivative of $f$ is less than $k$ even though it does not make any sense to talk about the derivative of $f$.
\end{remark}
\begin{proof}
According to \Cref{T79} (with $\alpha=1$), the limit
\begin{equation}
\label{eq121}
\int_{\cM} f^*\kappa\wedge\tau := \lim_{\eps\to 0} \int_{\cM} f^*_\eps\kappa\wedge\tau
\end{equation}
exists and does not depend on a choice of an approximation
$C^\infty(\overbar{\cM};\R^{2n+1})\ni f_\eps\to f$ if $f_\eps\to f$ in $C^{0,\gamma'}$, $\gamma'\in (\frac{k}{k+1},\gamma]$. In what follows we take $f_\eps$ to be the approximation of $f$ from \Cref{T50}.

Assume first that $\tau\in \Omega^{m-k}(\cM)$. According to \Cref{T49}, we have a decomposition 
$$
\kappa=\upalpha\wedge\beta+d\upalpha\wedge\delta
\quad
\text{for some}
\quad
\beta\in \Omega^{k-1}\left(\R^{2n+1}\right),\ 
\delta\in \Omega^{k-2}\left(\R^{2n+1}\right).
$$
A change of the order in the wedge product with respect to that in \Cref{T49} is not essential; it will make our estimates easier to write. Also we use $\delta$ instead of $\gamma$ because $\gamma$ is reserved for the exponent of the H\"older continuity. 

Stokes' theorem yields (there is no boundary term if $\partial\cM=\varnothing$)
$$
\int_{\cM} df_\eps^*\upalpha\wedge f_\eps^*\delta\wedge\tau=
\int_{\cM} f_\eps^*\upalpha\wedge d(f_\eps^*\delta\wedge\tau)+
\int_{\partial\cM} f_\eps^*\upalpha\wedge f_\eps^*\delta\wedge\tau.
$$
Therefore,
\[
\begin{split}
&\int_{\cM} f_\eps^*\kappa\wedge\tau
=
\int_{\cM} f_\eps^*\upalpha\wedge f_\eps^*\beta\wedge\tau +
df_\eps^*\upalpha\wedge f_\eps^*\delta\wedge\tau\\
&=
\int_{\cM} f_\eps^*\upalpha\wedge \big(f_\eps^*(\beta+d\delta)\wedge\tau+(-1)^kf_\eps^*\delta\wedge d\tau\big)
+
\int_{\partial\cM}f_\eps^*\upalpha\wedge f_\eps^*\delta\wedge\tau.
\end{split}
\]
According to \Cref{T50}
$$
\Vert f_\eps^*\upalpha\Vert_\infty\lesssim\eps^{2\gamma-1},
\quad
\Vert f_\eps^*(\beta+d\delta)\Vert_\infty\lesssim\eps^{-(k-1)(1-\gamma)},
\quad
\Vert f_\eps^*\delta\Vert_\infty\lesssim\eps^{-(k-2)(1-\gamma)}.
$$
Hence,
$$
\left|\int_{\cM} f^*_\eps\kappa\wedge\tau\right|\lesssim
\eps^{2\gamma-1}\big(\eps^{-(k-1)(1-\gamma)}+\eps^{-(k-2)(1-\gamma)}\big)=
\eps^{(k+1)\gamma-k}\big(1+\eps^{1-\gamma}\big)\stackrel{\eps\to 0}{\longrightarrow} 0,
$$
because $\gamma>\frac{k}{k+1}$.

In the general case when $\tau\in \Omega^{m-k}C_{\rm b}^{0,\gamma}(\cM)$, we approximate $\tau$ smoothly in $C^{0,\gamma'}$, $\gamma'\in(\frac{k}{k+1},\gamma)$, $C^\infty\ni\tau_\eps\to\tau$, and  we have
$$
0=\int_{\cM} f^*\kappa\wedge\tau_\eps\to \int_{\cM}f^*\kappa\wedge\tau,
$$
where the convergence follows from the fact that distributional integral against $f^*\kappa$ defines a bounded functional on $\Omega^{m-k}C_{\rm b}^{0,\gamma'}(\cM)$, see
\eqref{eq98} (with $\alpha=1$ and $\gamma$ replaced by $\gamma'$).
\end{proof}

\begin{corollary}
\label{T56}
There is no topological embedding $f:\Sph^{2n}\to\bbbh^n$ such that
$f\in C^{0,\gamma}(\Sph^{2n},\bbbh^n)$, $\gamma>\frac{2n}{2n+1}$.
\end{corollary}
\begin{remark}
In the next section we will prove a much deeper non-embedding theorem, \Cref{T53}.
\end{remark}
\begin{proof}
Suppose to the contrary that such an embedding exists.
Since $f\in C^{0,\gamma}(\Sph^{2n},\R^{2n+1})$ by \Cref{T40}, it follows from \Cref{T25} that
\begin{equation}
\label{eq34}
\int_{\Sph^{2n}} f_1\, df_2\wedge\ldots\wedge df_{2n+1}=
\int_{\R^{2n+1}}w(f,y)\, dy\neq 0,
\end{equation}
where the right hand side is non-zero since it equals $\pm$ the volume of the region bounded by $f(\Sph^{2n})$ (the sign depends on the orientation of $f$). On the other hand \Cref{T47} with
$$
\kappa = y_1dy_2\wedge\ldots\wedge dy_{2n+1}\in \Omega^{2n}\left(\R^{2n+1}\right)
\quad
\text{and}
\quad
\tau= 1
$$
yields
$$
\int_{\Sph^{2n}}f_1\, df_2\wedge\ldots\wedge df_{2n+1}=
\int_{\Sph^{2n}} f^*\kappa\wedge\tau=0
$$
which is a contradiction.
\end{proof}

In the next section we will need the following 
\begin{corollary}
\label{T47b}
Let $m>n$ be positive integers. Let $\cM$ be a smooth oriented and compact $m$-dimensional manifold with boundary, and let
$f\in C_{\rm b}^{0,\gamma}(\overline{\cM};\bbbh^n)$, where $\gamma\in (\frac{m}{m+1},1]$.
Then, for any $\kappa\in \Omega^{m-1}(\R^{2n+1})$ we have
$$
\int_{\partial \cM} f^*\kappa = \int_{\cM}f^*(d\kappa)=0.
$$
\end{corollary}
\begin{proof}
First, note that by \Cref{T79} with $\tau=1$, both of the integrals
$\int_{\partial\cM}f^*\kappa$ and $\int_{\cM}f^*(d\kappa)$ are well defined and independent of a smooth approximations in
$C_{\rm b}^{0,\gamma'}(\partial{\cM};\R^{2n+1})$, $\frac{m-1}{m}<\gamma'\leq\gamma$
and 
$C_{\rm b}^{0,\gamma'}(\cM;\R^{2n+1})$, $\frac{m}{m+1}<\gamma'\leq\gamma$ respectively.

Let $f_\eps$ be a $C^\infty(\overbar{\cM};\R^{2n+1})$ approximation of $f$ in $C_{\rm b}^{0,\gamma'}(\cM;\R^{2n+1})$, $\gamma'\in(\frac{m}{m+1},\gamma]$. Since $\gamma'>\frac{m}{m+1}>\frac{m-1}{m}$, we have that $f_\eps|_{\partial\cM}$ is also a good approximation for the integral $\int_{\partial\cM}f^*\kappa$.
Therefore, Stokes' theorem and \Cref{T47} with $\tau=1$ yields
$$
\int_{\partial \cM} f^*\kappa\leftarrow \int_{\partial \cM} f_\eps^*\kappa = \int_{\cM} f_\eps^*(d\kappa)\rightarrow \int_{\cM} f^*(d\kappa)=0.
$$
The proof is complete.
\end{proof}

\section{Non-embedding theorem of Gromov}
\label{S4}
Gromov \cite[3.1.A]{gromov2} proved the following non-embedding result:

\begin{theorem}
\label{T55}
Every $\gamma$-H\"older continuous embedding $f:\R^m\to\bbbh^n$, $m\geq n+1$, satisfies $\gamma\leq\frac{n+1}{n+2}$.
\end{theorem}
The original proof is quite challenging as it relies on the $h$-principle and microflexibility. For a more detailed explanation of this proof, as well as an alternative approach using the Rumin complex \cite{rumin}, see \cite{pansu}. Additionally, there is another proof by Z{\"u}st \cite{zust2015} that shows the non-existence of an embedding $f \in C^{0,\gamma}(\R^2,\bbbh^1)$ for $\gamma > 2/3$. This proof is based on a factorization through a topological tree. However, this argument cannot be extended to the general case of \Cref{T55}.
We should also mention here the work of Balogh, Kozhevnikov, and Pansu \cite{BKP} described in the introduction.
The aim of this section is to present an elementary proof of a more general result, \Cref{T53}, originally proved in \cite{HS23}.

Let us compare \Cref{T55} with \Cref{T56}. First of all, the bound for the exponent obtained in \Cref{T56} is larger than $\frac{n+1}{n+2}$. Both bounds are equal only when $n=1$. However, even in the case $n=1$, \Cref{T56} proves non-existence of an embedding $f\in C^{0,\gamma}(\Sph^2,\bbbh^1)$, $\gamma>2/3$, and it is crucial for the proof that $\Sph^2$ is a closed surface since the argument is based on the fact that the image of $\Sph^2$
bounds a region in $\R^3$ of positive volume. 
\begin{theorem}
\label{T53}
Suppose that $m\geq n+1$,
$$
\frac{1}{2}\leq\gamma\leq\frac{n+1}{n+2}
\quad
\text{and}
\quad
\theta=\frac{n+1}{n}-\frac{2\gamma}{n}.
$$
Then, there does not exist a map
$$
f\in C^{0,\gamma+}(\overbar{\bbbb}^m;\bbbh^n)\cap C^{0,\theta}(\overbar{\bbbb}^m;\R^{2n+1})
$$
or a map
$$
f\in C^{0,\gamma}(\overbar{\bbbb}^m;\bbbh^n)\cap C^{0,\theta+}(\overbar{\bbbb}^m;\R^{2n+1}),
$$
such that $f|_{\Sph^{m-1}}$ is a topological embedding. 
\end{theorem}
\begin{remark}
If $\gamma\in (\frac{n+1}{n+2},1]$ and $\theta=\frac{n+1}{n}-\frac{2\gamma}{n}$, then \Cref{T53} is still true, because it follows from the case $\gamma=\frac{n+1}{n+2}$ (see \Cref{T97}). Therefore, including the range $\gamma\in(\frac{n+1}{n+2},1]$ in \Cref{T53} would not provide any new information.
\end{remark}
\begin{corollary}
\label{T97}
If $m\geq n+1$, then there does not exist a mapping
\begin{equation}
\label{eq143}
f\in C^{0,\frac{n+1}{n+2}}(\overbar{\bbbb}^m;\bbbh^n)\cap C^{0,\frac{n+1}{n+2}+}(\overbar{\bbbb}^m;\R^{2n+1}).
\end{equation}
such that $f|_{\Sph^{m-1}}$ is a topological embedding. In particular, there is no topological embedding of $\overbar{\bbbb}^m$ into $\bbbh^n$ of the class \eqref{eq143}
\end{corollary}
\begin{proof}
If $\gamma=\frac{n+1}{n+2}$, then $\theta=\frac{n+1}{n+2}$.
\end{proof}
\begin{remark}
Since $C^{0,\gamma}\subset C^{0,\frac{n+1}{n+2}+}$ for $\gamma>\frac{n+1}{n+2}$, \Cref{T55} follows from \Cref{T97}.
\end{remark}
If $\gamma=\frac{1}{2}$, then $\theta=1$ and hence we have the following result which was proved in 
\cite[Theorem~1.11]{BHW}:
\begin{corollary}
\label{T95}
Suppose that $m\geq n+1$. Then there does not exist an embedding
$$
f\in C^{0,\frac{1}{2}+}(\overbar{\bbbb}^m;\bbbh^n)\cap C^{0,1}(\overbar{\bbbb}^m;\bbbr^{2n+1}).
$$
\end{corollary}
\begin{remark}
In some sense, we can regard \Cref{T53} as an interpolation of \Cref{T97} and \Cref{T95}.
\end{remark}

In order to prove \Cref{T53} we need to construct a certain differential form $\kappa$. This is done in \Cref{LN}.
As a straightforward application of \Cref{T47b} to this form $\kappa$, in \Cref{GroPr}, we present a very short proof of Gromov's \Cref{T55}, a special case of \Cref{T53}. The proof of \Cref{T53} is based on the same idea, but it is more complicated and it is presented in \Cref{ProofNonEmb}.

\subsection{Alexander duality}
\label{LN}
The content of this section is based on \cite{HS23}. We included it for the sake of completeness. 
The main result of this section is \Cref{T21}. While neither the statement nor the proof requires understanding of topology, the argument in the proof is based on ideas related to the Mayer–Vietoris sequence and the Alexander duality -- see e.g. \cite[Chapter~5]{MT97} and \cite[Corollary~1.29]{Vick}.

\begin{lemma}
\label{T20}
Let $k<n$ and let $f:\cM\to\R^n$ be a continuous map defined on a $k$-dimensional oriented and closed manifold $\cM$. If $\omega\in \Omega^k(\R^n)$ satisfies $d\omega=0$ in a neighborhood of the compact set $f(\cM)$, then there is $\delta>0$ such that if 
$g_0,g_1\in C^\infty(\cM;\R^n)$ satisfy $\Vert f-g_i\Vert_\infty<\delta$, $i=0,1$, then
$$
\int_{\cM} g_0^*\omega=\int_{\cM} g_1^*\omega.
$$
\end{lemma}
\begin{remark}
\label{R2}
Therefore, if $f_\eps\rightrightarrows f$ is a smooth approximation of $f$, then 
there is $\eps_0>0$ such that for all $0<\tau<\eps_0$
\begin{equation}
\label{eq301}
\int_{\cM} f^*\omega:=\lim_{\eps\to 0}\int_{\cM} f^*_\eps\omega=
\int_{\cM} f_\tau^*\omega.
\end{equation}
The limit exists and it does not depend on the choice of the approximation, allowing one to define the integral of the pullback of a form $\omega$ for any continuous map, provided $d\omega$ vanishes in a neighborhood of $f(\cM)$.
\end{remark}
\begin{proof}
Let
$$
U_\delta=\{y\in\R^n:\, \dist(y,f(\cM))<\delta\}.
$$
Assume $\delta>0$ is so small that $d\omega=0$ in $U_\delta$. Let
$$
H:\cM\times [0,1]\to\R^n,
\quad
H(x,t)=(1-t)g_0(x)+tg_1(x)
$$
be a smooth homotopy between $g_0$ and $g_1$. Since
$$
|H(x,t)-f(x)|\leq (1-t)|g_0(x)-f(x)|+t|g_1(x)-f(x)|<\delta,
$$
it follows that the image of the homotopy $H$ is contained in the open set $U_\delta$, where $d\omega=0$ and the result follows from the Stokes theorem
$$
\int_{\cM}g_1^*\omega-\int_{\cM}g_0^*\omega=
\int_{\partial(\cM\times[0,1])}H^*\omega=\int_{\cM\times [0,1]} H^*(d\omega)=0.
$$
\end{proof}
\begin{theorem}
\label{T21}
Let $k<n$ and let $f:\Sph^k\to\R^n$ be a topological embedding (i.e. $f$ is continuous and one-to-one). Then, there is $\omega\in \Omega^k_c(\R^n)$ such that
$d\omega\in\Omega_c^{k+1}(\R^n\setminus f(\Sph^k))$ and
\begin{equation}
\label{eq302}
\int_{\Sph^k} f^*\omega=1,
\end{equation}
where the integral is defined as the limit \eqref{eq301}.
\end{theorem}
\begin{remark}
\label{R13}
While we will not need this remark in what follows, it is important to understand that \Cref{T21} has a natural interpretation in the language of de~Rham cohomology. According to \cite[Corollary 1.29]{Vick} and the Poincar\'e duality \cite[Theorem~3.35]{Hatcher}, we have that
\begin{equation}
\label{eq133}
H_c^{k+1}(\R^n\setminus f(\Sph^k))=\R.
\end{equation}
If $\tau\in\Omega_c^{k+1}(\R^n\setminus f(\Sph^k))$, $d\tau=0$, then $\tau=d\eta$ for some $\eta\in \Omega^k(\R^n)$, and the mapping
$$
H_c^{k+1}(\R^n\setminus f(\Sph^k))\ni [\tau]\mapsto\int_{\Sph^k}f^*\eta\in\R
$$
is well defined and gives the isomorphism \eqref{eq133}. Therefore, if $\omega$ is as in \Cref{T21}, then the closed form $\tau:=d\omega\in\Omega_c^{k+1}(\R^n\setminus f(\Sph^k))$ is a generator of $H_c^{k+1}(\R^n\setminus f(\Sph^k))$. 
\end{remark}

\begin{proof}
We will prove the result using induction with respect to $k$.

For $k=0$, we have $\Sph^0=\{-1,+1\}$. Since $\Omega^0_c(\mathbb{R}^n)$ consists of compactly supported smooth functions $\omega \in C^\infty_c(\mathbb{R}^n)$, it follows that for any approximation $f_\eps$ of $f$ we have
$$
\int_{\Sph^0} f^*\omega=\lim_{\eps \to 0} \int_{\Sph^0}f_\eps^*\omega=
\lim_{\eps \to 0} \left (\omega(f_\eps(+1))-\omega(f_\eps(-1))\right ).
$$
So it suffices to pick $\omega \equiv 1$ around $f(1)$ and $\omega \equiv 0$ around $f(-1)$ to conclude the claim.

We now suppose that the result is true for $k-1$ and we will prove it for $k$.

Let $\Sph_+$, $\Sph_-$ and $E=\Sph_+\cap\Sph_-$ be the northern hemisphere, southern hemisphere and the equator of $\Sph^k$. Here we assume that the hemispheres are closed.
Let
$$
\tilde{\Sph}_+=f(\Sph_+),
\quad
\tilde{\Sph}_-=f(\Sph_-)
\quad
\text{and}
\quad
\tilde{E}=f(E).
$$

Since $E=\Sph^{k-1}$, by the induction hypothesis we can find
$\eta\in \Omega^{k-1}_c(\R^n)$ such that $d\eta=0$ is a neighborhood of $\tilde{E}$ and 
\begin{equation}
\label{eq134}
\int_{\Sph^{k-1}}f^*\eta=\int_{E}f^*\eta=1.
\end{equation}

Therefore, there is $\delta>0$ such that $d\eta=0$ on
$$
\tilde{E}_\delta =\{y\in \R^n:\dist(y,\tilde{E})<\delta\}.
$$
Let $\{\psi_+,\psi_-,\psi_E\}$ be a smooth partition of unity subordinate to the open covering 
$\{\R^n\setminus\tilde{\Sph}_-,\R^n\setminus\tilde{\Sph}_+,\tilde{E}_\delta\}$ of $\R^n$.
Note that $\psi_E d\eta=0$ since $d\eta$ vanishes on $\tilde{E}_\delta$. We have
$$
d\eta=\psi_+d\eta+\psi_-d\eta+\psi_Ed\eta=\psi_+d\eta+\psi_-d\eta:=\omega+\sigma
$$
where
$$
\omega=\psi_+d\eta\in \Omega^k_c(\R^n),
\quad
\sigma=\psi_-d\eta \in \Omega^k_c(\R^n).
$$
We claim that the form $\omega$ proves the induction step. To this end, we need to show that $d\omega$ vanishes in a neighborhood of $f(\Sph^k)$ and that \eqref{eq302} is satisfied.

First observe that $0=d^2\eta=d\omega+d\sigma$, so $d\omega=-d\sigma$. 
Now, $\omega$, and hence $d\omega$ vanish in a neighborhood of $\tilde{\Sph}_-$, because $\omega=\psi_+d\eta$, and $\psi_+$ is supported in $\R^n\setminus\tilde{\Sph}_-$. For a similar reason $d\omega=-d\sigma$ vanishes in a neighborhood of $\tilde{\Sph}_+$. Therefore we proved that $d\omega$ vanishes in a neighborhood of
$\tilde{\Sph}_-\cup\tilde{\Sph}_+=f(\Sph^k)$.

Let now $f_\eps\in C^\infty(\Sph^k;\R^n)$, $\Vert f-f_\eps\Vert_\infty\to 0$ as $\eps\to 0$. Then, according to \Cref{T20}, \Cref{R2}, and \eqref{eq134}, there is $\eps_o>0$ such that for $0<\eps<\eps_o$
$$
\int_E f^*\eta=\int_E f^*_\eps\eta=1,
\quad
\text{and}
\quad
\int_{\Sph^k} f^*\omega=\int_{\Sph^k} f^*_\eps\omega.
$$
Note that $\omega=\psi_+d\eta$ vanishes in a neighborhood of $\tilde{\Sph}_-=f(\Sph_-)$ and $\sigma =\psi_-d\eta$ vanishes in a neighborhood of $\tilde{S}_+=f(\Sph_+)$.
Therefore, if $\eps>0$ is sufficiently small, say $0<\eps<\eps_1<\eps_o$, then $\omega$ and $\sigma$ will vanish in neighborhoods of $f_\eps(\Sph_-)$ and $f_\eps(\Sph_+)$ respectively. In particular,
$$
\int_{\Sph^k} f_\eps^*\omega=\int_{\Sph_+} f_\eps^*\omega
\quad
\text{and}
\quad
\int_{\Sph_+} f_\eps^*\sigma=0.
$$
Therefore for $0<\eps<\eps_1$ we have
$$
\int_{\Sph^k} f^*\omega=\int_{\Sph^k}f_\eps^*\omega=\int_{\Sph_+} f_\eps^*\omega=
\int_{\Sph_+} f_\eps^*(d\eta-\sigma)=\int_{\Sph_+} f_\eps^*(d\eta)
\stackrel{\text{\scriptsize{(Stokes)}}}{=}
\int_E f^*_\eps\eta=1.
$$
The proof is complete.
\end{proof}

\subsection{Direct proof of  Theorem~\ref{T55}}
\label{GroPr}
Suppose to the contrary that there is a $\gamma$-H\"older continuous embedding $f:\R^m\to\bbbh^n$, where $n\geq m+1$ and $\gamma>\frac{n+1}{n+2}$. In particular, $f$ restricted to $\R^{n+1}\subset\R^m$ is an embedding and we can assume that $m=n+1$.
Since $f:\Sph^n\subset\R^{n+1}\to\bbbh^n$ is an embedding, Theorem~\ref{T21} yields a form $\omega\in \Omega^n_c(\R^{2n+1})$ such that
$
1=\int_{\Sph^n} f^*\omega=\int_{\bbbb^{n+1}} f^*(d\omega)=0,
$
where the last equality follows from Corollary~\ref{T47b}.
The contradiction completes the proof. 
\hfill$\Box$

The proof of \Cref{T53} follows from the same idea, but since the assumptions about the H\"older continuity of $f$ are of different type than the assumptions in \Cref{T47} and \Cref{T47b}, instead of applying these results directly, we have to adapt the proofs of these results.

\subsection{Proof of Theorem~\ref{T53}}
\label{ProofNonEmb}
In the proof, we denote by $[f]_\gamma$ and $[f]_\theta$ the H\"older seminorms with respect to the Kor\'anyi metric in $\bbbh^n$ and the Euclidean metric in $\R^{2n+1}$, respectively.

We can assume that $m=n+1$ since $\overbar{\bbbb}^m$ contains $\overbar{\bbbb}^{n+1}$. Suppose to the contrary that there is
$$
f\in C^{0,\gamma+}(\overbar{\bbbb}^{n+1},\bbbh^n)\cap
C^{0,\theta}(\overbar{\bbbb}^{n+1},\R^{2n+1})
\ \
\text{or}
\ \
f\in C^{0,\gamma}(\overbar{\bbbb}^{n+1},\bbbh^n)\cap
C^{0,\theta+}(\overbar{\bbbb}^{n+1},\R^{2n+1}),
$$
such that $f|_{\Sph^{n}}$ is a topological embedding and $\gamma$ and $\theta$ satisfy assumptions of the theorem. We can further assume that $f$ is defined on $\R^{n+1}$, because extending $f$ to  $x\not\in\overbar{\bbbb}^{n+1}$ by $f(x):=f(x/|x|)$ preserves regularity of the mapping.

By Theorem~\ref{T21} there is 
$\kappa\in \Omega_c^n(\R^{2n+1})$ such that 
$d\kappa\in\Omega^{n+1}_c(\R^{2n+1}\setminus f(\Sph^n))$ and
$$
\int_{\Sph^n} f^*\kappa =1.
$$
According to Corollary~\ref{T49} we have a decomposition
$$
d\kappa=\upalpha\wedge\beta+d\upalpha\wedge\delta,
\quad
\text{for some}
\quad
\beta\in \Omega^n_c\left(\R^{2n+1}\setminus f(\Sph^n)\right),\
\delta\in \Omega^{n-1}_c\left(\R^{2n+1}\setminus f(\Sph^n)\right).
$$
Let $f_\eps$ be the approximation of $f:\R^{n+1}\to\R^{2n+1}$ by mollification. 
Since $\delta=0$ near $f(\Sph^n)$, we have that $f^*_\eps\delta=0$ near $\partial\bbbb^{n+1}$, provided $\eps>0$ is sufficiently small. Thus, the form $d(f^*_\eps\upalpha\wedge f_\eps^*\delta)$ has compact support in $\bbbb^{n+1}$, and Stokes' theorem yields
$$
\int_{\bbbb^{n+1}}df_\eps^*\upalpha\wedge f_\eps^*\delta=
\int_{\bbbb^{n+1}} f_\eps^*\upalpha\wedge f_\eps^*(d\delta). 
$$
Therefore, Theorem~\ref{T45} and Lemma~\ref{T8} yield
\begin{align*}
1
&=
\int_{\Sph^n} f^*\kappa = 
\lim_{\eps\to 0} \int_{\Sph^n} f_\eps^*\kappa =
\lim_{\eps\to 0} \int_{\bbbb^{n+1}} f_\eps^*(d\kappa)=
\lim_{\eps\to 0} \int_{\bbbb^{n+1}}
\big(f^*_\eps\alpha\wedge f_\eps^*\beta+df_\eps^*\alpha\wedge f_\eps^*\delta\big)\\
&=
\lim_{\eps\to 0}\int_{\bbbb^{n+1}} f_\eps^*\alpha\wedge f_\eps^*(\beta+d\delta)
\lesssim 
\liminf_{\eps\to 0} 
\Vert f_\eps^*\alpha\Vert_\infty \Vert f_\eps^*(\beta+d\delta)\Vert_\infty\\
&\lesssim
\liminf_{\eps\to 0} \eps^{2\gamma-1}[f]^2_{\gamma,2\eps,B_2^{n+1}}\cdot\eps^{-n(1-\theta)}[f]_{\theta,\eps,B_2^{n+1}}^n\\
&=
\liminf_{\eps\to 0}[f]^2_{\gamma,2\eps,B_2^{n+1}}[f]_{\theta,\eps,B_2^{n+1}}^n
=0,
\end{align*}
because $f\in C^{0,\gamma+}$ or $f\in C^{0,\theta+}$.
\hfill $\Box$

\section{Hodge decomposition}
\label{S:Hodge}
The theory of differential forms and the Hodge decomposition in Sobolev spaces was developed in \cite{ISS99}. We slightly adapt their results for our purposes. We also refer the interested reader to the monograph \cite{CDK}.

Throughout this section we assume that $\cM$ is a compact and oriented $k$-dimensional Riemannian manifold without boundary.

The {\em Laplace-de Rham operator} is defined by
$$
\Delta:\Omega^\ell(\cM)\to\Omega^\ell(\cM),
\qquad
\Delta\omega=(d\delta+\delta d)\omega.
$$
We say that a form $\omega\in\Omega^\ell(\cM)$ is {\em harmonic} if $\Delta\omega=0$.
\begin{lemma}
\label{T103}
Let $\cM$ be a compact and oriented $k$-dimensional Riemannian manifold without boundary. Then, $\omega\in\Omega^\ell(\cM)$, $0\leq\ell\leq k$, is harmonic if and only if $\omega$ is closed $d\omega=0$ and co-closed $\delta\omega=0$.
\end{lemma}
\begin{proof}
The implication $\Leftarrow$ is obvious. The other implication $\Rightarrow$ easily follows from the integration by parts formula \eqref{eq137}. Indeed,
$$
0=\int_\cM\langle\Delta\omega,\omega\rangle=\int_\cM\langle d\omega,d\omega\rangle
+ \int_\cM\langle \delta\omega,\delta\omega\rangle=\Vert d\omega\Vert_2^2+\Vert\delta\omega\Vert_2^2,
$$
and hence $d\omega=0$ and $\delta\omega=0$.
\end{proof}
For the proof of the next result, see \cite[Proposition~6.5]{scott}.
\begin{lemma}[$L^p$-Hodge decomposition]
\label{T104}
Let $\cM$ be a compact and oriented $k$-dimensional Riemannian manifold without boundary. Then, for any $p\in (1,\infty)$ and any $\eta\in\Omega^\ell L^p(\cM)$, $0\leq\ell\leq k$, there exist
$$
\omega_1\in \Omega^{\ell-1}W^{1,p}(\cM) 
\quad
\text{and}
\quad
\omega_2\in\Omega^{\ell+1}(\cM),
$$
such that
$$
\eta=d\omega_1+\delta\omega_2+h,
$$
where $h\in\Omega^\ell(\cM)$ is a smooth harmonic form.
\end{lemma}
\begin{remark}
If $\ell=0$, then $\omega_1=0$, and similarly, if $\ell=k$, then $\omega_2=0$.
\end{remark}
\begin{remark}
Later, we will sketch a proof of a special case of this result, see Corollary~\ref{th:hodge}.
\end{remark}
Recall that de~Rham cohomology is denoted by $H^\ell(\cM)$. 
In particular, $H^\ell(\cM)=0$ if for every closed smooth $\ell$-form $\omega \in\Omega^\ell(\cM)$, $d\omega = 0$, there exists $\alpha \in \Omega^{\ell-1}(\cM)$ such that $\omega = d\alpha$.
For example $H^\ell(\mathbb{S}^n) = 0$ for all $\ell = 1,\ldots,n-1$, and $H^\ell(\R^n)=0$ for all $\ell\geq 1$.

\begin{lemma}
\label{T101}
Let $\cM$ be an oriented $k$-dimensional Riemannian manifold without boundary. Assume that $H^\ell(\cM)=0$ for some $\ell=1,2,\ldots,k-1$. Then, for every $\omega\in \Omega^{k-\ell}(\cM)$ satisfying $\delta\omega=0$, there is $\beta\in\Omega^{k-\ell+1}(\cM)$ such that $\delta\beta=\omega$.
\begin{proof}
Indeed, $\delta \omega = 0$ implies $d\ast \omega = 0$, thus $\ast \omega = d\alpha$ for some $\alpha \in \Omega^{\ell-1}(\cM)$, and then 
\[
\omega = (-1)^{\ell(k-\ell)} \ast \ast \omega = (-1)^{\ell(k-\ell)}\ast d\alpha = (-1)^{k(k-\ell+1)}\delta \ast \alpha.
\]
Thus $\omega = \delta \beta$ for $\beta = (-1)^{k(k-\ell+1)} \ast \alpha$.    
\end{proof}
\end{lemma}
The next result is a straightforward consequence of the Poincar\'e duality \cite[Theorem~6.13]{warner}.
\begin{lemma}
\label{T102}
Let $\cM$ be a compact and oriented $k$-dimensional Riemannian manifold without boundary.  Then for any $\ell=1,2,\ldots,k-1$ we have that $\dim H^{\ell}(\cM)=\dim H^{k-\ell}(\cM)$. In particular, if $H^\ell(\cM)=0$, then $H^{k-\ell}(\cM)=0$.
\end{lemma}

For the proof of the next result, see \cite[Theorem~4.8]{ISS99}.
\begin{lemma}[Gaffney's inequality]
\label{la:gaffney}
Let $\cM$ be a compact and oriented $k$-dimensional Riemannian manifold without boundary. 
Then, for any $\ell=0,1,2,\ldots,k$ we have
\[
 \|\omega\|_{1,2}  \approx 
 \|\omega\|_{2} + \|d\omega\|_{2} + \|\delta\omega\|_{2}
 \quad \text{for all } \omega \in\Omega^\ell W^{1,2}(\cM),
\]
with constants dependent on $\cM$ only.
\end{lemma}
Here $\Vert\cdot\Vert_{1,2}$ and $\Vert\cdot\Vert_{2}$ denote the Sobolev $W^{1,2}$ and the $L^2$ norms of the forms.
\begin{remark}
Usually Gaffney's inequality is stated in the form of the inequality $\lesssim$, but since the opposite inequality is obvious, we have the norm equivalence $\approx$.
\end{remark}

\begin{lemma}
\label{la:noharm}
Let $\cM$ be a compact and oriented $k$-dimensional Riemannian manifold without boundary. 
Assume that $H^\ell(\mathcal{M})=0$ for some $1 \leq \ell \leq k-1$.
If $\omega \in\Omega^\ell W^{1,2}(\cM)$ satisfies $d\omega=0$ and $\delta\omega=0$, then $\omega \equiv 0$.
\end{lemma}
\begin{remark}
In fact, a more general result of Hodge says that in every class of the cohomology group, there is a unique (up to a multiplicative constant) harmonic form, see \cite[Theorem~6.11]{warner}.  
\end{remark}
\begin{proof}
If  $\omega \in\Omega^\ell(\cM)$  satisfies $d\omega = 0$ and $\delta \omega = 0$ (i.e., if $\omega$ is a smooth harmonic form), then $\omega = d\alpha = \delta \beta$ for some $\alpha\in \Omega^{\ell-1}(\cM)$ and $\beta\in\Omega^{\ell+1}(\cM)$. Indeed, existence of $\alpha$ follows from $H^{\ell}(\cM)=0$. Since $H^{k-\ell}(\cM)=0$ by Lemma~\ref{T102}, existence of $\beta$ follows from Lemma~\ref{T101}.
This is also true when $\omega \in \Omega^\ell W^{1,2}(\cM)$ satisfies $d\omega=0$ and $\delta\omega=0$, but now the forms $\alpha$ and $\beta$ have Sobolev regularity, see \cite[Proposition 4.5.]{HST14}.
The integration by parts \eqref{eq137} yields 
\[
\|\omega\|_{2}^2 = \int_{\mathcal{M}} \langle\omega,\omega\rangle = \int_{\mathcal{M}} \langle d\alpha,\delta \beta\rangle = \int_{\mathcal{M}} \langle d^2\alpha,\beta\rangle=0.
\]
Consequently $\omega \equiv 0$.
\end{proof}

\begin{lemma}[Gaffney-Poincar\'e inequality]
\label{pr:Poincare}
Let $\cM$ be a compact and oriented $k$-dimensional Riemannian manifold without boundary. 
If $H^\ell(\mathcal{M})=0$ for some $1 \leq \ell \leq k-1$, then
we have
\[
\|\omega \|_{1,2} \approx \|d\omega \|_{2} + \|\delta \omega \|_{2} 
\quad
\text{for all } \omega\in\Omega^\ell W^{1,2}(\cM),
\]
with constants dependent on $\cM$ only.
\end{lemma}
\begin{proof}
According to Lemma~\ref{la:gaffney}, it suffices to prove the Poincar\'e type estimate
\[
\|\omega \|_{2} \lesssim \|d\omega \|_{2} + \|\delta \omega \|_{2}.
\]
The claim follows from a typical blowup-type argument used to prove the  Poincar\'e inequality: Assume the claim is false, then there exists a sequence $\omega_i$ with
\[
 \|\omega_i \|_{2} = 1
\]
and
\begin{equation}
\label{eq:poinc:1}
 \|d\omega_i \|_{2} + \|\delta \omega_i \|_{2} \leq \frac{1}{i}.
\end{equation}
By Gaffney's inequality, Lemma~\ref{la:gaffney},
\[
 \sup_{i} \|\omega_i\|_{1,2} < \infty.
\]
From the Rellich-Kondrachov theorem (we can apply the Euclidean version on every coordinate patch of $\mathcal{M}$), we obtain a subsequence (still denoted by $\omega_i$) convergent in $L^2$ to a limit $\omega \in \Omega^\ell W^{1,2}(\cM)$,
\begin{equation}
\label{eq:poinc:2}
\|\omega \|_{2} = \lim_{i \to \infty} \|\omega_i \|_{2} = 1.
\end{equation}
Moreover, using a partition of unity and arguing in coordinate-patches, we have weak convergence: for any $\eta \in \Omega^{k-\ell-1}(\cM)$
\[
\int_{\mathcal{M}} d\omega \wedge \eta = \lim_{i\to\infty}  \int_{\mathcal{M}} d\omega_i \wedge \eta=0,
\]
and  for any $\tau \in \Omega^{k-\ell+1}(\cM)$
\[
\int_{\mathcal{M}} \delta \omega \wedge \tau =\lim_{i\to\infty}  \int_{\mathcal{M}} \delta \omega_i \wedge \tau=0.
\]
The fact that the limits equal zero follows from \eqref{eq:poinc:1}.
Therefore, 
$d\omega = 0$ and $\delta \omega = 0$. Hence, Lemma~\ref{la:noharm} yields $\omega \equiv 0$, which contradicts \eqref{eq:poinc:2}.
\end{proof}

\begin{theorem}
\label{th:poisson}
Let $\cM$ be a compact and oriented $k$-dimensional Riemannian manifold without boundary. 
Assume also that $H^\ell(\mathcal{M})=0$ for some $1 \leq \ell \leq k-1$.
Then, for any $\ell$-form $\omega \in\Omega^\ell(\cM)$ there exists a unique $\eta \in\Omega^\ell(\cM)$ such that 
\begin{equation}
\label{eq:poisson:pde}
\Delta \eta=(d\delta+\delta d)\eta = \omega. 
\end{equation}
Moreover,
\begin{equation}
\label{eq145}
\Vert\eta\Vert_{1,2}\leq C(\cM)\Vert\omega\Vert_2.
\end{equation}
\end{theorem}
\begin{proof}
Consider the energy functional
\[
\mathcal{E}(\eta) 
:= 
\frac{1}{2}\int_\cM\langle d\eta,d\eta\rangle+\langle\delta\eta,\delta\eta\rangle-2\langle\omega,\eta\rangle
=
\frac{1}{2}\brac{\Vert d\eta\Vert_2^2+\Vert\delta\eta\Vert_2^2}-\int_\cM \langle\omega,\eta\rangle.
\]
We can minimize $\mathcal{E}$ over $\eta \in\Omega^\ell W^{1,2}(\cM)$ by the direct method of the calculus of variations. Indeed, \Cref{pr:Poincare}, yields
\[
\begin{split}
\|\eta\|_{1,2}^2  
&\leq 
C(\mathcal{M}) \frac{1}{2}\brac{\Vert d\eta\Vert_2^2+\Vert\delta\eta\Vert_2^2}=
C(\cM)\bigg(\mathcal{E}(\eta)+\int_\cM\langle\omega,\eta\rangle\bigg)\\
&\leq 
C(\cM)\brac{\mathcal{E}(\eta) + \|\omega\|_{2}\, \|\eta\|_{2}}. 
\end{split}
\]
By Young's inequality, 
\[
 \|\eta\|_{1,2}^2  \leq C(\mathcal{M})\mathcal{E}(\eta) + C'(\mathcal{M}) \|\omega\|_{2}^2 +\frac{1}{2} \|\eta\|_{2}^2. 
\]
Absorbing $\frac{1}{2}\|\eta\|_{2}^2$ to the left-hand side we obtain 
\begin{equation}
\label{eq136}
\Vert\eta\Vert_{1,2}^2\leq C(\cM)(\mathcal{E}(\eta)+\Vert\omega\Vert_2^2),
\end{equation}
and hence $\mathcal{E}$ is coercive. 
Since $\mathcal{E}$ is sequentially weakly lower semicontinuous on $W^{1,2}$, the direct method of calculus of variations gives a minimizer $\eta \in \Omega^\ell W^{1,2}(\cM)$ of $\mathcal{E}$. 

To derive the Euler-Lagrange equation for the minimizer $\eta$, observe that for any $\phi\in\Omega^\ell W^{1,2}(\cM)$, we have
$\frac{d}{dt}\big|_{t=0}\mathcal{E}(\eta+t\phi)=0$, and hence
\begin{equation}
\label{eq135}
\int_\cM \langle d\eta,d\phi\rangle+\langle\delta\eta,\delta\phi\rangle=\int_\cM\langle\omega,\phi\rangle
\quad
\text{for all } \phi\in \Omega^\ell W^{1,2}(\cM).
\end{equation}
The integration by parts formula \eqref{eq137}
shows that \eqref{eq135} is the weak formulation of \eqref{eq:poisson:pde}.
Observe now that $\Delta\eta=\omega$ restricted to every coordinate patch is an elliptic equation \cite[Section~6.35]{warner}, and by elliptic regularity theory (\cite[Theorem 6.30]{warner}), $\eta \in C^\infty$.

For uniqueness, we need to show that if $\eta \in\Omega^\ell (\cM)$ is a harmonic form i.e., $\lap\eta=0$ on $\cM$, then $\eta \equiv 0$. This, however follows from Lemma~\ref{T103} and Lemma~\ref{la:noharm}. 

It remains to prove inequality \eqref{eq145}. Since $\eta$ is a minimizer over $\Omega^\ell W^{1,2}(\cM)$, $\mathcal{E}(\eta)\leq\mathcal{E}(0)=0$, and hence estimate \eqref{eq136} proves \eqref{eq145}.
\end{proof}

As a consequence of Theorem~\ref{th:poisson} we obtain the following Hodge decomposition. 
\begin{corollary}
\label{th:hodge}
Let $\cM$ be a compact and oriented $k$-dimensional Riemannian manifold without boundary. 
Assume that $H^\ell(\cM)=0$ for some $1\leq \ell\leq k-1$.
If $\omega \in\Omega^\ell(\cM)$, then there are forms $\omega_1 \in \Omega^{\ell-1}(\cM)$, $\omega_2 \in \Omega^{\ell+1}(\cM)$ such that 
\begin{equation}
\label{eq:hodge:1}
\omega = d\omega_1 + \delta \omega_2 
\quad 
\text{and}
\quad
\delta \omega_1 = 0, \
d\omega_2 = 0,
\end{equation}
\begin{equation}
\label{eq156}
\Vert\omega_1\Vert_2\lesssim_\cM \Vert d\omega_1\Vert_2,
\qquad
\Vert\omega_2\Vert_2\lesssim_\cM \Vert \delta\omega_2\Vert_2,
\end{equation}
\begin{equation}
\label{eq138}
\Vert\omega_1\Vert_2+\Vert\omega_2\Vert_2\lesssim_\cM \Vert\omega\Vert_2.
\end{equation}
Moreover, the forms $d\omega_1$ and $\delta\omega_2$ satisfying \eqref{eq:hodge:1} are unique.
\end{corollary}
\begin{remark}
Uniqueness of the forms $d\omega_1$ and $\delta\omega_2$ means that if
$$
\omega = d\tau_1 + \delta \tau_2 
\quad 
\text{and}
\quad
\delta \tau_1 = 0, \
d\tau_2 = 0,
$$
then $d\tau_1=d\omega_1$ and $\delta\tau_2=\delta\omega_2$. However, in general, the forms $\omega_1$ and $\omega_2$ need not be unique, because we can add to $\omega_1$ and to $\omega_2$ any harmonic form in $\Omega^{\ell-1}(\cM)$ and $\Omega^{\ell+1}(\cM)$ respectively.
\end{remark}
\begin{proof}
To prove uniqueness of a decomposition satisfying \eqref{eq:hodge:1} we need to show that if
$$
\omega = d\omega_1 + \delta \omega_2 =d\tau_1+ \delta\tau_2,
\quad
\delta\omega_1=\delta\tau_1=0,
\quad
d\omega_2=d\tau_2=0,
$$
then $d\omega_1=d\tau_1$ and $\delta\omega_2=\delta\tau_2$.
Since $d(\omega_1-\tau_1)+\delta(\omega_2-\tau_2)=\omega-\omega=0$, integration by parts \eqref{eq137} yields
\[
0=
\int_\cM\langle d(\omega_1-\tau_1),d(\omega_1-\tau_1)\rangle+\langle\delta(\omega_2-\tau_2),d(\omega_1-\tau_1)\rangle
=
\Vert d(\omega_1-\tau_1)\Vert_2^2,
\]
and hence $d\omega_1=d\tau_1$. Similar argument shows that $\Vert\delta(\omega_2-\tau_2)\Vert_2^2=0$, so $\delta\omega_2=\delta\tau_2$.

Therefore, it remains to prove existence of the decomposition \eqref{eq:hodge:1} satisfying \eqref{eq156} and \eqref{eq138}.

Let $\eta\in\Omega^\ell(\cM)$ be a unique solution to $\Delta\eta=\omega$, see \Cref{th:poisson}. Let $\lambda_1:=\delta\eta$ and $\lambda_2:=d\eta$. Clearly,
\begin{equation}
\label{eq157}
d\lambda_1+\delta\lambda_2=\omega,
\quad
\delta\lambda_1=0,
\quad
d\lambda_2=0,
\quad
\Vert\lambda_1\Vert_2+\Vert\lambda_2\Vert_2\lesssim\Vert\eta\Vert_{1,2}\lesssim\Vert\omega\Vert_2.
\end{equation}
That is, conditions \eqref{eq:hodge:1} and \eqref{eq138} are satisfied, but \eqref{eq156} need not be true.

Let $\eta',\eta''\in\Omega^\ell(\cM)$ be unique solutions to
$$
\Delta\eta'=d\lambda_1
\quad
\text{and}
\quad
\Delta\eta''=\delta\lambda_2.
$$
We claim that
\begin{equation}
\label{eq158}
d\delta\eta'=d\lambda_1
\quad
\text{and}
\quad\delta d\eta''=\delta\lambda_2.
\end{equation}
Indeed, since
$$
d\delta\eta'+\delta d\eta'=\Delta\eta'=d\lambda_1,
$$
we need to show that $\delta d\eta'=0$ and it easily follows from the integration by parts \eqref{eq137}
$$
\Vert\delta d\eta'\Vert_2^2=\int_\cM\langle\delta d\eta',\delta d\eta'\rangle=\int_\cM\langle\delta d\eta',d\lambda_1-d\delta\eta'\rangle=0, 
$$
so $\delta d\eta'=0$. This proves the first equality in \eqref{eq158}. The proof of the second equality is similar and left to the reader.
Let 
$$
\omega_1:=\delta\eta'
\quad
\text{and}
\quad
\omega_2:=d\eta''.
$$
Then \eqref{eq158} and \eqref{eq157} yield
$$
d\omega_1+\delta\omega_2=d\lambda_1+\delta\lambda_2=\omega,
\quad\delta\omega_1=0,
\quad
d\omega_2=0.
$$
Moreover, \eqref{eq145} gives
$$
\Vert\omega_1\Vert_2\lesssim\Vert\eta'\Vert_{1,2}\lesssim\Vert d\lambda_1\Vert_2=\Vert d\omega_1\Vert_2
\quad
\text{and}
\quad
\Vert\omega_2\Vert_2\lesssim\Vert\eta''\Vert_{1,2}\lesssim\Vert\delta\lambda_2\Vert_2=\Vert\delta\omega_2\Vert_2.
$$
This proves \eqref{eq156}. Finally,
$$
\Vert\omega_1\Vert_2^2+\Vert\omega_2\Vert_2^2\lesssim\Vert d\omega_1\Vert_2^2+\Vert\delta\omega_2\Vert_2^2=
\Vert\omega\Vert_2^2,
$$
where the last equality follows from the fact that $\omega=d\omega_1+\delta\omega_2$ is a decomposition of $\omega$ into elements that are orthogonal in $\Omega^\ell L^2(\cM)$.
\end{proof}

As a consequence of \Cref{th:hodge}  we obtain the following Hodge decomposition result for pullbacks:
\begin{theorem}
\label{th:hodgepullback}
Let $\cM$ be a compact and oriented $k$-dimensional Riemannian manifold without boundary and assume that
$H^\ell(\mathcal{M})=0$ for some $1 \leq \ell \leq k-1$. 

If $\kappa \in\Omega_c^\ell(\R^N)$ and $f \in C^\infty(\mathcal{M};\R^N)$, then there are forms
$\omega_f\in\Omega^{\ell-1}(\cM)$ and $\beta_f\in\Omega^{\ell+1}(\cM)$, such that
\begin{equation}
\label{eq146}
f^\ast(\kappa) = d\omega_f + \delta \beta_f 
\quad 
\text{and}
\quad
\delta\omega_f=0,\ d\beta_f=0,
\end{equation}
and
\begin{equation}
\label{eq159}
\Vert\omega_f\Vert_2\lesssim_\cM\Vert d\omega_f\Vert_2,
\qquad
\Vert\beta_f\Vert_2\lesssim_\cM\Vert\delta\beta_f\Vert_2,
\end{equation}
\begin{equation}
\label{eq151}
\Vert\omega_f\Vert_2+\Vert\beta_f\Vert_2\lesssim_\cM \Vert f^*(\kappa)\Vert_2.
\end{equation}
In the case when $N = 2n+1$, $H^n(\cM)=0$, and $\kappa \in\Omega_c^n(\R^{2n+1})$, we additionally have
\begin{equation}
\label{eq160}
 \|\delta \beta_f\|_{2} \lesssim_{\kappa,\cM} \Vert f^\ast(\upalpha)\Vert_{\infty}\ \brac{1+\Vert Df\Vert_{\infty}^{n}},
\end{equation}
where $\upalpha$ is the standard contact form from Definition~\ref{def:alpha}. 
\end{theorem}
\begin{proof}
Let $\omega:=f^*(\kappa)$ and let $\omega_f:=\omega_1$, $\beta_f:=\omega_2$ be the forms from \Cref{th:hodge}. Then \eqref{eq146}, \eqref{eq159} and \eqref{eq151} follow immediately from \Cref{th:hodge}.

It remains to prove \eqref{eq160}. From the integration by parts \eqref{eq137} we have
$$
\Vert\delta\beta_f\Vert_2^2=\int_\cM\langle f^*(\kappa)-d\omega_f,\delta\beta_f\rangle=
\int_\cM\langle f^*(d\kappa),\beta_f\rangle.
$$
By assumption $\R^N = \R^{2n+1}$, and $d\kappa$ is an $(n+1)$-form. From Corollary~\ref{T49} we find 
$\tau \in \Omega_c^n(\R^{2n+1})$ and $\gamma \in\Omega_c^{n-1}(\R^{2n+1})$ such that 
\[
d\kappa = \tau \wedge \upalpha + d(\gamma \wedge \upalpha).
\]
Note that
\begin{equation}
\label{eq161}
\Vert f^*(\tau)\Vert_2\lesssim_\cM\Vert f^*(\tau)\Vert_\infty\lesssim_{\kappa,\cM}\Vert Df\Vert_\infty^n,
\quad
\Vert f^*(\gamma)\Vert_2\lesssim_\cM\Vert f^*(\gamma)\Vert_\infty\lesssim_{\kappa,\cM}\Vert Df\Vert_\infty^{n-1}.
\end{equation}
The last constants depend on $\kappa$ and $\cM$ only, because the forms $\tau$ and $\gamma$ are uniquely determined by $\kappa$ and they are bounded.

Now, integration by parts and \eqref{eq159} yield
\[
\begin{split}
\Vert\delta\beta_f\Vert_2^2
&=
\int_\cM \langle f^*(\tau)\wedge f^*(\upalpha),\beta_f\rangle+
\langle f^*(\gamma)\wedge f^*(\upalpha),\delta\beta_f\rangle\\
&\leq\Vert f^*(\upalpha)\Vert_\infty\big(\Vert f^*(\tau)\Vert_2\Vert\beta_f\Vert_2+\Vert f^*(\gamma)\Vert_2\Vert\delta\beta_f\Vert_2\big)\\
&\lesssim_\cM 
\Vert f^*(\upalpha)\Vert_\infty\big(\Vert f^*(\tau)\Vert_2+\Vert f^*(\gamma)\Vert_2\big)\Vert\delta\beta_f\Vert_2.
\end{split}
\]
This and \eqref{eq161} give 
$$
\Vert\delta\beta_f\Vert_2\lesssim_{\kappa,\cM}\Vert f^*(\upalpha)\Vert_\infty
\big(\Vert Df\Vert_\infty^n+\Vert Df\Vert_\infty^{n-1}\big)
\lesssim
\Vert f^*(\upalpha)\Vert_\infty
\big(1+\Vert Df\Vert_\infty^n\big).
$$
The proof is complete. 
\end{proof}
\section{Nontriviality of H\"older homotopy groups}
\label{NHG}

\subsection{H\"older homotopy groups}
H\"older homotopy groups are defined as homotopy groups with the additional assumptions that all maps and homotopies involved are H\"older continuous. 

More precisely, for a pointed metric space $(X,x_0)$, $k\in\bbbn$, and $0<\gamma\leq 1$, we define the group $\pi_k^\gamma(X,x_0)$ as follows. Let $I^k = [0,1]^k$ be the $k$-dimensional cube, then we set 
\[
\pi_k^\gamma(X,x_0) = \left \{f \in C^{0,\gamma}(I^k; X):\, \quad f \equiv x_0 \quad \mbox{in a neighborhood of $\partial I_k$} \right \}/ \sim
\]
where $\sim$ is the equivalence relation given by
\[
f \sim g \quad \Longleftrightarrow \mbox{there is $H \in C^{0,\gamma}([0,1] \times I^k; X)$ with} 
\begin{cases} 
H(0,\cdot) = f(\cdot)\\
H(1,\cdot) = g(\cdot)\\
H(t,\cdot) \equiv x_0\ \mbox{near $\partial I_k$.}
\end{cases}
\]
Although we will make no further mention of it, the group structure of $\pi_k^\gamma(X,x_0)$ is then defined as for the usual homotopy group $\pi_k(X,x_0)$. 

We say that a metric space $X$ is {\em rectifiably connected} if any two points in $X$ can be connected by a curve of finite length or equivalently, by a Lipschitz curve. 
Equivalence follows from a well known fact that any curve of finite length admits an arc-length parametrization that is $1$-Lipschitz, see e.g., \cite[Theorem~3.2]{Haj2003}.

If $X$ is rectifiably connected, then for any $x_1,x_2\in X$, the groups $\pi_k^\gamma(X,x_1)$ and $\pi_k^\gamma(X,x_2)$ are isomorphic and we will simply wite $\pi_k^\gamma(X)$. Note that the Heisenberg group $\bbbh^n$ is rectifiably connected.

We will be mainly interested in determining whether $\pi_k^\gamma(\bbbh^n)\neq 0$. To do so the following observation will be useful. The proof is standard and left to the reader.
\begin{lemma}
\label{T105}
Assume that a metric space $X$ is rectifiably connected and let $0<\gamma\leq 1$. Then the following conditions are equivalent:
\begin{enumerate}
\item $\pi_k^\gamma(X)\neq 0$;
\item For any $f\in C^{0,\gamma}(\Sph^k;X)$ there is $H\in C^{0,\gamma}(\Sph^k\times [0,1];X)$ and $x_0\in X$ such that $H(\cdot,1)=f(\cdot)$ and $H(\cdot,0)\equiv x_0$;
\item For any $f\in C^{0,\gamma}(\Sph^k;X)$ there is $F\in C^{0,\gamma}(\overline{\bbbb}^{k+1};X)$ such that $F(\theta)=f(\theta)$ for all $\theta\in\Sph^k=\partial\overline{\bbbb}^{k+1}$.
\end{enumerate}
\end{lemma}

Let us start with a simple observation.
\begin{proposition}
\label{T106}
If $k,n\geq 1$ and $0<\gamma\leq\frac{1}{2}$, then $\pi_k^\gamma(\bbbh^n)=0$.
\end{proposition}
\begin{proof}
Recall that $\delta_t:\bbbh^n\to\bbbh^n$ is the dilation defined in \eqref{eq147}.

If $f\in C^{0,\gamma}(\Sph^k;\bbbh^n)$, then
$F\in C^{0,\gamma}(\Sph^k\times[0,1];\bbbh^n)$, where
$$
F(\theta,t)=\delta_t f(\theta)
\quad
\text{for }
(\theta,t)\in\Sph^k\times[0,1].
$$
Indeed, \eqref{eq28} and \eqref{eq142} yield
\[
\begin{split}
&
d_K(F(\theta,t),F(\xi,s))
\leq d_K(\delta_t f(\theta),\delta_tf(\xi))+d_K(\delta_t f(\xi),\delta_sf(\xi))\\
&\lesssim
td_K(f(\theta),f(\xi))+|t-s|\, \Vert f\Vert_\infty +|t^2-s^2|^{1/2}\Vert f\Vert_\infty^{1/2}
\lesssim
(|\theta-\xi|+|t-s|)^\gamma.
\end{split}
\]
Here $\Vert f\Vert_\infty$ denotes the supremum norm of $f$ with respect to the Euclidean norm in $\R^{2n+1}$.
Since $F(\cdot,1)=f(\cdot)$ and $F(\cdot,0)\equiv0$, the result follows from Lemma~\ref{T105}.
\end{proof}

When $\gamma > \frac{1}{2}$ the situation is more complicated. If $\gamma = 1$, we have Lipschitz homotopy groups which we shall denote by $\pi_k^{\lip}(\mathbb{H}^n)$. The following result is known.
\begin{theorem}${}$
\label{T111}
\begin{enumerate}
 \item $\pi^{\lip}_m(\mathbb{H}^n)=0$ for all $1\leq m<n$;
\item $\pi_m^{\lip}(\mathbb{H}^1)=0$ for all $m\geq 2$;
 \item $\pi^{\lip}_n(\mathbb{H}^n)\neq 0$;
 \item $\pi^{\lip}_{4n-1}(\mathbb{H}^{2n}) \neq 0$; 
 \item $\pi^{\lip}_{n+1}(\mathbb{H}^{n}) \neq 0$. 
 \item $\pi^{\lip}_{n+k(n-1)}(\mathbb{H}^n) \neq 0$ for $k = 0,1,2,\ldots$
\end{enumerate}
 \end{theorem}
(1) was proven in \cite[Theorem~1.1]{WengerY1}, see also \cite[Theorem~4.11]{DHLT}. (2) is due to \cite[Theorem~5]{WengerY2}. (3) is due to \cite{Balogh-Faessler-2009}. For other proofs, see \cite[Theorem~4.11]{DHLT} and \cite[Theorem~1.7]{HST14}. (4) was the main result of \cite{HST14}. (5) is due to \cite{H18}. (6) is proven in the forthcoming paper \cite{EHPS}. Currently the authors of \cite{EHPS} are working on extending the result to the H\"older category. Manin \cite{manin} proved a similar result to (6) in the Lipschitz category, but with a different technique.

In the remaining part of \Cref{NHG}, we will extend (3) and (4) to the H\"older case, see Theorems~\ref{T108} and~\ref{th:hopfgamma} respectively. 

The proofs of (3) and (4) in the Lipschitz case were based on the fact that if $f \in C^{0,1}(\Sph^k;\mathbb{H}^n)$ then $\rank Df \leq n$ almost everywhere and this was combined with the fact that intgrals of pullbacks of differential forms (degree or Hopf invariant) could detect non-trivial elements of the homotopy groups of spheres. 
In \Cref{T47} we have found an alternative formulation for the notion of $\rank Df \leq n$, for H\"older continuous maps, and we combine this notion with the previous arguments to obtain our results.

To link the homotopy groups of spheres with the H\"older homotopy groups of $\bbbh^n$ we will need the followings result whose proof can be found in \cite[Section~4]{Balogh-Faessler-2009}, \cite[Theorem~3.2]{DHLT}, \cite[Example~3.1]{EkholmEtnyreSullivan05}.
\begin{lemma}
\label{T107}
For each $n\geq 1$, there is a smooth embedding $\phi:\Sph^n\to\bbbr^{2n+1}$, such that $\phi^*(\upalpha)=0$. As a consequence $\phi$ is a bi-Lipschitz embedding of $\Sph^n$ into $\bbbh^n$.
\end{lemma}

The next result is a straightforward consequence of Lemma~\ref{T107} and Corollary~\ref{T97}.

\begin{theorem}
\label{T108}
For any  $n \geq 1$ we have $\pi^{\gamma}_n(\mathbb{H}^n)\neq 0$ for all $\gamma \in \big(\frac{n+1}{n+2},1\big]$.
\end{theorem}
\begin{proof}
If $\phi:\Sph^n\to\bbbh^n$ is a bi-Lipschitz embedding from Lemma~\ref{T107}, then according to Corollary~\ref{T97},
$f$ cannot be extended to a map $\tilde{\phi}\in C^{0,\gamma}(\overbar{\bbbb}^{n+1};\bbbh^n)$ whenever 
$\gamma \in \big(\frac{n+1}{n+2},1\big]$. Hence Lemma~\ref{T105} yields that $0\neq[\phi]\in\pi_n^\gamma(\bbbh^n)$.
\end{proof}
\begin{remark}
A straightforward proof of Theorem~\ref{T108}, that does not refer to Theorem~\ref{T53}, follows also from the argument used in Section~\ref{GroPr}.
\end{remark}

The main result of \Cref{NHG} is, however, \Cref{th:hopfgamma}. It is much more difficult than \Cref{T108}. The remaining part of the paper is solely devoted to its proof.
\begin{theorem}
\label{th:hopfgamma}
For any  $n \geq 1$, 
we have $\pi^{\gamma}_{4n-1}(\mathbb{H}^{2n})\neq 0$ for all $\gamma \in \big(\frac{4n+1}{4n+2},1\big]$.
\end{theorem}
In order to prove \Cref{th:hopfgamma} we need to extend the strategy from \cite{HST14} and use the Hopf invariant combined with the new notion of the ``rank of the derivative'' for H\"older maps, \Cref{T47}. Classically, the Hopf invariant $\HI(f)$ is a homotopy invariant defined on maps $f: \mathbb{S}^{4n-1} \to \mathbb{S}^{2n}$, but we need to extend it to maps $f:\Sph^{4n-1}\to\R^{4n+1}$.

\subsection{Generalized Hopf invariant}
Let us start with recalling the definition of the classical Hopf invariant (cf.\ \cite{BT82}).

Let $\nu_{\Sph^{2n}}\in\Omega^{2n}(\Sph^{2n})$ be the volume form on $\Sph^{2n}$. For any smooth map $f:\Sph^{4n-1}\to\Sph^{2n}$, we have $df^*(\nu_{\Sph^{2n}})=f^*(d\nu_{\Sph^{2n}})=0$. Since $H^{2n}(\Sph^{4n-1})=0$, it follows that $f^*(\nu_{\Sph^{2n}})=d\omega$ for some $\omega\in \Omega^{2n-1}(\Sph^{4n-1})$. Then, the {\em Hopf invariant} is defined by
$$
\HI(f):=\int_{\Sph^{4n-1}}\omega\wedge d\omega=\int_{\Sph^{4n-1}}\omega\wedge f^*(\nu_{\Sph^{2n}}).
$$
If $f,g\in C^\infty(\Sph^{4n-1};\Sph^{2n})$ are homotopic, then $\HI(f)=\HI(g)$.

Hopf~\cite[Satz II, Satz II']{Hopf} proved the following important result which along with the homotopy invariance of $\HI(\cdot)$ implies that $\pi_{4n-1}(\Sph^{2n})\neq 0$.
\begin{lemma}
\label{la:hopffibration}
For any $n \in \mathbb{N}$ there exists a smooth map
$g: \mathbb{S}^{4n-1} \to \mathbb{S}^{2n}$ such that $\HI (g) \neq 0$.
\end{lemma}
Our aim is to generalize the Hopf invariant to the case of mappings $f:\Sph^{4n-1}\to\R^{4n+1}$. The construction presented below is a modification of that given in \cite{HST14}.

Fix any form $\kappa\in\Omega_c^{2n}(\R^{4n+1})$. Let $f\in C^\infty(\Sph^{4n-1};\R^{4n+1})$. Then $f^*(\kappa)\in\Omega^{2n}(\Sph^{4n-1})$. Since $H^{2n}(\Sph^{4n-1})=0$, Theorem~\ref{th:hodgepullback} gives the Hodge decomposition
\begin{equation}
\label{eq162}
f^*(\kappa)=d\omega_f+\delta\beta_f,
\quad
\omega_f\in\Omega^{2n-1}(\Sph^{4n-1}),\ 
\beta_f\in\Omega^{2n+1}(\Sph^{4n-1}),
\quad
\delta\omega_f=0,\ d\beta_f=0.
\end{equation}
Then, we define the {\em generalized Hopf invariant} by
$$
\HI_\kappa(f):=\int_{\Sph^{4n-1}}\omega_f\wedge d\omega_f.
$$
Note that
\begin{equation}
\label{eq169}
\HI_\kappa(f)=\int_{\Sph^{4n-1}} \omega_f\wedge f^*(\kappa).
\end{equation}
Indeed, $d\omega_f=f^*(\kappa)-\delta\beta_f$ and the equality follows from integration by parts, because $\delta\omega_f=0$.

Observe also that $\HI_\kappa(f)$ does not depend on a particular choice of the Hodge decomposition in \eqref{eq162}. That is, if
$$
f^*(\kappa)=d\tau+\delta\gamma,
\quad
\delta\tau=0,
\quad
d\gamma=0,
$$
then
\begin{equation}
\label{eq170}
\HI_\kappa(f)=\int_{\Sph^{4n-1}}\tau\wedge d\tau=\int_{\Sph^{4n-1}}\tau\wedge f^*(\kappa).
\end{equation}
Indeed, the second equality follows from the same argument as in the proof of \eqref{eq169}. As for the first equality, in \Cref{th:hodge} we proved that $d\tau=d\omega_f$ and hence integration by parts gives
$$
\HI_\kappa(f)-\int_{\Sph^{4n-1}}\tau\wedge d\tau=\int_{\Sph^{4n-1}}\omega_f\wedge d\omega_f -
\int_{\Sph^{4n-1}}\tau\wedge d\tau=
\int_{\Sph^{4n-1}}(\omega_f-\tau)\wedge d\omega_f=0,
$$
because $d(\omega_f-\tau)=0$.
However, in what follows, we will use the Hodge decomposition \eqref{eq162} from \Cref{th:hodgepullback}, because we will need to refer to estimates proved in \Cref{th:hodgepullback}.

It is important to be aware that $\HI_\kappa(\cdot)$ depends on the choice of the form $\kappa$.

We want to emphasize that $\HI_\kappa(f)$ is {\em not} an invariant of homotopies of $f:\Sph^{4n-1}\to\R^{4n+1}$, unless $\HI_\kappa(\cdot)\equiv 0$. 
It was proved in \cite{HST14} that $\HI_{\kappa}(\cdot)$ restricted to the class of Lipschitz maps
$f\in C^{0,1}(\Sph^{4n-1};\R^{4n+1})$ satisfying 
\[
\rank Df \leq 2n \quad \text{almost everywhere in $\Sph^{4n-1}$,}                                                      \]
is invariant under Lipschitz homotopies satisfying the same rank condition,
$\rank DH\leq 2n$ almost everywhere in $\Sph^{4n-1}\times [0,1]$. Since the rank of the derivative of any Lipschitz mapping into $\bbbh^{2n}$ is less than or equal $2n$, $\HI_\kappa(\cdot)$ is an invariant of Lipschitz homotopies of Lipschitz maps $f\in C^{0,1}(\Sph^{4n-1};\bbbh^{2n})$.
On the other hand, H\"older continuous mappings into $\bbbh^{2n}$ (and hence into $\R^{4n+1})$ need not be differentiable at any point, but \Cref{T47} (see also \Cref{R47}) is a substitute for the condition about the rank of the derivative. Using this idea,  we will prove in \Cref{pr:homotopy} that $\HI_\kappa(\cdot)$
is an invariant of H\"older homotopies of H\"older maps into the Heisenberg group $\bbbh^{2n}$. 
For a precise statement, see \Cref{pr:homotopy} and \Cref{R12}.
This will be a far reaching and non-trivial generalization of the results from \cite{HST14}.

While it can happen that $\HI_\kappa(\cdot)\equiv 0$ (for example if $\kappa=0$), in order to prove that $\pi_{4n-1}^\gamma(\bbbh^{2n})\neq 0$ (see \Cref{th:hopfgamma}), we will need to select a suitable form $\kappa$ such that $\HI_\kappa(f)\neq 0$ for some smooth map $f:\Sph^{4n-1}\to\bbbh^{2n}=\R^{4n+1}$. This and \Cref{pr:homotopy} will imply that $0\neq[f]\in\pi^\gamma_{4n-1}(\bbbh^{2n})$. The form $\kappa$ will be a compactly supported extension of the volume form of the embedded sphere $\Sph^{2n}$ into $\R^{4n+1}$ discussed in \Cref{T107}. For details, see \Cref{LastProof}.

\begin{remark}
\label{R3}
Note that if $f$ is smooth and $\rank Df\leq 2n$, then $d(f^*(\kappa))=f^*(d\kappa)=0$, because $d\kappa$ is a $(2n+1)$-form. Hence $f^*(\kappa)=d\omega_f$ i.e., $\delta\beta_f=0$. For any smooth map $f$, the term $\delta\beta_f$ measures how far $f^*(\kappa)$ is from being exact. If $f:\Sph^{4n-1}\to\bbbh^{2n}$ is in a suitable H\"older class, then ``$\rank Df\leq 2n$'' in the distributional sense of \Cref{T47}. Therefore, one should expect that if $f_\eps$ is an approximation of $f$ by mollification, then suitable integrals involving $\delta\beta_{f_\eps}$ would converge to zero as $\eps\to 0$, and the key estimate to achieve that is \eqref{eq160} combined with \Cref{T45}. This idea is used in the proof of the next result, see for example \eqref{eq149}. It allows us to extend the proof of \cite[Proposition~5.8]{HST14} to the H\"older case. We hope that this explanation will help understand the proof.
\end{remark}
\begin{remark}
\label{R12}
\Cref{pr:homotopy} allows us to define $\HI_\kappa(f)$ when 
$\gamma\in\big(\frac{4n+1}{4n+2},1\big]$ and 
$f\in C^{0,\gamma}(\Sph^{4n-1};\bbbh^{2n})$ is $C^{0,\gamma}$-homotopic to a map in $C^\infty\cap C^{0,\gamma}(\Sph^{4n-1};\bbbh^{2n})$.
 
If $\gamma\in\big(\frac{4n-1}{4n},1\big]$, then one can prove that
\begin{equation}
\label{eq173}
|\HI_\kappa(f)-\HI_\kappa(g)|\lesssim
\big(1+\Vert f\Vert_{C^{0,\gamma}}^{4n+1}+\Vert g\Vert_{C^{0,\gamma}}^{4n+1}\big)\Vert f-g\Vert_{C^{0,\gamma}}
\text{ for }
f,g\in C^\infty(\Sph^{4n-1};\R^{4n+1}).
\end{equation}
This estimate allows one to define $\HI_\kappa(f):=\lim_{\eps\to 0}\HI_\kappa(f_\eps)$ for all $f\in C^{0,\gamma}(\Sph^{4n-1};\R^{4n+1})$, and hence for all $f\in C^{0,\gamma}(\Sph^{4n-1};\bbbh^{4n+1})$. However, regarding the $C^{0,\gamma}$-homotopy invariance of maps in $f\in C^{0,\gamma}(\Sph^{4n-1};\bbbh^{4n+1})$, we could prove it only for $\gamma\in\big(\frac{4n+1}{4n+2},1\big]$, which is the same range as in \Cref{pr:homotopy}. 

One can prove \eqref{eq173} using hard harmonic analysis related to the paraproduct estimates in \cite{SY1999,LS19}. Alternatively, one can prove \eqref{eq173} using methods from \cite{SVS}.

However, the proof of \Cref{pr:homotopy} presented below is substantially easier, and arguably more elegant; although we do not prove \eqref{eq173}, \Cref{pr:homotopy} suffices to prove \Cref{th:hopfgamma}.
\end{remark}
\begin{theorem}
\label{pr:homotopy}
Let $\gamma\in\big(\frac{4n+1}{4n+2},1\big]$ and let $\kappa\in \Omega_c^{2n}(\R^{4n+1})$. The generalized Hopf invariant $\HI_\kappa(\cdot)$ 
is a $C^{0,\gamma}$-homotopy invariant of mappings in the class $C^\infty\cap C^{0,\gamma}(\Sph^{4n-1};\mathbb{H}^{2n})$.

More precisely, assume that $f,g \in C^\infty\cap C^{0,\gamma}(\mathbb{S}^{4n-1};\mathbb{H}^{2n})$ are $C^{0,\gamma}(\mathbb{S}^{4n-1};\mathbb{H}^{2n})$-homotopic. That is, assume there exists $H \in C^{0,\gamma}([0,1] \times \mathbb{S}^{4n-1}; \mathbb{H}^{2n})$ such that $H(0,\cdot) = f(\cdot)$ and $H(1,\cdot) = g(\cdot)$. Then
\[
 \mathcal{H}_\kappa(f) = \mathcal{H}_\kappa(g).
\]
\end{theorem}
\begin{remark}
While $f,g\in C^\infty$, we require only H\"older continuity for $H$.
\end{remark}
\begin{proof}
Let $\varphi: \R /2\Z \to [0,1]$ be a Lipschitz map such that $\varphi(t) = 0$ for all $t \approx 0$ and $\varphi(t) = 1$ for all $t \approx 1$. Considering $H(\varphi(t),x)$ instead of $H(t,x)$ we can assume without lost of generality that for $\delta =\frac{1}{10}$ we have $H \in C^{0,\gamma}(\R/2\Z\times \Sph^{4n-1};\mathbb{H}^{2n})$ with 
\begin{equation}
\label{eq:hom:1}
H(t,x) = f(x) \quad \mbox{for $|t| < 2\delta$} \quad \text{and}\quad  H(t,x) = g(x) \quad \text{for $|t - 1| < 2\delta$}
\end{equation}
Denote by $f_\eps$, $g_\eps$, and $H_\eps$, the mollifications of $f$, $g$, and $H$, as maps on $\Sph^{4n-1}$ and $\R /2\Z \times \Sph^{4n-1}$ respectively. Such mollifications were used in \Cref{S3}. In particular we have
\begin{equation} 
\label{eq:fepsgepsest}
\Vert Df_\eps\Vert_{\infty} + \Vert Dg_\eps\Vert_{\infty} + \Vert DH_\eps\Vert_{\infty} \leq \eps^{\gamma-1} C(f,g,H).
\end{equation}
The proof follows from \eqref{eq163} combined with the arguments used in the proof of \Cref{T50}.

In view of \eqref{eq:hom:1}, for small enough $\eps > 0$, we still have
\begin{equation}
\label{eq:hom:2}
H_\eps(t,x) = f_\eps(x) \quad \mbox{for $|t| < \delta$} \quad \text{and}\quad  H_\eps(t,x) = g_\eps(x) \quad \text{for $|t-1|<\delta$}.
\end{equation}
Since $\R/2\mathbb{Z}$ is diffeomorphic to $\Sph^1$, K\"unneth formula, \cite[p.\ 47]{BT82}, yields 
\[
 H^{2n}(\R / 2\Z \times \Sph^{4n-1})=0
\quad
\text{for all } n\geq 1. 
\]
Thus, we may apply Hodge decomposition, \Cref{th:hodgepullback},
\[
 f_\eps^\ast (\kappa) = d\omega_{f_\eps} + \delta \beta_{f_\eps} \quad \mbox{in $\mathbb{S}^{4n-1}$},
\]
\[
 g_\eps^\ast (\kappa) = d\omega_{g_\eps} + \delta \beta_{g_\eps} \quad \mbox{in $\mathbb{S}^{4n-1}$}.
\]
\[
 H_\eps^\ast (\kappa) = d\omega_{H_\eps} + \delta \beta_{H_\eps} \quad \mbox{in $\R / 2\Z \times \mathbb{S}^{4n-1}$}.
\]
Note that since $f\in C^\infty$, we have
\begin{equation}
\label{eq164}
 \HI_{\kappa}(f) 
=
\lim_{\eps \to 0} \int_{\Sph^{4n-1}} \omega_{f_\eps} \wedge f_{\eps}^\ast(\kappa).
\end{equation}
Indeed, according to \Cref{th:hodge}, we have a Hodge decomposition
$$
f_\eps^*(\kappa)-f^*(\kappa)=d\tau_\eps+\delta\gamma_\eps,
\quad
\delta\tau_\eps=0,
\quad
d\gamma_\eps=0,
$$
satisfying
\begin{equation}
\label{eq171}
\Vert\tau_\eps\Vert_2\lesssim\Vert f_\eps^*(\kappa)-f^*(\kappa)\Vert_2.
\end{equation}
Hence,
$$
f^*_\eps(\kappa)=d(\tau_\eps+\omega_f)+\delta(\gamma_\eps+\beta_f).
$$
Although, there is no reason to claim that $\tau_\eps+\omega_f=\omega_{f_\eps}$ or $\gamma_\eps+\beta_f=\beta_{f_\eps}$, \eqref{eq170} yields
$$
\HI_\kappa(f_\eps)=\int_{\Sph^{4n-1}}(\tau_\eps+\omega_f)\wedge f_\eps^*(\kappa).
$$
Since
$f_\eps^*(\kappa)\rightrightarrows f^*(\kappa)$ uniformly, \eqref{eq171} yields that $\tau_\eps+\omega_f\to\omega_f$ in $L^2$, and hence $\HI_\kappa(f_\eps)\to\HI_\kappa(f)$.

Let
\[
\imath_t:\mathbb{S}^{4n-1}\to \{ t\}\times \mathbb{S}^{4n-1}\subset \R / 2\Z\times\mathbb{S}^{4n-1}.
\]
Then $H_\eps \circ \imath_t = f_\eps$ for $|t|<\delta$ and $H_\eps \circ \imath_t = g_\eps$ for $|t-1|<\delta$. 
In particular
$$
f_\eps^*(\kappa)=\imath_t^*H_\eps^*(\kappa) \text{ if } |t|<\delta
\qquad
\text{and}
\qquad
g_\eps^*(\kappa)=\imath_t^*H_\eps^*(\kappa) \text{ if } |t-1|<\delta.
$$
We claim that 
\begin{equation}
\label{eq152}
\HI_{\kappa}(f) 
= 
\lim_{\eps\to 0} \mvint_{[-\delta,\delta]}\Bigg(\, \int_{\Sph^{4n-1}} \imath^\ast_{t} \brac{\omega_{H_{\eps}} \wedge H_{\eps}^\ast(\kappa)}\Bigg)\, dt,
\end{equation}
and similarly
\begin{equation}
\label{eq153}
\begin{split}
\HI_{\kappa}(g) 
&= 
\lim_{\eps\to 0} \mvint_{[1-\delta,1+\delta]}\bigg(\, \int_{\Sph^{4n-1}} \imath^\ast_{t} \brac{\omega_{H_{\eps}} \wedge H_{\eps}^\ast(\kappa)}\bigg)\, dt\\
&=
\lim_{\eps\to 0} \mvint_{[-\delta,\delta]}\bigg(\, \int_{\Sph^{4n-1}} \imath^\ast_{t+1} \brac{\omega_{H_{\eps}} \wedge H_{\eps}^\ast(\kappa)}\bigg)\, dt.
\end{split}
\end{equation}
Recall that the barred integral stands for the integral average, see \eqref{eq172}.
We will prove \eqref{eq152}; the proof of \eqref{eq153} is the same.
Note that \eqref{eq164} trivially gives
\begin{equation}
\label{eq165}
\HI_{\kappa}(f) 
=
\lim_{\eps \to 0} \mvint_{[-\delta,\delta]}\Bigg(\, \int_{\Sph^{4n-1}} \omega_{f_\eps} \wedge f_{\eps}^\ast(\kappa)\Bigg)\, dt,
\end{equation}
because the integral over $\Sph^{4n-1}$ does not depend on $t$.

In the calculations presented below one has to be careful: While $d\imath_t^*=\imath^*_td$, the co-differential does not commute with pullback, $\delta \imath_t^*\neq\imath_t^*\delta$.

For every $t \in (-\delta,\delta)$ we have
\[
\begin{split}
&
\int_{\Sph^{4n-1}} \omega_{f_\eps} \wedge f_{\eps}^\ast(\kappa)
= 
\int_{\Sph^{4n-1}} \omega_{f_\eps} \wedge \imath_t^\ast H_\eps^\ast(\kappa)\\
&= 
\int_{\Sph^{4n-1}} \omega_{f_\eps} \wedge \imath_t^\ast d\omega_{H_\eps} + \int_{\Sph^{4n-1}} \omega_{f_\eps} \wedge \imath_t^\ast \delta \beta_{H_\eps}\\
&\overset{\text{Stokes}}{=} 
\int_{\Sph^{4n-1}} \imath_t^\ast \omega_{H_\eps} \wedge d\omega_{f_\eps} + \int_{\Sph^{4n-1}} \omega_{f_\eps} \wedge \imath_t^\ast \delta \beta_{H_\eps}\\
&= 
\int_{\Sph^{4n-1}} \imath_t^\ast \omega_{H_\eps} \wedge f_\eps^\ast(\kappa) 
-\int_{\Sph^{4n-1}} \imath_t^\ast \omega_{H_\eps} \wedge \delta \beta_{f_\eps}
+ \int_{\Sph^{4n-1}} \omega_{f_\eps} \wedge \imath_t^\ast \delta \beta_{H_\eps}\\
&= 
\int_{\Sph^{4n-1}} \imath_t^\ast \brac{\omega_{H_\eps} \wedge H_\eps^\ast(\kappa)} 
-\int_{\Sph^{4n-1}} \imath_t^\ast \omega_{H_\eps} \wedge \delta \beta_{f_\eps}
+ \int_{\Sph^{4n-1}} \omega_{f_\eps} \wedge \imath_t^\ast \delta \beta_{H_\eps}.
\end{split}
\]
Therefore, in order to conclude \eqref{eq152} from \eqref{eq165}, it suffices to show that
\begin{equation} 
\label{eq149}
\lim_{\eps\to 0}\int_{-\delta}^\delta 
\Bigg|\int_{\Sph^{4n-1}} \imath_{t}^\ast \omega_{H_\eps} \wedge \delta \beta_{f_\eps}\Bigg|\, dt
=
\lim_{\eps\to 0}\int_{-\delta}^\delta 
\Bigg|\int_{\Sph^{4n-1}} \omega_{f_\eps} \wedge \imath_{t}^\ast \delta \beta_{H_\eps}\Bigg|\, dt = 0.
\end{equation}
We have
\[
\begin{split}
\int_{-\delta}^\delta 
\Bigg|\int_{\Sph^{4n-1}} \imath_{t}^\ast \omega_{H_\eps} \wedge \delta \beta_{f_\eps}\Bigg|\, dt
&\leq
\int_{[-\delta,\delta]\times\Sph^{4n-1}} |\omega_{H_\eps}|\, |\delta\beta_{f_\eps}|
\leq (2\delta)^{1/2}\Vert\omega_{H_\eps}\Vert_2\Vert\delta\beta_{f_\eps}\Vert_2\\
&\lesssim
\Vert H_\eps^*(\kappa)\Vert_2(1+\Vert Df_\eps\Vert_\infty^{2n})\Vert f_\eps^*(\upalpha)\Vert_\infty =:\heartsuit.
\end{split}
\]
In the last inequality we used \eqref{eq151} and  \eqref{eq160} respectively. The factor $(2\delta)^{1/2}$ appears, because $\delta\beta_{f_\eps}$ is defined on $\Sph^{4n-1}$ and does not depend on the variable $t\in[-\delta,\delta]$.

Clearly,
$\Vert H_\eps^*(\kappa)\Vert_2\lesssim\Vert DH_\eps\Vert_\infty^{2n}$, so \eqref{eq:fepsgepsest} and \Cref{T45} give
$$
\heartsuit\lesssim\eps^{2n(\gamma-1)}\cdot\eps^{2n(\gamma-1)}\cdot\eps^{2\gamma-1}\xrightarrow{\eps\to 0} 0,
$$
because $4n(\gamma-1)+2\gamma-1>0$ for $\gamma>\frac{4n+1}{4n+2}$.
The proof that 
$$
\lim_{\eps\to 0}\int_{-\delta}^\delta 
\Bigg|\int_{\Sph^{4n-1}} \omega_{f_\eps} \wedge \imath_{t}^\ast \delta \beta_{H_\eps}\Bigg|\, dt = 0
$$
is similar and left to the reader. This completes the proof of \eqref{eq149} and hence that of \eqref{eq152}. The proof of \eqref{eq153} is the same.

Since,
\begin{equation}
\label{eq154}
\begin{split}
&
\HI_{\kappa}(f) - \HI_{\kappa}(g)\\ 
&= \lim_{\eps\to 0}\mvint_{[-\delta,\delta]}\Bigg(\, \int_{\Sph^{4n-1}} \imath^\ast_{t} \brac{\omega_{H_{\eps}} \wedge H_{\eps}^\ast(\kappa)} -\int_{\Sph^{4n-1}} \imath^\ast_{t+1} \brac{\omega_{H_{\eps}} \wedge H_{\eps}^\ast(\kappa)}\Bigg)\, ,
\end{split}
\end{equation}
in order to prove that $\HI_\kappa(f)=\HI_\kappa(g)$, it suffices to show that the above limit equals zero.
We have
\[
\begin{split}
&\int_{\Sph^{4n-1}} \imath^\ast_{t+1} \brac{\omega_{H_{\eps}} \wedge H_{\eps}^\ast(\kappa)}-
\int_{\Sph^{4n-1}} \imath^\ast_{t} \brac{\omega_{H_{\eps}} \wedge H_{\eps}^\ast(\kappa)}\\
\overset{\text{Stokes}}{=}&
\int_{[t,t+1] \times \Sph^{4n-1}} d\brac{\omega_{H_{\eps}} \wedge H_{\eps}^\ast(\kappa)}\\
=&
\int_{[t,t+1] \times \Sph^{4n-1}} d\omega_{H_{\eps}} \wedge H_{\eps}^\ast(\kappa) - \int_{[t,t+1] \times \Sph^{4n-1}} \omega_{H_{\eps}} \wedge H_{\eps}^\ast(d\kappa) \\
=&
\int_{[t,t+1] \times \Sph^{4n-1}} H_{\eps}^\ast(\kappa\wedge\kappa)
- \int_{[t,t+1] \times \Sph^{4n-1}} \delta \beta_{H_{\eps}} \wedge H_{\eps}^\ast(\kappa)
- \int_{[t,t+1] \times \Sph^{4n-1}} \omega_{H_{\eps}} \wedge H_{\eps}^\ast(d\kappa).
\end{split}
\]
Therefore, it suffices to show that each of the following integrals converges to zero:
\begin{equation}
\label{eq166}
\begin{split}
&
\lim_{\eps\to 0}\mvint_{[-\delta,\delta]}\bigg(\,\int_{[t,t+1] \times \Sph^{4n-1}} H_{\eps}^\ast(\kappa\wedge\kappa)\Bigg)\, dt
= 
\lim_{\eps\to 0}\mvint_{[-\delta,\delta]}\Bigg(\, \int_{[t,t+1] \times \Sph^{4n-1}} \delta \beta_{H_{\eps}} \wedge H_{\eps}^\ast(\kappa)\Bigg)\, dt\\
&= 
\lim_{\eps\to 0}\mvint_{[-\delta,\delta]}\Bigg(\int_{[t,t+1] \times \Sph^{4n-1}} \omega_{H_{\eps}} \wedge H_{\eps}^\ast(d\kappa)\Bigg)\, dt=0.
\end{split}
\end{equation}

As for the first integral, from Corollary~\ref{T49} we find $\eta_1 \in \Omega_c^{4n-1}(\R^{4n+1})$ and $\eta_2 \in \Omega_c^{4n-2}(\R^{4n+1})$ such that 
\[
 \kappa\wedge\kappa = \eta_1 \wedge \upalpha + d(\eta_2 \wedge \upalpha).
\]
The Stokes theorem yields
\[
\begin{split}
\Bigg| \int_{[t,t+1] \times \Sph^{4n-1}} H_{\eps}^\ast(\kappa\wedge\kappa) \Bigg| 
&\leq 
\int_{[t,t+1] \times \Sph^{4n-1}} |H_{\eps}^*(\eta_1)|\, |H_{\eps}^*(\upalpha)|
+
\int_{\partial([t,t+1] \times \Sph^{4n-1})} |H_{\eps}^*(\eta_2)|\, |H_{\eps}^*(\upalpha)|\\
&\lesssim
\big(\Vert DH_{\eps}\Vert_\infty^{4n-1}+\Vert DH_{\eps}\Vert_\infty^{4n-2}\big)\Vert H_{\eps}^*(\upalpha)\Vert_\infty\\
&\lesssim
\big(\eps^{(4n-1)(\gamma-1)} + \eps^{(4n-2)(\gamma-1)}\big) \eps^{2\gamma-1} 
\xrightarrow{\eps\to 0} 0.
\end{split}
\]
Since the estimates {\em do not} depend on $t$, the first integral in \eqref{eq166} converges to zero.

Regarding the second integral in \eqref{eq166} we have
\[
\begin{split}
&\Bigg| \int_{[t,t+1] \times \Sph^{4n-1}} \delta \beta_{H_{\eps}} \wedge H_{\eps}^\ast(\kappa) \Bigg| 
\leq 
\Vert\delta \beta_{H_{\eps}}\Vert_{2}\ \Vert H_{\eps}^*(\kappa)\Vert_2\\
&\lesssim
\Vert H_{\eps}^*(\upalpha)\Vert_\infty(1+\Vert DH_{\eps}\Vert_\infty^{2n})\Vert DH_{\eps}\Vert_\infty^{2n}
\lesssim
\eps^{2\gamma-1}\cdot\eps^{4n(\gamma-1)}
\xrightarrow{\eps\to 0} 0.
\end{split}
\]
Again, the estimates do not depend on $t$ and hence the second integral in \eqref{eq166} converges to zero.

It remains to show that the third integral in \eqref{eq166} converges to zero. This time the estimates will depend on $t$.

From \Cref{T49} we find $\eta_3 \in \Omega^{2n}_c(\R^{4n+1})$ and $\eta_4 \in \Omega^{2n-1}_c(\R^{4n+1})$ such that 
\[
 d\kappa = \eta_3 \wedge \upalpha + d(\eta_4 \wedge \upalpha).
\]
Stokes' theorem yields
\[
\begin{split}
&
\int_{[t,t+1] \times \Sph^{4n-1}} \omega_{H_{\eps}} \wedge H_{\eps}^\ast(d\kappa)
=
\int_{[t,t+1] \times \Sph^{4n-1}} \omega_{H_{\eps}}\wedge H_{\eps}^*(\eta_3)\wedge H_{\eps}^*(\upalpha)\\
&+
\int_{[t,t+1] \times \Sph^{4n-1}} d\omega_{H_{\eps}}\wedge H_{\eps}^*(\eta_4)\wedge H_{\eps}^*(\upalpha)
-
\int_{\partial([t,t+1] \times \Sph^{4n-1})} \omega_{H_{\eps}}\wedge H_{\eps}^*(\eta_4)\wedge H_{\eps}^*(\upalpha)\\
&=
I_{1,\eps}(t)+I_{2,\eps}(t)-I_{3,\eps}(t),
\end{split}
\]
and it remains to show that
\begin{equation}
\label{eq167}
\lim_{\eps\to 0}\mvint_{[-\delta,\delta]} I_{1,\eps}(t)\, dt
=
\lim_{\eps\to 0}\mvint_{[-\delta,\delta]} I_{2,\eps}(t)\, dt
=
\lim_{\eps\to 0}\mvint_{[-\delta,\delta]} I_{3,\eps}(t)\, dt
=0.
\end{equation}
Since \eqref{eq151} gives
\begin{equation}
\label{eq168}
\Vert\omega_{H_{\eps}}\Vert_1 \lesssim \Vert\omega_{H_{\eps}}\Vert_2\lesssim
\Vert H_{\eps}^*(\kappa)\Vert_2\lesssim \Vert DH_{\eps}\Vert_\infty^{2n},
\end{equation}
we have
$$
|I_{1,\eps}(t)|\lesssim\Vert\omega_{H_{\eps}}\Vert_1\Vert H_{\eps}^*(\eta_3)\Vert_\infty\Vert H_{\eps}^*(\upalpha)\Vert_\infty
\lesssim
\Vert DH_{\eps}\Vert_\infty^{4n}\Vert H_{\eps}
^*(\upalpha)\Vert_\infty
\lesssim
\eps^{4n(\gamma-1)}\cdot\eps^{2\gamma-1}\to 0.
$$
The estimates do not depend on $t$ so the first integral in \eqref{eq167} converges to zero.

Note that
$$
|I_{2,\eps}(t)|\leq\Vert d\omega_{H_\eps}\Vert_1 \Vert H_\eps^*(\eta_4)\Vert_\infty\Vert H_\eps^*(\upalpha)\Vert_\infty
\lesssim
\Vert d\omega_{H_\eps}\Vert_2\Vert DH_\eps\Vert_\infty^{2n-1}\Vert H_\eps^*(\upalpha)\Vert_\infty.
$$
Since
\[
\Vert d\omega_{H_{\eps}}\Vert_2=
\Vert H^*_{\eps}(\kappa)-\delta\beta_{H_{\eps}}\Vert_2
\lesssim
\Vert DH_{\eps}\Vert_\infty^{2n}+\Vert H_{\eps}^*(\upalpha)\Vert_\infty\big(1+\Vert DH_{\eps}\Vert_\infty^{2n}\big),
\]
we estimate
\[
\begin{split}
|I_{2,\eps}(t)|
&\lesssim
\big(\Vert DH_{\eps}\Vert_\infty^{2n}+\Vert H_{\eps}^*(\upalpha)\Vert_\infty
(1+\Vert DH_{\eps}\Vert_\infty^{2n})\big)\Vert DH_{\eps}\Vert_\infty^{2n-1}\Vert H_{\eps}^*(\upalpha)\Vert_\infty\\
&\lesssim
\big(\eps^{2n(\gamma-1)}+\eps^{2\gamma-1}\cdot\eps^{2n(\gamma-1)}\big)\cdot\eps^{(2n-1)(\gamma-1)}
\cdot\eps^{2\gamma-1}\to 0.
\end{split}
\]
Again, the estimates do not depend on $t$ so the second integral in \eqref{eq167} converges to zero.

While we could estimate $I_{1,\eps}(t)$ and $I_{2,\eps}(t)$ for each $t\in [-\delta,\delta]$, with estimates in terms of $\eps$ that are independent of $t$, this is not possible for $I_{3,\eps}(t)$, because we do not have good estimates for $\omega_{H_\eps}$ on individual slices $\{t\}\times\Sph^{4n-1}$ and $\{t+1\}\times\Sph^{4n-1}$. However, we can easily estimate the integral $\int_{-\delta}^\delta I_{3,\eps}(t)\, dt$, where instead of estimating $\omega_{H_\eps}$ on slices we can use the global estimate \eqref{eq168} on $\R/2\mathbb{Z}\times \Sph^{4n-1}$. We have
\[
\begin{split}
\Bigg|\int_{-\delta}^\delta I_{3,\eps}(t)\, dt\bigg|
&\leq
\int_{\R/2\mathbb{Z}\times\Sph^{4n-1}} |\omega_{H_\eps}|\, |H_\eps^*(\eta_4)|\, |H_\eps^*(\upalpha)|
\leq
\Vert\omega_{H_\eps}\Vert_1\Vert H_\eps^*(\eta_4)\Vert_\infty\Vert H_\eps^*(\upalpha)\Vert_\infty \\
&\lesssim 
\Vert DH_{\eps}\Vert_\infty^{2n}\Vert DH_{\eps}\Vert_\infty^{2n-1}\Vert H_{\eps}(\upalpha)\Vert_\infty
\lesssim
\eps^{2n(\gamma-1)}\cdot\eps^{(2n-1)(\gamma-1)}\cdot\eps^{2\gamma-1}\to 0.
\end{split}
\]
This proves convergence to zero of the last integral in \eqref{eq167} and completes the proof of
\Cref{pr:homotopy}.
\end{proof}

\subsection{Proof of \Cref{th:hopfgamma}}
\label{LastProof}
We start with constructing $\kappa\in\Omega_c^{2n}(\R^{4n+1})$ such that $\HI_\kappa(f)\neq 0$ for some smooth map $f:\Sph^{4n-1}\to\R^{4n+1}$ that is horizontal in the sense that $f^*(\upalpha)=0$; such a mapping necessarily belongs to $f\in C^{0,1}(\Sph^{4n-1};\bbbh^{2n})$.

Let $\phi:\Sph^{2n}\to\R^{4n+1}$ be the smooth embedding from \Cref{T107}; $\phi$ is a bi-Lipschitz embedding of $\Sph^{2n}$ to $\bbbh^{2n}$. Let $\psi\in C_c^\infty(\R^{4n+1};\R^{2n+1})$ be a smooth and compactly supported extension of $\phi^{-1}:\phi(\Sph^{2n})\to\Sph^{2n}\subset\R^{2n+1}$.

Recall that
$$
\nu_{\Sph^{2n}}=\sum_{j=1}^{2n+1} (-1)^{j-1}x_j\, dx_1\wedge\ldots\wedge\widehat{dx_j}\wedge\ldots\wedge dx_{2n+1}\in\Omega^{2n}(\R^{2n+1})
$$
restricted to $\Sph^{2n}$ is the volume form on $\Sph^{2n}$.
Now we define
\begin{equation}
\label{eq150}
\kappa:=
\psi^*(\nu_{\Sph^{2n}})\in\Omega_c^{2n}(\R^{4n+1}).
\end{equation}
\begin{lemma}
\label{T109}
If $\kappa\in\Omega_c^{2n}(\R^{4n+1})$ is defined by \eqref{eq150}, then there is a smooth map $f\in C^\infty(\Sph^{4n-1};\R^{4n+1})$, such that $f\in C^{0,1}(\Sph^{4n-1};\bbbh^{2n})$ and
$\HI_\kappa(f)\neq 0$.
\end{lemma}
\begin{proof}
Let $g\in C^\infty(\Sph^{4n-1};\Sph^{2n})$, $\HI(g)\neq 0$, be as in \Cref{la:hopffibration}, and let $\phi:\Sph^{2n}\to\R^{4n+1}$ be the embedding from \Cref{T107}. Clearly,
$f:=\phi\circ g\in C^\infty(\Sph^{4n-1};\R^{4n+1})$ satisfies $f\in C^{0,1}(\Sph^{4n-1};\bbbh^{2n})$. We claim that
$\HI_\kappa(f)\neq 0$. Indeed, since $\psi\circ\phi=\operatorname{id}_{\Sph^{2n}}$, we have
$$
f^*(\kappa)=g^*\phi^*(\kappa)=g^*\phi^*\psi^*(\nu_{\Sph^{2n}})=g^*(\psi\circ\phi)^*(\nu_{\Sph^{2n}})
=g^*(\nu_{\Sph^{2n}}),
$$
and it easily follows that $\HI_\kappa(f)=\HI(g)\neq 0$.
\end{proof}
If $f$ is as in \Cref{T109}, then \Cref{pr:homotopy} implies that $0\neq [f]\in\pi^\gamma_{4n-1}(\bbbh^{2n})$. This completes the proof of \Cref{th:hopfgamma}.
\hfill $\Box$

\end{document}